\documentclass[12pt,reqno,english]{smfart}

\usepackage{amsmath,amssymb,smfthm,stmaryrd,epic,enumerate,dsfont}
\usepackage[english]{babel}
\usepackage[all]{xy}
\usepackage{csquotes}

\usepackage{mathrsfs}
\usepackage[active]{srcltx}

\newcommand*{\htarrow}{\lhook\joinrel\twoheadrightarrow}

\NumberTheoremsIn{subsection}\SwapTheoremNumbers

\begin{document}

\newcommand{\REMARK}[1]{\marginpar{\tiny #1}}
\newtheorem{thm}{Theorem}[subsection]
\newtheorem{lemma}[thm]{Lemma}
\newtheorem{corol}[thm]{Corollary}
\newtheorem{propo}[thm]{Proposition}
\newtheorem{defin}[thm]{Definition}
\newtheorem{Remark}[thm]{Remark}
\numberwithin{equation}{subsection}

\newtheorem{notas}[thm]{Notations}
\newtheorem{nota}[thm]{Notation}
\newtheorem{defis}[thm]{Definitions}
\newtheorem*{thm*}{Theorem}
\newtheorem*{conj*}{Conjecture}

\def\Tm{{\mathbb T}}
\def\Um{{\mathbb U}}
\def\Am{{\mathbb A}}
\def\Fm{{\mathbb F}}
\def\Mm{{\mathbb M}}
\def\Nm{{\mathbb N}}
\def\Pm{{\mathbb P}}
\def\Qm{{\mathbb Q}}
\def\Zm{{\mathbb Z}}
\def\Dm{{\mathbb D}}
\def\Cm{{\mathbb C}}
\def\Rm{{\mathbb R}}
\def\Gm{{\mathbb G}}
\def\Lm{{\mathbb L}}
\def\Km{{\mathbb K}}
\def\Om{{\mathbb O}}
\def\Em{{\mathbb E}}
\def\Xm{{\mathbb X}}

\def\BC{{\mathcal B}}
\def\QC{{\mathcal Q}}
\def\TC{{\mathcal T}}
\def\ZC{{\mathcal Z}}
\def\AC{{\mathcal A}}
\def\CC{{\mathcal C}}
\def\DC{{\mathcal D}}
\def\EC{{\mathcal E}}
\def\FC{{\mathcal F}}
\def\GC{{\mathcal G}}
\def\HC{{\mathcal H}}
\def\IC{{\mathcal I}}
\def\JC{{\mathcal J}}
\def\KC{{\mathcal K}}
\def\LC{{\mathcal L}}
\def\MC{{\mathcal M}}
\def\NC{{\mathcal N}}
\def\OC{{\mathcal O}}
\def\PC{{\mathcal P}}
\def\UC{{\mathcal U}}
\def\VC{{\mathcal V}}
\def\XC{{\mathcal X}}
\def\SC{{\mathcal S}}
\def\RC{{\mathcal R}}

\def\BF{{\mathfrak B}}
\def\AF{{\mathfrak A}}
\def\GF{{\mathfrak G}}
\def\EF{{\mathfrak E}}
\def\CF{{\mathfrak C}}
\def\DF{{\mathfrak D}}
\def\JF{{\mathfrak J}}
\def\LF{{\mathfrak L}}
\def\MF{{\mathfrak M}}
\def\NF{{\mathfrak N}}
\def\XF{{\mathfrak X}}
\def\UF{{\mathfrak U}}
\def\KF{{\mathfrak K}}
\def\FF{{\mathfrak F}}

\def \longmapright#1{\smash{\mathop{\longrightarrow}\limits^{#1}}}
\def \mapright#1{\smash{\mathop{\rightarrow}\limits^{#1}}}
\def \lexp#1#2{\kern \scriptspace \vphantom{#2}^{#1}\kern-\scriptspace#2}
\def \linf#1#2{\kern \scriptspace \vphantom{#2}_{#1}\kern-\scriptspace#2}
\def \linexp#1#2#3 {\kern \scriptspace{#3}_{#1}^{#2} \kern-\scriptspace #3}

\def \soc {{\mathop{\mathrm{soc}}\nolimits}}
\def \Ext{\mathop{\mathrm{Ext}}\nolimits}
\def \ad{\mathop{\mathrm{ad}}\nolimits}
\def \sh{\mathop{\mathrm{Sh}}\nolimits}
\def \irr{\mathop{\mathrm{Irr}}\nolimits}
\def \FH{\mathop{\mathrm{FH}}\nolimits}
\def \FPH{\mathop{\mathrm{FPH}}\nolimits}
\def \coh{\mathop{\mathrm{Coh}}\nolimits}
\def \res{\mathop{\mathrm{Res}}\nolimits}
\def \op{\mathop{\mathrm{op}}\nolimits}
\def \rec {\mathop{\mathrm{rec}}\nolimits}
\def \art{\mathop{\mathrm{Art}}\nolimits}
\def \vol {\mathop{\mathrm{vol}}\nolimits}
\def \cusp {\mathop{\mathrm{Cusp}}\nolimits}
\def \scusp {\mathop{\mathrm{Scusp}}\nolimits}
\def \Iw {\mathop{\mathrm{Iw}}\nolimits}
\def \JL {\mathop{\mathrm{JL}}\nolimits}
\def \speh {\mathop{\mathrm{Speh}}\nolimits}
\def \isom {\mathop{\mathrm{Isom}}\nolimits}
\def \Vect {\mathop{\mathrm{Vect}}\nolimits}
\def \groth {\mathop{\mathrm{Groth}}\nolimits}
\def \hom {\mathop{\mathrm{Hom}}\nolimits}
\def \deg {\mathop{\mathrm{deg}}\nolimits}
\def \val {\mathop{\mathrm{val}}\nolimits}
\def \det {\mathop{\mathrm{det}}\nolimits}
\def \rep {\mathop{\mathrm{Rep}}\nolimits}
\def \spec {\mathop{\mathrm{Spec}}\nolimits}
\def \fr {\mathop{\mathrm{Fr}}\nolimits}
\def \frob {\mathop{\mathrm{Frob}}\nolimits}
\def \ker {\mathop{\mathrm{Ker}}\nolimits}
\def \im {\mathop{\mathrm{Im}}\nolimits}
\def \Red {\mathop{\mathrm{Red}}\nolimits}
\def \red {\mathop{\mathrm{red}}\nolimits}
\def \aut {\mathop{\mathrm{Aut}}\nolimits}
\def \diag {\mathop{\mathrm{diag}}\nolimits}
\def \spf {\mathop{\mathrm{Spf}}\nolimits}
\def \Def {\mathop{\mathrm{Def}}\nolimits}
\def \twist {\mathop{\mathrm{Twist}}\nolimits}
\def \supp {\mathop{\mathrm{Supp}}\nolimits}
\def \Id {{\mathop{\mathrm{Id}}\nolimits}}
\def \lie {{\mathop{\mathrm{Lie}}\nolimits}}
\def \Ind{\mathop{\mathrm{Ind}}\nolimits}
\def \ind {\mathop{\mathrm{ind}}\nolimits}
\def \bad {\mathop{\mathrm{Bad}}\nolimits}
\def \top {\mathop{\mathrm{Top}}\nolimits}
\def \ker {\mathop{\mathrm{Ker}}\nolimits}
\def \coker {\mathop{\mathrm{Coker}}\nolimits}
\def \gal {{\mathop{\mathrm{Gal}}\nolimits}}
\def \Nr {{\mathop{\mathrm{Nr}}\nolimits}}
\def \rn {{\mathop{\mathrm{rn}}\nolimits}}
\def \tr {{\mathop{\mathrm{Tr~}}\nolimits}}
\def \Sp {{\mathop{\mathrm{Sp}}\nolimits}}
\def \st {{\mathop{\mathrm{St}}\nolimits}}
\def \sp{{\mathop{\mathrm{Sp}}\nolimits}}
\def \perv{\mathop{\mathrm{Perv}}\nolimits}
\def \tor {{\mathop{\mathrm{Tor}}\nolimits}}
\def \gr {{\mathop{\mathrm{gr}}\nolimits}}
\def \nilp {{\mathop{\mathrm{Nilp}}\nolimits}}
\def \obj {{\mathop{\mathrm{Obj}}\nolimits}}
\def \spl {{\mathop{\mathrm{Spl}}\nolimits}}
\def \unr {{\mathop{\mathrm{Unr}}\nolimits}}
\def \alg {{\mathop{\mathrm{Alg}}\nolimits}}
\def \grr {{\mathop{\mathrm{grr}}\nolimits}}
\def \cogr {{\mathop{\mathrm{cogr}}\nolimits}}
\def \coFil {{\mathop{\mathrm{coFil}}\nolimits}}
\def \nrd {{\mathop{\mathrm{nrd}}\nolimits}}

\def \rem{{\noindent\textit{Remark:~}}}
\def \rems{{\noindent\textit{Remarques:~}}}
\def \ext {{\mathop{\mathrm{Ext}}\nolimits}}
\def \End {{\mathop{\mathrm{End}}\nolimits}}

\def\semi{\mathrel{>\!\!\!\triangleleft}}
\let \DS=\displaystyle
\def\HT{{\mathop{\mathcal{HT}}\nolimits}}

\def \hi{\HC}
\newcommand*{\tarrow}{\relbar\joinrel\mid\joinrel\twoheadrightarrow}
\newcommand*{\harrow}{\lhook\joinrel\relbar\joinrel\mid\joinrel\rightarrow}
\newcommand*{\rarrow}{\relbar\joinrel\mid\joinrel\rightarrow}
\def \coim {{\mathop{\mathrm{Coim}}\nolimits}}
\def \can {{\mathop{\mathrm{can}}\nolimits}}
\def\LFF{{\mathscr L}}

\setcounter{secnumdepth}{3} \setcounter{tocdepth}{3}

\def \Fil{\mathop{\mathrm{Fil}}\nolimits}
\def \CoFil{\mathop{\mathrm{CoFil}}\nolimits}
\def \Fill{\mathop{\mathrm{Fill}}\nolimits}
\def \CoFill{\mathop{\mathrm{CoFill}}\nolimits}
\def\SF{{\mathfrak S}}
\def\PF{{\mathfrak P}}
\def \EFil{\mathop{\mathrm{EFil}}\nolimits}
\def \EFill{\mathop{\mathrm{EFill}}\nolimits}
\def \FP{\mathop{\mathrm{FP}}\nolimits}

\let \longto=\longrightarrow
\let \oo=\infty

\let \d=\delta
\let \k=\kappa

\renewcommand{\theequation}{\arabic{section}.\arabic{subsection}.\arabic{thm}}
\newcommand{\marque}{\addtocounter{thm}{1}
{\smallskip \noindent \textit{\thethm}~---~}}

\renewcommand\atop[2]{\ensuremath{\genfrac..{0pt}{1}{#1}{#2}}}

\newcommand\atopp[2]{\genfrac{}{}{0pt}{}{#1}{#2}}

\title[Local Ihara's lemma]{Local Ihara's lemma and applications}


\author{Boyer Pascal}

\thanks{The authors thanks the ANR for his support through the project CoLoss AAPG2019.}


\begin{abstract}
Persistence of non-degeneracy is a phenomenon which appears in the theory of 
$\overline \Qm_l$-representations of the linear group: every irreducible submodule of the 
restriction
to the mirabolic sub-representation of a non-degenerate irreducible representation is 
non-degenerate. 
This is not true anymore in general, 
if we look at the modulo $l$ reduction of some stable lattice.
As in the Clozel-Harris-Taylor generalization of global Ihara's lemma, we show that
this property, called non-degeneracy persistence and related to the notion
of essentially absolutely irreducible and generic representations in the work of Emerton-Helm, remains true for lattices 
given by the cohomology of Lubin-Tate spaces. As an global application, we give 
a new construction of automorphic congruences in the Ribet spirit.
\end{abstract}

\begin{altabstract}
La persistence de la non d\'eg\'en\'erescence est un ph\'enom\`ene qui apparait dans la th\'eorie des
$\overline \Qm_l$-repr\'esentations du groupe lin\'eaire: toute sous-repr\'esentation 
irr\'eductible de la restriction
au groupe mirabolique d'une repr\'esentation irr\'eductible non d\'eg\'en\'er\'ee, est non d\'eg\'en\'er\'ee.
Ce n'est plus le cas en g\'en\'eral pour la r\'eduction modulo $l$ d'un r\'eseau stable.
Comme dans la g\'en\'eralisation par Clozel-Harris-Taylor du lemme d'Ihara, nous montrons que cette
propri\'et\'e de non d\'eg\'en\'erescence, qui est reli\'ee \`a la notion de repr\'esentation
essentiellement absolument g\'en\'erique de Emerton-Helm, reste valide pour les r\'eseaux donn\'es par la cohomologie
des espaces de Lubin-Tate. Nous une application de nature globale en construisant
des congruences automorphes dans l'esprit du travail de Ribet.
\end{altabstract}


%


\maketitle

\pagestyle{headings} \pagenumbering{arabic}

\tableofcontents
%
%

\section*{Introduction}

Before the \enquote{Ihara avoidance} argument of Taylor, the proof of Sato-Tate conjecture by 
Clozel, Harris and Taylor, rested on a conjectural generalization in higher dimension of the 
classical Ihara's lemma for $GL_2$. Their formulation can be understood as some
persistence of the non-degeneracy property by reduction modulo $l$ of automorphic 
representations.

Fix prime numbers $l \neq p$ and a finite extension $K$ of $\Qm_p$. Recall
then \cite{zelevinski2} corollary 6.8, that any irreducible $\overline \Qm_l$-representation 
$\pi$ of $GL_d(K)$ is homogeneous which means, cf. \cite{zelevinski2} definition 5.1,
that its restriction to the mirabolic
subgroup $M_d(K)$ of matrices such that the last row is $(0,\cdots,0,1)$, is homogeneous
in the sense that every irreducible sub-$M_d(K)$-representation has the same level
of degeneracy, cf. \cite{zelevinski2} 4.3 or \cite{zelevinski1} 3.5. In particular if
$\pi$ is non-degenerate i.e. its level of degeneracy equals $d$, then any irreducible
sub-representation of $\pi_{|M_d(K)}$ is also non-degenerate. Modulo $l$, for
$\pi$ an irreducible non-degenerate representation of $GL_d(K)$, there might exist stable
lattices such that
$\pi_{|M_d(K)} \otimes_{\overline \Zm_l} \overline \Fm_l$ owns irreducible
degenerate subspaces, cf. corollary \ref{coro-RI-nd}.

We then propose to prove some persistence of non-degeneracy phenomenons in the 
cohomology groups of Lubin-Tate spaces. Consider
a finite extension $K/\Qm_p$ with ring of integers $\OC_K$. 
For $d\geq 1$, denote by $\widehat \MC_{LT,d,n}$ the 
formal scheme representing the functor of isomorphism classes of deformations
by quasi-isogenies of the formal $\OC_K$-module over $\overline \Fm_p$ of dimension $1$
and height $d$
with $n$-level structure. We denote by $\MC_{LT,d,n}$ its generic fiber over 
$\widehat K^{un}$. For $\Lambda=\overline \Qm_l, \overline \Zm_l$ or $\overline \Fm_l$,
consider both
$$\UC^{d-1}_{LT,d,\Lambda}:=\lim_{\atop{\longrightarrow}{n}} H^{d-1}(\MC_{LT,d,n} 
\widehat \otimes_{\widehat K^{un}} \widehat{\overline K},\Lambda)$$
and
$$\VC^{d-1}_{LT,d,\Lambda}:=\lim_{\atop{\longrightarrow}{n}} H_c^{d-1}(\MC_{LT,d,n} 
\widehat \otimes_{\widehat K^{un}} \widehat{\overline K},\Lambda).$$


There is a natural action of $GL_d(K) \times D_{K,d}^\times \times W_K$ on
$\UC^{d-1}_{LT,d,\overline \Qm_l}$ and $\VC^{d-1}_{LT,d,\Lambda}$, 
where $D_{K,d}$ (resp. $W_K$) is the central division algebra over $K$
with invariant $1/d$ (resp. the Weil group of $K$). In this paper we focus on the action of
$GL_d(K)$ and it appears, cf. \cite{boyer-invent2}, that every irreducible 
$GL_d(K)$-subquotient of $\UC^{d-1}_{LT,d,\overline \Qm_l}$ and
$\VC^{d-1}_{LT,d,\overline \Qm_l}$ is either a cuspidal or a generalized
Steinberg representation, so it is always non-degenerate. 
One can then ask if any irreducible $GL_d(K)$-equivariant subspace of
$\UC^{d-1}_{LT,d,\overline \Fm_l}$ (resp. $\VC^{d-1}_{LT,d,\overline \Fm_l}$)
is still non-degenerate or even more if any irreducible $M_d(K)$-equivariant
subspace is non-degenerate.

\begin{thm*} (cf. corollaries \ref{coro-principal1} and \ref{coro-principal2}) \\
The persistence of non-degeneracy property relatively to $M_d$
holds for $\VC^{d-1}_{LT,d,\overline \Zm_l} \otimes_{\overline \Zm_l} \overline \Fm_l$ and 
$\UC^{d-1}_{LT,d,\overline \Zm_l,free} \otimes_{\overline \Zm_l} \overline \Fm_l$, i.e.
any irreducible $M_d(K)$-equivariant subspace is non-degenerate.
\end{thm*}

\rem $\UC^{d-1}_{LT,d,\overline \Zm_l,free}$ is the free quotient of $\UC^{d-1}_{LT,d,\overline \Zm_l}$.
In \cite{boyer-duke} we prove that $\VC^i_{LT,d,\overline \Zm_l}$ and $\UC^i_{LT,d,\overline \Zm_l}$
are free for every $i$ so that 
$$\VC^{d-1}_{LT,d,\overline \Zm_l} \otimes_{\overline \Zm_l} \overline \Fm_l \simeq 
\VC^{d-1}_{LT,d,\overline \Fm_l} \quad \hbox{ and } \quad 
\UC^{d-1}_{LT,d,\overline \Zm_l,free} \otimes_{\overline \Zm_l} 
\overline \Fm_l \simeq \UC^{d-1}_{LT,d,\overline \Fm_l}.$$ 
Note that we do not use this result to prove the theorem.

\smallskip

The main motivation of this work is to obtain a geometric incarnation of the local Langlands
correspondance in families of Emerton-Helm-Moss and we hope to come to this project soon.

\medskip

The strategy for proving this property of the Lubin-Tate cohomology, is to argue globally on
Shimura varieties  of Harris-Taylor type, $X_I \rightarrow \spec \OC_K$ where $\OC_K$ is the ring
of integers of $K$, cf. \ref{para-KHT}. 
Thanks to Berkovich's comparison theorem in \cite{berk2},
we have to understand the stalk of the $\overline \Zm_l$-vanishing cycle perverse sheaf $\Psi_I$
at some geometric supersingular point
of the geometric special fiber $X_{I,\bar s}$ of $X_I$. 

Using the Newton stratification of $X_{I,\bar s}$
and usual adjunction properties, cf. \cite{boyer-torsion}, 
we can construct various filtrations of $\Psi_I$. 
The main issue about these general constructions
is to understand, with the terminology of \S \ref{para-terminology}, 
the phenomenon of saturation which is a blind process consisting
of choosing artificially the right sub-perverse sheaves so that all the graded pieces
are free. In particular it seems impossible to follow the lattices during this process.
One solution is to use the construction of \cite{boyer-FT} based on a coarse filtration
of stratification as recalled in \S \ref{para-coarse}, which introduce no saturation process during the
construction: see lemmas \ref{lem-icpsi} and \ref{lem-icpsi2}. As explained in \cite{boyer-FT} the main reason that
this coarse filtration is more interesting, is its link with the small mirabolic induction
as defined in (\ref{eq-small}) rather than the full parabolic induction appearing
in \cite{boyer-invent2}. For more details, we advice the reader to look at the introduction of \S \ref{para-lattice}.

For $z$ a geometric supersingular point and $i_z: \{ z \} \hookrightarrow
X_{I,\bar s}$, by considering either $i_z^* \hi^i \Psi_I$ or $i_z^! \hi^i \Psi_I$, where
$\hi^\bullet$ designates the functor of sheaf cohomology, we then obtain
a filtration of $\UC^{d-1}_{LT,d,\overline \Zm_l,free}$ and 
$\VC^{d-1}_{LT,d,\overline \Zm_l,free}$. The graded pieces of these filtrations are then lattices
of the irreducible $\overline \Qm_l[GL_d(K) \times D_{K,d}^\times \times W_K]$-subquotients
of $\UC^{d-1}_{LT,d,\overline \Qm_l}$ (resp.  $\VC^{d-1}_{LT,d,\overline \Qm_l}$), which can be described
as a tensorial product of stable lattices $\Lambda_G \otimes \Lambda_D \otimes \Lambda_W$
of respectively $GL_d(K)$,  $D_{K,d}^\times$ and  $W_K$. 
Using the combinatorics of the non supersingular strata and the classical properties of the induced
representations, cf. proposition \ref{prop-zele}, we are then able to prove that  
$V:=\Lambda_G \otimes_{\overline \Zm_l} \overline \Fm_l$ is 
\emph{an essentially absolutely irreducible
and generic representation in the sense of \cite{emer-helm} definition 3.2.1}, i.e. 
\begin{itemize}
\item the socle $\soc(V)$ of $V$ is absolutely irreducible and generic,

\item the quotient $V/\soc(V)$ contains no generic Jordan-Holder factors,

\item the representation $V$ is the union of its finite length submodules.
\end{itemize}

In \S \ref{para-other}, using results of \cite{boyer-duke}, we also look at 
$\UC_{LT,d,\overline \Fm_l}^{d-1-\delta}$ (resp. $\VC_{LT,d,\overline \Fm_l}^{d-1+\delta}$)
for $\delta>0$. The situation is less pleasant to state but we can find cases where,
cf. proposition \ref{prop-principal} and the remarks before and after it, that irreducible
subspaces must have minimal derivative order, but among the
irreducible quotients of such derivative order, the lattices of Lubin-Tate cohomology groups
select the one with non-degenerate highest derivative.
%

In the last section, we give a global application with
%
%
%
new congruences between tempered and non tempered automorphic
representations with the same level at $l$: their level are the same except at one place 
which can be chosen almost arbitrary.

Finally to give a perspective about this work, we could say, using the 
terminology cf. \S \ref{para-terminology}, that in  \cite{boyer-duke} we solve the question
about positions of the perverse Harris-Taylor sheaves inside the perverse sheaf
of nearby cycles, and here we elucidate that of lattices.

\clearpage

\section{Review on the representation theory for $GL_n(\Qm_p)$}

We fix a finite extension $K/\Qm_p$ with residue field $\Fm_q$. We denote by $|-|$ its absolute
value.

\subsection{Induced representations}

For a representation $\pi$ of $GL_d(K)$ and $n \in \frac{1}{2} \Zm$, set 
$$\pi \{ n \}:= \pi \otimes q^{-n \val \circ \det}.$$

\begin{notas} \label{nota-ind}
For $\pi_1$ and $\pi_2$ representations of respectively $GL_{n_1}(K)$ and
$GL_{n_2}(K)$, we will denote by
$$\pi_1 \times \pi_2:=\ind_{P_{n_1,n_1+n_2}(K)}^{GL_{n_1+n_2}(K)}
\pi_1 \{ \frac{n_2}{2} \} \otimes \pi_2 \{-\frac{n_1}{2} \},$$
the normalized parabolic induced representation where for any sequence
$\underline r=(0< r_1 < r_2 < \cdots < r_k=d)$, we write $P_{\underline r}$ for 
the standard parabolic subgroup of $GL_d$ with Levi
$$GL_{r_1} \times GL_{r_2-r_1} \times \cdots \times GL_{r_k-r_{k-1}}.$$ 
The symbol $\times$ being associative, we define inductively 
$\pi_1 \times \cdots \times \pi_s$ as
$(\pi_1 \times \cdots \times \pi_{s-1}) \times \pi_s=\pi_1 \times (\pi_2 \times \cdots \times \pi_s)$.
\end{notas}

Recall that a representation
$\varrho$ of $GL_d(K)$ is called \emph{cuspidal} (resp. \emph{supercuspidal})
if it is not a subspace (resp. subquotient) of a proper parabolic induced representation.
When the field of coefficients is of characteristic zero then these two notions coincide,
but this is not true anymore for $\overline \Fm_l$.

\begin{defin} \label{defi-rep} (see \cite{zelevinski2} \S 9 and \cite{boyer-compositio} \S 1.4)
Let $g$ be a divisor of $d=sg$ and $\pi$ an irreducible cuspidal 
$\overline \Qm_l$-representation of $GL_g(K)$. 
\begin{itemize}
\item The induced representation
$$\pi\{ \frac{1-s}{2} \} \times \pi \{ \frac{3-s}{2} \} \times \cdots \times \pi \{ \frac{s-1}{2} \}$$ 
holds a unique irreducible quotient (resp. subspace) denoted by $\st_s(\pi)$ (resp.
$\speh_s(\pi)$); it's a generalized Steinberg (resp. Speh) representation.

\item For any integers $t,r \geq 1$,
the induced representation $\st_t(\pi \{ \frac{-r}{2} \}) \times \speh_r(\pi \{ \frac{t}{2} \} )$
(resp. $\st_{t-1}(\pi \{ \frac{-r-1}{2} \}) \times \speh_{r+1}(\pi \{ \frac{t-1}{2} \} )$)
owns a unique irreducible subspace (resp. quotient), denoted by
$LT_\pi(t-1,r)$.
\end{itemize}
\end{defin}

\subsection{Reduction modulo $l$ of a Steinberg representation}

Denote by $e_l(q)$ the order of $q \in \Fm_l^\times$. 

\begin{nota} For $\Lambda=\overline \Qm_l$ or $\overline \Fm_l$, denote by
$\scusp_\Lambda(g)$ the set of equivalence classes of irreducible supercuspidal 
$\Lambda$-representations of $GL_g(K)$.
\end{nota}

\begin{propo} \label{prop-red-modl} (cf. \cite{vigneras-livre} III.5.10)
Let $\pi$ be an irreducible cuspidal representation of $GL_g(K)$ with a stable
$\overline \Zm_l$-lattice\footnote{We say that $\pi$ is entire.}, then its modulo $l$ reduction
is irreducible and cuspidal but not necessarily supercuspidal.
\end{propo}

In the following we will denote  by $r_l$ the functor of modulo $l$ reduction, i.e. for a 
$\overline \Qm_l$-representation $\pi$ of a group $G$, with a stable $\overline \Zm_l$-lattice $\Lambda$,
then $r_l(\pi)$ is $\Lambda \otimes_{\overline \Zm_l} \overline \Fm_l$ with the induced action of $G$.
Note that such $r_l(\pi)$ should depend on the chosen lattice $\Lambda$ but its semi-simplification
doesn't.

\begin{propo} \cite{dat-jl} \S 2.2.3 \\
Let $\pi$ be an irreducible entire cuspidal representation, and $s \geq 1$. Then
the modulo $l$ reduction of $\speh_{s}(\pi)$ is irreducible.
\end{propo}

\begin{nota}
The Zelevinski line associated with some irreducible supercuspidal 
$\overline \Fm_l$-representation $\varrho$, is the set $\{ \varrho\{ i\} ~/~ i \in \Zm \}$. 
It is clearly a finite set and we
denote by $\epsilon(\varrho)$ its cardinal which is a divisor of $e_l(q)$.
We also introduce, cf. \cite{vigneras-induced} p.51
$$m(\varrho)=\left\{ \begin{array}{ll} \epsilon(\varrho), & \hbox{if } \epsilon(\varrho)>1; \\ l, & \hbox{sinon.} \end{array} \right.$$
\end{nota}

\begin{defin}  \label{defi-nd}
Consider a multiset\footnote{meaning we take into account the multiplicities} 
$\underline s=\{ \rho_1^{n_1},\cdots,\rho_r^{n_r} \}$
of irreducible supercuspidal $\overline \Fm_l$-representations.  Following
\cite{vigneras-induced} V.7, we then denote by $\st(\underline s)$ the unique non-degenerate 
irreducible sub-quotient of the
induced representation
$$\rho(\underline s):=\overbrace{\rho_1 \times \cdots \times \rho_1}^{n_1} \times \cdots \times \overbrace{\rho_r \times \cdots \times \rho_r}^{n_r}.$$
\end{defin}

\rem Thanks to \cite{vigneras-induced} V.7, every irreducible non-degenerate $\overline \Fm_l$-representation can be written as $\st(\underline s)$.

\begin{nota}  \label{nota-st}
For $s \geq 1$ and $\varrho$ an irreducible cuspidal $\overline \Fm_l$-representation, we denote by
$\underline s(\varrho)$ for the multi-segment $\{ \varrho, \varrho \{ 1 \}, \cdots , \varrho\{ s-1 \} \}$ 
and, cf. \cite{vigneras-induced} V.4, $\st_s(\varrho):=\st(\underline s(\varrho))$.
\end{nota}

\begin{propo} (cf. \cite{vigneras-induced} V.4) \\
With the previous notation, the $\overline \Fm_l$-representation
$\st_s(\varrho)$ is cuspidal if and only if  $s=1$ or $m(\varrho)l^k$ for some $k \geq 0$.
\end{propo}

\rem by \cite{vigneras-livre} III-3.15 and 5.14, every irreducible cuspidal 
$\overline \Fm_l$-representation can be written $\st_s(\varrho)$ 
for some irreducible supercuspidal representation $\varrho$, and $s=1$ or
$s=m(\varrho)l^k$ with $k \geq 0$. 

\begin{notas} \label{nota-rhoi}
Let $\varrho$ be an irreducible cuspidal $\overline \Fm_l$-representation of $GL_g(K)$.
We then denote  
\begin{itemize}
\item $g_{-1}(\varrho):=g$ and for $i \geq 0$, $g_i(\varrho):=m(\varrho)l^i g$;

\item $\varrho_{-1}=\varrho$ and for all $i \geq 0$, $\varrho_i=\st_{m(\varrho)l^i}(\varrho)$.

\item $\cusp(\varrho,i)$ the set of equivalence classes of irreducible  entire
$\overline \Qm_l$-representations such that modulo $l$ it is isomorphic
to $\varrho_i$,

\item and $\cusp(\varrho)=\bigcup_{i \geq -1} \cusp(\varrho,i)$.
\end{itemize}
\end{notas}

\begin{nota} \label{nota-order}
Let $s \geq 1$ and $\varrho$ an irreducible cuspidal $\overline \Fm_l$-representation of
$GL_g(K)$. We denote by
$\IC_\varrho(s)$ the set of sequences $(m_{-1},m_0,\cdots)$ of non-negative integers such that 
$$s=m_{-1}+m(\varrho)\sum_{k=0}^{+\oo} m_k l^k.$$
We denote by $\lg_\varrho(s)$ the cardinal of $\IC_{\varrho}(s)$.
We then define a relation of order on $\IC_{\varrho}(s)$ by
$$(m_{-1},m_0,\cdots) > (m'_{-1},m'_0,\cdots) \Leftrightarrow \exists k \geq -1 \hbox{ s.t. }
\forall i > k: m_i=m'_i \hbox{ and } m_k> m'_k.$$
\end{nota}

\begin{defin} \label{defi-spm}
For $\underline i=(i_{-1},i_0,\cdots) \in \IC_\varrho(s)$, we define
$$\st_{\underline i}(\varrho):= \st_{i_{-1}}(\varrho_{-1}) \times
\st_{i_{0}}(\varrho_0) \times \cdots \times \st_{i_u}(\varrho_u)$$
where $i_k=0$ for all $k >u$.
\end{defin}

\noindent \textit{Remark}: we will denote by $\underline{s_{\max}}$ the maximal element of
$\IC_\varrho(s)$ so that $\st_{\underline{s_{\max}}}(\varrho)$ is non-degenerate.

\begin{thm}  \label{theo-ss-quotient} (cf. \cite{boyer-repmodl} proposition 3.1.5)
Consider $\pi$ an entire irreducible cuspidal $\overline \Qm_l$-representation of $GL_g(K)$ 
and let $\varrho$ be its modulo $l$ reduction. In the Grothendieck group of 
$\overline \Fm_l$-representations of $GL_{sg}(K)$, we have the following equality:
$$r_l \Bigl ( \st_s(\pi) \Bigr )=\sum_{\underline i \in \IC_\varrho(s)} \st_{\underline i}(\varrho).$$
\end{thm}

\noindent \textit{Remark}: for $s<m(\varrho)$, it is irreducible so, up to isomorphism,
it possesses a unique stable lattice, cf. \cite{bellaiche-ribet} proposition 3.3.2
and the following remark.

\subsection{Restriction to the mirabolic group}

In this paragraph, we want to state some of the main results of \cite{zelevinski1} \S 4 about
$\overline \Qm_l$-representations\footnote{In loc. cit. the author consider complex representations,
but for admissible ones, so in particular for irreducible smooth representations, they are defined
over a finite extension of $\Qm$ so that the facility consisting to fix an isomorphism
$\overline \Qm_l \simeq \Cm$ is harmless.}: for $\overline \Fm_l$-representations the 
usual reference is \cite{vigneras-livre} \S III.

Recall first some notations of \cite{zelevinski1} \S 3, see also \cite{vigneras-livre} \S III-1 or
\cite{emer-helm} \S 3. The mirabolic subgroup
$M_d(K)$ of $GL_d(K)$ is the set of matrices with last row $(0,\cdots,0,1)$: we denote
$$V_d(K)=\{ (m_{i,j}) \in M_d(K):~m_{i,j}= \delta_{i,j} \hbox{ for } j < d \}.$$ 
its unipotent radical. We fix a non trivial character $\psi$ of $K$ and let $\theta$ the
character of $V_d(K)$ defined by $\theta( (m_{i,j}))=\psi(m_{d-1,d})$.
For $G=GL_r(K)$ or $M_r(K)$, we denote $\alg(G)$ the abelian category of smooth
representations of $G$ and, following \cite{zelevinski1}, we introduce
$$\Psi^-: \alg(M_d(K)) \longrightarrow \alg(GL_{d-1}(K)),$$
and
$$\Phi^-: \alg (M_d(K)) \longrightarrow \alg (M_{d-1}(K)),$$
defined by $\Psi^-=r_{V_d,1}$ (resp. $\Phi^-=r_{V_d,\theta}$) the functor of $V_{d}$
coinvariants (resp. $(V_{d},\theta)$-coinvariants), cf. \cite{zelevinski1} 1.8.
We also introduce the un-normalized compact induced functor
$$\Psi^+:=i_{V,1}: \alg(GL_{d-1}(K)) \longrightarrow \alg (M_d(K)),$$ 
$$\Phi^+:=i_{V,\theta}: \alg(M_{d-1}(K)) \longrightarrow \alg(M_d(K)).$$

\begin{propo} (\cite{zelevinski1} p451, \cite{emer-helm} proposition 3.1.3 or 
\cite{vigneras-livre} \S III-1)
\begin{itemize}
\item The functors $\Psi^-$, $\Psi^+$, $\Phi^-$ and $\Phi^+$ are exact.

\item $\Phi^- \circ \Psi^+=\Psi^- \circ \Phi^+=0$.

\item $\Psi^-$ (resp. $\Phi^+$) is left adjoint to $\Psi^+$ (resp. $\Phi^-$) and the
following adjunction maps 
$$\Id \longrightarrow \Phi^- \Phi^+, \qquad \Psi^+ \Psi^- \longrightarrow \Id,$$
are isomorphisms and the following sequence is exact
$$0 \rightarrow \Phi^+ \Phi^- \longrightarrow \Id \longrightarrow \Psi^+ \Psi^- 
\rightarrow 0.$$
\end{itemize}
\end{propo}

\begin{defin}
For $\tau \in \alg(M_d(K))$, the representation 
$$\tau^{(k)}:=\Psi^- \circ (\Phi^-)^{k-1}(\tau)$$
is called the $k$-th derivative of $\tau$. If $\tau^{(k)}\neq 0$ and $\tau^{(m)}=0$
for all $m > k$, then $\tau^{(k)}$ is called the highest derivative of $\tau$.
\end{defin}

\begin{nota} (cf. \cite{zelevinski2} 4.3)
Let $\pi \in \alg(GL_d(K))$ (or $\pi \in \alg(M_d(K))$). The maximal number $k$ such that 
$(\pi_{|M_d(K)})^{(k)} \neq (0)$ is called the level of non-degeneracy of $\pi$ and 
denoted by $\lambda(\pi)$.
\end{nota}

\noindent \textit{Remark}: cf \cite{zelevinski1} 3.5, there exists a natural filtration
$0 \subset \tau_d \subset \cdots \subset \tau_1=\tau$ with
$$\tau_k=(\Phi^+)^{k-1} \circ (\Phi^-)^{k-1} (\tau) \hbox{ and }
\tau_k/\tau_{k+1}=(\Phi^+)^{k-1} \circ \Psi^+(\tau^{(k)}).$$
In particular for $\tau$ irreducible there is exactly one $k$ such that $\tau^{(k)} \neq (0)$
and then $\tau \simeq (\Phi^+)^{k-1} \circ \Psi^+(\tau^{(k)})$.

\begin{nota} 
In the particular case where
$k=d$, there is a unique irreducible representation $\tau_{nd}$ of $M_d(K)$
with derivative of order $d$.
\end{nota}

 \rem Note then by \cite{zelevinski1} 4.4, for every irreducible supercuspidal representation
$\pi$ of $GL_d(K)$, we have 
$$\pi_{|M_d(K)} \simeq \tau_{nd}.$$
We can moreover understand theorem \ref{theo-ss-quotient}
as giving a partition of $\st_t(\pi)_{|M_d(K)}$ that associates to each part an
irreducible constituent of $r_l(\st_t(\pi))$.

Consider first the following embedding $GL_r(K) \times M_s(K) \hookrightarrow
M_{r+s}(K)$ sending 
$$(A, M) \mapsto \left ( \begin{array}{cc} A & 0 \\ 0 & M \end{array} \right ).$$
Imposing $\left ( \begin{array}{cc} I_r & U \\ 0 & I_s \end{array} \right )$ acting trivially,
and considering the normalized induced functor, we then define  
\addtocounter{thm}{1}
\begin{equation} \label{eq-small}
\rho \otimes \tau \in \alg(GL_r(K)) \times \alg(M_s(K)) \mapsto \rho \times \tau \in 
\alg(M_{r+s}(K)).
\end{equation}
Secondly we consider $M_r(K) \times GL_s(K) \hookrightarrow M_{r+s}(K)$
sending 
$$\Bigl ( \left ( \begin{array}{cc} A & V \\ 0 & 1 \end{array} \right ) , B \Bigr ) \mapsto
\left ( \begin{array}{ccc} A & 0 & V \\ 0 & B & 0 \\ 0 & 0 & 1 \end{array} \right ),$$
imposing $\left ( \begin{array}{ccc} I_{r-1} & U & 0 \\ 0 & I_s & 0 \\ 0 & 0 & 1 \end{array} \right )$
acting trivially and considering the normalized compact induction functor, we define 
\addtocounter{thm}{1}
\begin{equation} \label{eq-big}
\tau \otimes \rho \in \alg(M_r(K)) \times \alg(GL_s(K)) \mapsto \tau \times \rho \nu^{-1/2}
\in \alg(M_{r+s}(K)),
\end{equation}
where  for $g \in GL_s(K)$, we denote by $\nu(g):=q^{\val (\det g)}$.

\begin{propo} (cf. \cite{zelevinski1} 4.13) \label{prop-zele}
Let $\rho \in \alg (GL_r(K))$, $\sigma \in \alg(GL_t(K))$ and $\tau \in \alg(M_s(K))$.
\begin{itemize}
\item[(a)] In $\alg(M_{r+t}(K))$, we have
$$0 \rightarrow (\rho_{|M_r(K)}) \times \sigma \longrightarrow 
(\rho \times \sigma)_{|M_{r+t}(K)} \longrightarrow \rho \times (\sigma_{|M_t(K)}) 
\rightarrow 0.$$

\item[(b)] If $\Omega$ is one of the functors $\Psi^\pm, \Phi^\pm$, then
$\rho \times \Omega(\tau) \simeq \Omega(\rho \times \tau)$.

\item[(c)] $\Psi^-(\tau \times \rho) \simeq \Psi^-(\tau) \times \rho$ and
$$0 \rightarrow \Phi^-(\tau) \times \rho \longrightarrow \Phi^-(\tau \times \rho)
\longrightarrow \Psi^-(\tau) \times (\rho_{|M_r(K)}) \rightarrow 0.$$

\item[(d)] Suppose $r>0$. Then for any non-zero $M_{r+s}(K)$-submodule 
$\omega \subset \tau \times \rho$, we have $\Phi^-(\omega) \neq (0)$.
\end{itemize}
\end{propo}

We will call the the induced representation $\rho \times \tau$ (resp. $\tau \times \rho$)
the small (resp. the big) mirabolic induced representation in the sense that the big one
owns the highest derivative as you can see it in the proposition above or in lemma \ref{lem-mirabolique}.

\begin{defin} (\cite{zelevinski2} 5.1)
A representation $\tau \in \alg(M_d(K))$ is called homogeneous if
for all non-zero submodules $\sigma \subset \tau$, we have 
$\lambda(\sigma)=\lambda(\tau)$.
\end{defin}

\begin{propo} (cf. \cite{zelevinski2} 6.8)
Let $\pi$ be an irreducible representation of $GL_d(K)$. Then $\pi_{|M_d(K)}$
is homogeneous.
\end{propo}

In the sequel we will use, in some sense dually, 
the group $P_{d}(K)$ with first column equals to
$\lexp t (1,0,\cdots,0)$. The map $g \mapsto \sigma (\lexp t g^{-1} ) \sigma^{-1}$
where $\sigma$ is the matrix permutation associated with the cycle $(1~2~\cdots~n)$,
induces an isomorphism between $P_{d}(K)$ and $M_d(K)$. After twisting with this
isomorphism, we obtain analogs of the previous results with for example the following
short exact sequence
\addtocounter{thm}{1}
\begin{equation} \label{eq-sec-indP}
0 \rightarrow \rho \times (\sigma_{|P_t(K)}) \longrightarrow 
(\rho \times \sigma)_{|P_{r+t}(K)} \longrightarrow  (\rho_{|P_r(K)}) \times \sigma
\rightarrow 0,
\end{equation}
where the first representation is the compact induction relatively to
$$\left ( \begin{array}{ccc} 1 & 0 & V_{t-1} \\ 0 & GL_{r} & U \\ 0 & 0 & GL_{t-1} 
\end{array} \right ),$$
and the second one is the induction from
$$\left ( \begin{array}{cc} P_{r} & U \\ 0 & GL_{t} \end{array} \right ).$$
We will particularly use the following case.

\begin{lemma} \label{lem-mirabolique} (cf. \cite{boyer-FT} lemme 4.4)
Let $\pi$ be an irreducible cuspidal representation of $GL_g(K)$. Then as a representation
of $P_{(t+s)g}(K)$, we have isomorphisms
$$\st_t(\pi\{ -\frac{s}{2} \} )_{|P_{tg}(K)} \times \speh_s (\pi \{ \frac{t}{2} \} ) \simeq
LT_{\pi} (t-1,s)_{|P_{(t+s)g}(K)},$$
and
$$\st_t(\pi\{ -\frac{s}{2} \} ) \times \speh_s (\pi \{ \frac{t}{2} \} )_{|P_{sg}(K)} \simeq 
LT_{\pi} (t,s-1)_{|P_{(t+s)g}(K)}.$$
\end{lemma}

\begin{nota} \label{nota-Mc}
For $c \in K^d$, we will denote by $M_c(K)$
the mirabolic subgroup stabilizing $c$.
\end{nota}

\rem with this notation $P_d(K)$ is $M_c(K)$ for $c=(1,0,\cdots,0)$.

\subsection{Some lattices of Steinberg representations}
\label{para-RI}

Let $\pi$ be an irreducible cuspidal $\overline \Qm_l$-representation of $GL_g(K)$, supposed to be entire.
As its reduction modulo $l$, denoted by $\varrho$, is still irreducible, up to isomorphism, it has a unique stable
lattice, cf. \cite{bellaiche-ribet} proposition 3.3.2 and its following remark.

\begin{defin} \label{defi-RI} (cf. \cite{boyer-repmodl})
Given a stable lattice of $\st_t(\pi)$, the surjection (resp. the embedding) 
$$\st_t(\pi) \times \pi \{ t \} \twoheadrightarrow \st_{t+1}(\pi), \quad resp.~ \st_{t+1}(\pi) \hookrightarrow
\st_t(\pi\{ 1 \}) \times \pi$$
gives a stable lattice of $\st_{t+1}(\pi)$ so that inductively starting from $t=1$, we construct a lattice denoted by
$RI_{\bar \Zm_l,-}(\pi,t)$ (resp. $RI_{\bar \Zm_l,+}(\pi,t)$). We then denote by
$$RI_{\bar \Fm_l,-}(\pi,t):= RI_{\bar \Zm_l,-}(\pi,t) \otimes_{\bar \Zm_l} \bar \Fm_l, \quad
resp.~ RI_{\bar \Fm_l,+}(\pi,t):= RI_{\bar \Zm_l,+}(\pi,t) \otimes_{\bar \Zm_l} \bar \Fm_l.$$
\end{defin}

\begin{propo} \label{prop-defi-Vk} (cf. \cite{boyer-repmodl} propositions
3.2.2 and 3.2.7)
For every $0 \leq k \leq \lg_\varrho(s)$, there exists a unique length $k$ sub-representation $V_{\varrho,\pm}(s;k)$ 
of $RI_{\bar \Fm_l,\pm}(\pi,s)$
$$(0)=V_{\varrho,\pm}(s;0) \varsubsetneq V_{\varrho,\pm}(s;1) \varsubsetneq \cdots \varsubsetneq
V_{\varrho,\pm}(s;\lg_\varrho(s))= RI_{\bar \Fm_l,\pm}(\pi,s),$$
such that the image of $V_{\varrho,-}(s;k)$ (resp. $V_{\varrho,+}(s;k)$) in the Grothendieck group verifies the
following property: all its irreducible constituents are strictly greater (resp. smaller) than any irreducible constituent of
$$W_{\varrho,-}(s;k):= V_{\varrho,-}(s;\lg_\varrho(s))/V_{\varrho,-}(s;k)$$ 
(resp. $W_{\varrho,+}(s;k):=
V_{\varrho,+}(s;\lg_\varrho(s))/V_{\varrho,+}(s;k)$), relatively to the relation of order of \ref{nota-order}.
\end{propo}

%
%

\begin{corol} \label{coro-RI-nd}
If $\lg_\varrho(s) \geq 2$, then we have two irreducible subspaces of
$RI_{\overline \Zm_l,+}(\pi,s)_{|P_{sg}(K)} \otimes_{\overline \Zm_l} \overline \Fm_l$ 
which are 
\begin{itemize}
\item first some irreducible $P_{sg}(K)$-subspaces of $\st_{\underline s}(\varrho)$ which
is necessarily degenerate,

\item and the non-degenerate irreducible $P_{sg}(K)$-representation, $\tau_{nd}$ 
which is a subspace of $\st_{\underline{s_{\max}}}(\varrho)_{|P_{sg}(K)}$.
\end{itemize}
\end{corol}

\begin{propo} \label{prop-RI-M}
The only irreducible subspace of the modulo $l$ reduction of
 $RI_{\overline \Zm_l,-}(\pi,s)_{|P_{sg}(K)}$ 
is the non-degenerate one $\tau_{nd}$.
\end{propo}

\begin{proof}
From the previous section, we have
$$RI_{\overline \Zm_l,-}(\pi,s)_{|P_{sg}(K)} \simeq RI_{\overline \Zm_l,-}
(\pi\{ \frac{-1}{2} \},s-1) \times (\pi\{ \frac{s-2}{2} \})_{|P_{sg}(K)},$$
so that the result follows by induction using proposition \ref{prop-zele}.
\end{proof}

\section{Review on the geometric objects}

\subsection{Lubin-Tate spaces}

Let $\OC_K$ be the ring of integers of $K$, $\PC_K$ its maximal ideal, $\varpi_K$ a
uniformizer and $\kappa=\OC_K/\PC_K$ the residue field of cardinal $q=p^f$. 
Let $K^{nr}$ be the maximal unramified extension of $K$ and $\hat K^{nr}$ its
completion with ring of integers $\OC_{\hat K^{nr}}$. Let $\Sigma_{K,d}$ be 
the one-dimensional $\OC_K$-formal module of Barsotti-Tate over $\overline \Fm_p$ 
with height $d$,
cf. \cite{h-t} \S II. We consider the category $C$ of artinian local 
$\OC_{\hat K^{nr}}$-algebras
with residue field $\overline \k$.

\begin{defin} The functor $\MC_{LT,d,n}$ which associates to an object $R$ of $C$, 
the set of isomorphism classes of deformations by quasi-isogenies over $R$ of 
$\Sigma_{K,d}$,
equipped with a $n$-level structure, is a disjoint union of sub-functors
$\MC_{LT,d,n}^{(h)}$ of deformations by a quasi-isogeny of height $h$ which is
representable by a formal scheme
$\widehat{\MC}_{LT,d,n}^{(h)}$ where
$\widehat{\MC}_{LT,d,n}^{(h)}$.
\end{defin}

\noindent \textit{Remark}: each of the $\widehat{\MC}_{LT,d,n}^{(h)}$ is non canonically
isomorphic to the formal scheme $\widehat{\MC}_{LT,d,n}^{(0)}$ denoted by
$\spf \Def_{d,n}$ in \cite{boyer-invent2}. We will use the notations without hat for
the Berkovich generic fibers which are $\widehat{K^{nr}}$-analytic spaces in the sense
of \cite{berk0} and we note
$\MC_{LT,n}^{d/K}:=\MC_{LT,d,n} \hat \otimes_{\hat K^{nr}} \hat{\overline K}.$

The group of quasi-isogenies of $\Sigma_{K,d}$ is isomorphic to the unit group 
$D_{K,d}^\times$ of the central division algebra over $K$ with invariant $1/d$, which then
acts on $\MC_{LT,n}^{d/K}$. For all $n \geq 1$, we have a natural action of 
$GL_d(\OC_K/\PC_K^n)$ on the level structures and then on $\MC_{LT,n}^{d/K}$.
This action can be extended to $GL_d(K)$ on the projective limit 
${\displaystyle \lim_{\leftarrow ~n}} \MC_{LT,n}^{d/K}$ which is then equipped with 
the action of $GL_d(K) \times D_{K,d}^\times$ which factorises by
$\Bigl ( GL_d(K) \times D_{K,d}^\times \Bigr ) / K^\times$ where $K^\times$ 
is embedded diagonally. 

\begin{defin} 
Let $\Psi_{K,\Lambda,d,n}^{i}\simeq H^i(\MC_{LT,d,n}^{(0)} 
\hat \otimes_{\hat K^{nr}} \hat{\overline K}, \Lambda)$, be the $\Lambda$-module of finite type
associated, by the vanishing cycle theory of Berkovich, to the structural morphism
$\widehat{\MC_{LT,d,n}^{(0)}} \longto \spf \hat \OC_K^{nr}.$ 
\end{defin}

We also introduce $\UC_{K,\Lambda,d,n}^i := H^i(\MC_{LT,n}^{d/K},\Lambda)$ and
$\UC_{K,\Lambda,d}^{i}= {\DS \lim_{\atop{\longto}{n}}} ~ \UC_{K,\Lambda,d,n}^{i}$
as well as the cohomology groups with compact supports
$$\VC_{K,\Lambda,d,n}^i := H_c^i(\MC_{LT,n}^{d/K},\Lambda), \hbox{ and }
\VC_{K,\Lambda,d}^{i}= \lim_{\atop{\longto}{n}} ~ \VC_{K,\Lambda,d,n}^{i}.$$
As $\KF_{n}:=\ker (GL_d(\OC_K) \longto GL_d(\OC_K/\PC_K^n))$ is pro-$p$ 
for all $n \geq 1$, then we have
$\UC_{K,\Lambda,d,n}^{i}=(\UC_{K,\Lambda,d}^{i})^{\KF_{n}}$ and
$\VC_{K,\Lambda,d,n}^{i}=(\VC_{K,\Lambda,d}^{i})^{\KF_{n}}$.

The description of the $\UC_{K,\overline \Qm_l,d}^i$ is given in \cite{boyer-invent2} 
theorem 2.3.5. We will denote by $\UC_{K,\overline \Zm_l,d,free}^i$ (resp. 
$\VC_{K,\overline \Zm_l,d,free}^i$) the free quotient
which is the whole of $\UC_{K,\overline \Zm_l,d}^i$ (resp. $\VC_{K,\overline \Zm_l,d}^i$)
by the main result of \cite{boyer-duke}.

\subsection{KHT-Shimura varieties}
\label{para-KHT}

Let $F=F^+E$ be a CM field with $E/\Qm$ quadratic imaginary.
For $B/F$ a central division algebra with dimension $d^2$ equipped with an 
involution of
second kind $*$ and $\beta \in B^{*=-1}$, consider the similitude group $G/\Qm$ 
defined for any $\Qm$-algebra $R$ by
$$G(R):= \{ (\lambda,g) \in R^\times \times (B^{op} \otimes_\Qm R)^\times \hbox{ such that }
gg^{\sharp_\beta}=\lambda \}$$
with $B^{op}=B \otimes_{F,c} F$ where $c=*_{|F}$ is the complex conjugation and
$\sharp_\beta$ the involution
$x \mapsto x^{\sharp_\beta}=\beta x^* \beta^{-1}$. 
For $p=uu^c$ decomposed in $E$, we have
$$G(\Qm_p) \simeq \Qm_p^\times \times \prod_{w | u} (B^{op}_w)^\times$$
where $w$ describes the places of $F$ above $u$. We suppose as in \cite{h-t} that
\begin{itemize}
\item the associated unitary group $G_0(\Rm)$ has signatures $(1,d-1) \times (0,d) \times \cdots \times (0,d)$;

\item for any place $x$ of $\Qm$ inert or ramified in $E$, then $G(\Qm_x)$ 
is quasi-split.

\item We moreover suppose that $u$ is chosen so that there exists a fixed place
$v | u$ with $B_v \simeq M_d(F_v)$.
\end{itemize}

\begin{notas} \label{nota-spl}
\begin{itemize}
\item Denote by $\Am$ the adele ring of $\Qm$. For a finite set
$S$ of places of $\Qm$, we then introduce $\Am^S$ the adele ring of $\Qm$ outside $S$.

\item The set of places $p$ of $\Qm$ decomposed in $E$ is denoted $\spl$ and we also
introduce $\spl^S$ the subset of places of $\spl$ which does not belong to $S$.
\end{itemize}
\end{notas}

For all open compact subgroups $U^p$ of $G(\Am^{\oo,p})$ and 
$m=(m_w)_{w|u}$ a collection of non-negative integers, we consider
$$U^p(m)=U^p \times \Zm_p^\times \times \prod_{w | u}
\ker ( \OC_{B_{w}}^\times 
\longrightarrow (\OC_{B_{w}}/\PC_{w}^{m_w})^\times ),$$
where $\OC_{B_{w}}$ is the maximal order of $B_{w}$.

We then denote by $\IC$ the set of these $U^p(m)$ such that it exists a place
$x$ for which the projection from $U^p$ to $G(\Qm_x)$ doesn't contain any element
with finite order except the identity, cf. \cite{h-t} bellow of page 90. 

Attached to each $I \in \IC$ is a Shimura variety $X_I \rightarrow \spec \OC_v$ 
of type Kottwitz-Harris-Taylor. The projective system $X_\IC=(X_I)_{I \in \IC}$
is then equipped with a Hecke action of $G(\Am^\oo)$, the transition morphisms
$r_{J,I}:X_J \rightarrow X_I$ for $J \subset I$ being finite flat and even etale when $m_v(J)=m_v(I)$.

\begin{notas} (cf. \cite{boyer-invent2} \S 1.3) \label{nota-strate}
Let $I \in \IC$, 
\begin{itemize}
\item the special (resp. generic) fiber of $X_I$ at $v$ will be denoted by $X_{I,s}$ (resp. $X_{I,\eta}$)
and its geometric special (resp. generic) fiber $X_{I,\bar s}:=X_{I,s} \times \spec \overline \Fm_p$
(resp. $X_{I,\bar \eta}$).

\item For $1 \leq h \leq d$, let $X_{I,\bar s}^{\geq h}$ (resp. $X_{I,\bar s}^{=h}$)
be the closed (resp. open) Newton stratum of height $h$, defined as the subscheme
where the connected component of the universal Barsotti-Tate group is of height 
greater or equal to $h$ (resp. equal to $h$).
\end{itemize}
\end{notas}

\noindent \textit{Remark}: $X_{\IC,\bar s}^{\geq h}$ is of pure dimension $d-h$.
For $1 \leq h < d$, the Newton stratum $X_{\IC,\bar s}^{=h}$ is geometrically induced
under the action of the parabolic subgroup $P_{h,d}(\OC_v)$ in the sense where there
exists a closed subscheme $X_{I,\bar s,\overline{1_h}}^{=h}$ stabilized by the Hecke 
action of $P_{h,d}(\OC_v)$ and such that
$$X_{\IC,\bar s}^{=h} \simeq X_{\IC,\bar s,\overline{1_h}}^{=h} 
\times_{P_{h,d}(\OC_v)} GL_d(\OC_v).$$
Let $\GC(h)$ denote the universal Barsotti-Tate group over $X_{I,\bar s,\overline{1_h}}^{=h}$:
$$0 \rightarrow \GC(h)^c \longrightarrow \GC(h) \longrightarrow \GC(h)^{et} \rightarrow 0$$
where $\GC(h)^c$ (resp. $\GC(h)^{et}$) is connected (resp. \'etale) of height $h$ 
(resp. $d-h$). Denote by
$\iota_{m_v}:(\PC_v^{-m_v}/\OC_v)^d \longrightarrow \GC(h)[\PC_v^{m_v}]$
the universal level structure. If we denote by $(e_i)_{1 \leq i \leq d}$ the canonical basis
of $(\PC_v^{-m_v}/\OC_v)^d$, then the Newton stratum 
$X_{I,\bar s,\overline{1_h}}^{=h}$ is defined by asking
$\bigl \{ \iota_{m_v}(e_i):~1 \leq i \leq h \bigr \}$ to be a Drinfeld basis of
$\GC(h)^c[\PC_v^{m_v}]$.

\begin{nota} \label{nota-a}
In the following, we won't make any distinction between an element
$a \in GL_d(F_v)/P_{h,d}(F_v)$ and the subspace $\langle a(e_1),\cdots,a(e_h) \rangle$
generated by the image through $a$ of the first $h$ vectors $e_1,\cdots, e_h$ of the
canonical basis of $F_v^d$. Denote by $P_a(F_v):=aP_{h,d}(F_v)a^{-1}$ the parabolic
subgroup of elements of $GL_d(F_v)$ stabilizing $a \subset F_v^d$.
\end{nota}

For $I \in \IC$, the element $a \in GL_d(F_v)/P_{h,d}(F_v)$ gives
a direct factor $a_{m_v}$ of $(\PC_v^{-m_v}/\OC_v)^d$ and so a stratum 
$X_{I,\bar s,a}^{=h}$ which is defined by asking for a basis 
$(f_1,\cdots,f_h)$ of  $a_{m_v}$, that
$\bigl \{ \iota_{m_v}(f_i):~1 \leq i \leq h \}$ is a Drinfeld basis of $\GC(h)^c[\PC_v^{m_v}]$. 
We also denote by $X_{I,\bar s,a}^{\geq h}$ its closure in $X_{I,\bar s}^{\geq h}$.
Such a stratum is said pure compared to the following situation.
For a pure stratum $X^{=h}_{\IC,\bar s,c}$ and $h'\geq h$, denote by
$$X^{=h'}_{\IC,\bar s,c}:=\coprod_{\atop{a:~\dim a=h'}{c \subset a}} X^{=h'}_{\IC,\bar s,a}$$
and $X^{\geq h'}_{\IC,\bar s,c}$ its closure.

\subsection{Harris-Taylor perverse sheaves}

We recall now some notations about Harris-Taylor local systems of \cite{h-t}. Let
$\pi_v$ be an irreducible cuspidal $\overline \Qm_l$-representation of $GL_g(F_v)$.
Fix $t \geq 1$ such that $tg \leq d$. The Jacquet-Langlands correspondence associates
to $\st_t(\pi_v)$, an irreducible representation $\pi_v[t]_D$ of $D_{v,tg}^\times$.
For $\DC_{v,tg}$ the maximal order of $D_{v,tg}$, we decompose
$(\pi_v[t]_D)_{|\DC_{v,tg}^\times}=\bigoplus_{i=1}^{e_{\pi_v}} \rho_{v,i}$ 
with the $\rho_{v,i}$ irreductible. Thanks to Igusa varieties, we then have a local system
$\LC_{\overline \Qm_l}(\rho_{v,i})_{\overline{1_{tg}}}$ on 
$X^{=tg}_{\IC,\bar s,\overline{1_{tg}}}$ associated with each $\rho_{v,i}$ and we write

\addtocounter{thm}{1}
\begin{equation} \label{eq-epiv}
\LC(\pi_v[t]_D)_{\overline{1_{tg}}}:=
\LC_{\overline \Qm_l}(\rho_{v})_{\overline{1_{tg}}}
\end{equation}
where $\rho_v$ is any of the $\rho_{v,i}$, cf. \cite{boyer-invent2} \S 2.4.4.
Note that the Hecke action of $P_{tg,d}(F_v)$ on (\ref{eq-epiv}) is given
through its quotient $GL_{d-tg}(F_v) \times \Zm$. 

\begin{nota} Let $e_{\pi_v}$ denote the order of the set of $\pi'_v$ inertially equivalent
to $\pi_v$ in the sense there exists a character $\xi:\Zm \longrightarrow \overline \Qm_l^\times$
such that $\pi' \simeq \pi \otimes (\xi \circ \val \circ \det)$ where $\val:F_v \longrightarrow \Zm$
is the valuation of $F_v$; cf. definition 1.1.3 of \cite{boyer-invent2}. 
\end{nota}
%

\begin{notas} \label{nota-rem}
For $\Pi_t$ any representation of $GL_{tg}(F_v)$ and 
$\Xi:\frac{1}{2} \Zm \longrightarrow \overline \Zm_l^\times$ defined by 
$\Xi(\frac{1}{2})=q^{1/2}$, we introduce
$$\widetilde{HT}_{\overline{1_{tg}}}(\pi_v,\Pi_t):=\LC(\pi_v[t]_D)_{\overline{1_{tg}}} 
\otimes \Pi_t \otimes \Xi^{\frac{tg-d}{2}}$$
living over $X_{\IC,\bar s,\overline{1_{tg}}}^{=tg}$ and its induced version living
over $X_{\IC,\bar s}^{=tg}$
$$\widetilde{HT}(\pi_v,\Pi_t):=\Bigl ( \LC(\pi_v[t]_D)_{\overline{1_{tg}}} 
\otimes \Pi_t \otimes \Xi^{\frac{tg-d}{2}} \Bigr) \times_{P_{tg,d}(F_v)} GL_d(F_v),$$
where the unipotent radical of $P_{tg,d}(F_v)$ acts trivially and the action of
$$(g^{\oo,v},\left ( \begin{array}{cc} g_v^c & * \\ 0 & g_v^{et} \end{array} \right ),\sigma_v) 
\in G(\Am^{\oo,v}) \times P_{tg,d}(F_v) \times W_{F_v}$$ 
is given
\begin{itemize}
\item by the action of $g_v^c$ on $\Pi_t$ and 
$\deg(\sigma_v) \in \Zm$ on $ \Xi^{\frac{tg-d}{2}}$, and

\item the action of $(g^{\oo,v},g_v^{et},\val(\det g_v^c)-\deg \sigma_v)
\in G(\Am^{\oo,v}) \times GL_{d-tg}(F_v) \times \Zm$ on $\LC_{\overline \Qm_l}
(\pi_v[t]_D)_{\overline{1_{tg}}} \otimes \Xi^{\frac{tg-d}{2}}$.
\end{itemize}
We also introduce
$$HT_{\overline{1_{tg}}}(\pi_v,\Pi_t):=\widetilde{HT}_{\overline{1_{tg}}}(\pi_v,\Pi_t)[d-tg],$$
and the perverse sheaf
$$P(t,\pi_v)_{\overline{1_{tg}}}:=j^{=tg}_{\overline{1_{tg}},!*} 
HT_{\overline{1_{tg}}}(\pi_v,\st_t(\pi_v)) \otimes \Lm(\pi_v),$$
and their induced version, $HT(\pi_v,\Pi_t)$ and $P(t,\pi_v)$,
where 
\begin{itemize}
\item for any $1 \leq h \leq d$ we denote by
$$j^{=h}:=i^h \circ j^{\geq h}:X^{=h}_{\IC,\bar s} \hookrightarrow
X^{\geq h}_{\IC,\bar s} \hookrightarrow X_{\IC,\bar s},$$
and 
$$j^{=h}_{\overline{1_h}}:=i^h_{\overline{1_h}} \circ j^{\geq h}_{\overline 1_h}:
X^{=h}_{\IC,\bar s,\overline{1_h}} \hookrightarrow
X^{\geq h}_{\IC,\bar s,\overline{1_h}} \hookrightarrow X_{\IC,\bar s}.$$
 
\item The contragredient $\Lm^\vee$ of $\Lm$, is the local Langlands correspondence.
\end{itemize}
\end{notas}
%

\noindent \textit{More notations}:
for $a \in GL_d(F_v)/P_{tg,d}(F_v)$ as in notation \ref{nota-a}, we will also
denote by $HT_a(\pi_v,\Pi_t)$ (resp. $P_a(\pi_v,t)$) 
the image of $HT_{\overline{1_{tg}}}(\pi_v,\Pi_t)$ (resp. $P_{\overline{1_{tg}}}(\pi_v,t)$)
under the action of $a$, which is then equivariant for $P_a(F_v)$.
We will also consider, cf. the end of the previous paragraph, 
non necessarily pure strata $X^{\geq tg}_{\IC,\bar s,c}$
associated with some pure strata $X^{=h}_{\IC,\bar s,c}$ with $h \leq tg$ and 
more specifically to the case where $h=1$ to which we then restrict ourselves.
Denote by 
\addtocounter{thm}{1}
\begin{equation} \label{eq-mirabolic2}
HT_c(\pi_v,\Pi_t):=\ind_{P_{c \subset a}(F_c)}^{P_c(F_v)} HT_a(\pi_v,\Pi_t)
\end{equation}
(resp. $P_c(\pi_v,t):=\ind_{P_{c \subset a}(F_c)}^{P_c(F_v)} P_a(\pi_v,t)$)
where, using notation \ref{nota-a}, 
\begin{itemize}
\item the index $a \in GL_d(F_v)/P_{tg,d}(F_v)$ is 
any element containing $c \in GL_d(F_v)/P_{1,d}(F_v)$ and $P_c(F_v)$
(resp. $P_{c \subset a}(F_v)$) is the parabolic subgroup of $GL_d(F_v)$ stabilizing $c$
(resp. $c \subset a$);

\item we restrict the natural action of $P_a(F_v)$ on $HT_a(\pi_v,\Pi_t)$ 
(resp. on $P_a(\pi_v,t)$) to 
$P_{c \subset a}(F_v)$ before inducing.
\end{itemize}

Consider a geometric supersingular point $z$ and denote
by $i_z: \{ z \} \hookrightarrow X_{\IC,\bar s}$. We then have an action of $P_c(F_v)$
on $\ind_{(D_{v,d}^\times)^0 \varpi_v^\Zm}^{D_{v,d}^\times} \hi^i i_z^* \lexp p j^{=tg}_{c,!*} 
HT_c(\pi_v,\Pi_t)$ (resp. $\ind_{(D_{v,d}^\times)^0 \varpi_v^\Zm}^{D_{v,d}^\times}\hi^i i_z^! 
\lexp p j^{=tg}_{c,!*} HT_c(\pi_v,\Pi_t)$). From (\ref{eq-mirabolic2})  we then obtain
the following abstract description.

\begin{lemma} \label{lem-important}
For any $tg-d < i \leq 0$, 
$\ind_{(D_{v,d}^\times)^0 \varpi_v^\Zm}^{D_{v,d}^\times} \hi^i i_z^* \lexp p j^{=tg}_{c,!*} 
HT_c(\pi_v,\Pi_t)$ as a representation of the mirabolic group
$M_c(F_v)$ associated with $c$, cf. notation \ref{nota-Mc}, is 
isomorphic to a small mirabolic induced representation 
of \ref{prop-zele}, cf. also the remark after loc. cit.
$$(\Pi_t)_{|M_c(F_v)} \times \tau,$$ 
for $\tau$ some representation of $GL_{d-tg}(F_v)$ and where, by abuse of
notation in the term $(\Pi_t)_{|M_c(F_v)}$ of the above formula, 
$M_c(F_v)$ is the mirabolic subgroup of $GL_{tg}(F_v)$.

\end{lemma}

\begin{nota} \label{nota-fact0}
Let $X^{\geq 1}_{\IC,\bar s,c}$ be a pure stratum and denote by
$$j_{\neq c}:=j^{\geq 1}_{\neq c}:X^{\geq 1}_{\IC,\bar s} \setminus X^{\geq 1}_{\IC,\bar s,c}
\hookrightarrow X^{\geq 1}_{\IC,\bar s}.$$
\end{nota}

For $X^{\geq 1}_{\IC,\bar s,c} \neq X^{\geq 1}_{\IC,\bar s,c'}$ two distinct pure strata,
and for $h \geq 2$, we write $\langle c, c' \rangle$ the subspace of $F_v^d$
generated by $\{ c,c' \}$ and 
$$X^{=h}_{\IC,\bar s,\langle c,c' \rangle}=
\coprod_{\atop{a: \dim a=h}{\langle c,c' \rangle \subset a}} X^{=h}_{\IC,\bar s,a},$$
with $j^{=h}_{\langle c,c' \rangle}:X^{=h}_{\IC,\bar s,\langle c,c' \rangle} \hookrightarrow
X^{\geq h}_{\IC,\bar s,\langle c,c' \rangle} \hookrightarrow X^{\geq 1}_{\IC,\bar s}.$

Consider a pure stratum $X^{=h}_{\IC,\bar s,a}$ with $a \supset \langle c,c' \rangle$.
For $HT_a(\pi_v,\Pi_t)$ a Harris-Taylor local system on $X^{=h}_{\IC,\bar s,a}$,
we will denote by 
\addtocounter{thm}{1}
\begin{equation} \label{eq-mirabolic3}
HT_{\langle c,c' \rangle}(\pi_v,\Pi_t):=
\ind_{P_{\langle c,c' \rangle \subset a}(F_v)}^{P_{\langle c,c'\rangle}(F_v)} HT_a(\pi_v,\Pi_t),
\end{equation}
where $P_{\langle c,c' \rangle}(F_v)$ (resp. $P_{\langle c,c' \rangle \subset a}(F_v)$) 
is the parabolic subgroup of elements of 
$GL_d(F_v)$ stabilizing $\langle c,c' \rangle$ (resp. $\langle c,c'\rangle \subset a$).

Consider now the subgroup $P_{c,c'}(F_v)$
of $P_c(F_v)$ stabilizing $c'$. Every element of $P_c(F_v)$
induces a endomorphism of $F_v^d/c$ and the image of $P_{c,c'}(F_v)$ is then 
the parabolic $P_{c'}(F_v)$ of $GL(F_v^d/c)$.

\begin{lemma} \label{lem-j-c}
With the previous notations, and $\pi_v$ an irreducible cuspidal entire
$\overline \Qm_l$-representation of $GL_g(F_v)$, we have the following short exact sequence
of $P_{c,c'}(F_v)$-equivariant perverse sheaves
\begin{multline*}
0 \rightarrow
\lexp p j^{=(t+1)g}_{\langle c,c' \rangle,!*} HT_{\langle c,c' \rangle}
\bigr (\pi_v,\st_{t+1}(\pi_v) \bigl ) (\frac{1}{2})
\longrightarrow \\ 
j^{=1}_{\neq c',!} ~ j^{=1,*}_{\neq c'} 
\bigl ( \lexp p j^{=tg}_{c,!*} HT_{c}(\pi_v,\st_t(\pi_v) )
\bigr ) \\ \longrightarrow \lexp p  j^{=1}_{\neq c',!*} ~ j^{=1,*}_{\neq c'} 
\bigl ( \lexp p j^{=tg}_{c,!*} HT_{c}(\pi_v,\st_t(\pi_v) ) \bigr ) \rightarrow 0.
\end{multline*}
\end{lemma}

\rem using the main results of \cite{boyer-duke}, we could easily proved that the lemma 
is still valid over $\overline \Zm_l$.

\begin{proof}
Recall that for $P$ a perverse sheaf on $X_{I,\bar s}$, the kernel of 
$j^{=1}_{\neq c',!}j^{=1,*}_{\neq c'} P \twoheadrightarrow j^{=1}_{\neq c',!*}j^{=1,*}_{\neq c'} P$
is given by $\lexp p \hi^{-1} i^{1,*}_{c'} P$ so that we are reduced to prove that
$$\lexp p \hi^{-1} i^{1,*}_{c'} (  \lexp p j^{=tg}_{c,!*} HT_c(\pi_v,\st_t(\pi_v)) )
\simeq \lexp p j^{=(t+1)g}_{\langle c,c' \rangle,!*} HT_{\overline \Qm_l, \langle c,c' \rangle}
\bigr (\pi_v,\st_{t+1}(\pi_v) \bigl ) (\frac{1}{2}).$$
In \cite{boyer-invent2} 4.3.1, we described $j^{=tg}_{a,!} 
HT_{a} (\pi_v,\st_t(\pi_v))$ in the Grothendieck group of equivariant
perverse sheaves and the weight filtration gives us a filtration, cf. also \cite{boyer-FT} (5.4), with the successive
graded parts 
$$\lexp p j^{=(t+\delta)g}_{a,!*} HT_a(\pi_v,\st_t(\pi_v) \{ \frac{\delta(g-1)}{2} \}
\otimes \st_\delta(\pi_v)  \{ \frac{t(1-g)}{2} \}) (\delta/2).$$
By inducing from $P_{c \subset a}(F_v)$ to $P_c(F_v)$, we then obtain a filtration 
$\Fil^\bullet_c(t)$ of
$\lexp p j^{=tg}_{c,!} HT_c(\pi_v,\st_t(\pi_v))$ with graded parts
$$\gr^{-\delta}_c(t):=\lexp p j^{=(t+\delta)g}_{c,!*} HT_c(\pi_v,\st_t(\pi_v)_{|P_c(F_v)} \{ \frac{-\delta}{2} \}
\times \st_\delta(\pi_v)  \{ \frac{t}{2} \}) (\delta/2),$$
where by lemma \ref{lem-mirabolique},  and taking into account the notation 
$\times$ in \ref{nota-ind},
$$\bigl ( \st_t(\pi_v \{ \frac{-\delta}{2} \} )  \bigr )_{|P_{c}(F_v)} \times \st_\delta(\pi_v \{ \frac{t}{2} \}) 
\simeq \st_{t+\delta}(\pi_v)_{|P_{c}(F_v)}.$$ 
We then apply then the functor $\lexp p \hi^{-1} i^{1,*}_{c'}$ to this filtration
of $j^{=tg}_{c,!} HT_{c} (\pi_v,\st_t(\pi_v))$, so that we obtain
$$
\lexp p \hi^{-1} i^{1,*}_{c'} \bigl( \lexp p j^{=tg}_{c,!*} HT_{c}
(\pi_v,\st_t(\pi_v)) \bigr ) \simeq \lexp p \hi^0  i^{1,*}_{c'}  \Fil^{-1}_c(t) 
\twoheadrightarrow  \lexp p \hi^0  i^{1,*}_{c'}  \gr^{-1}_c(t),$$
with $ \lexp p \hi^0  i^{1,*}_{c'}  \gr^{-1}_c(t) \simeq
\lexp p j^{=(t+1)g}_{\langle c,c' \rangle,!*} HT_{\langle c,c' \rangle}  ( \pi_v,
\st_{t+1}(\pi_v)  ) (\frac{1}{2})$.

Now we use the computations of \cite{boyer-FT} corollary 6.6, where we proved the same
result for the whole of $X^{\geq 1}_{I,\bar s}$ instead of $X^{\geq 1}_{I,\bar s,c}$, i.e.
$$\lexp p \hi^{-1} i^{1,*}_{c'} \bigl( \lexp p j^{=tg}_{!*} HT
(\pi_v,\st_t(\pi_v)) \bigr ) \simeq \lexp p j^{=(t+1)g}_{c',!*} HT_{c}(\pi_v,
\st_{t+1}(\pi_v)  ) (\frac{1}{2}).$$
Note also, cf. \cite{boyer-FT} lemma 6.2, that 
$\lexp p \hi^{-k} i^{1,*}_{c'} \bigl( \lexp p j^{=tg}_{!*} HT
(\pi_v,\st_t(\pi_v)) \bigr ) $ is zero for any $k \neq -1$. We then apply the functor 
$\lexp p \hi^{-1} i^{1,*}_{c'}$ to the short exact sequence of perverse sheaves
\begin{multline*}
0 \rightarrow  \lexp p j^{=tg}_{c,!*} HT_{c} (\pi_v,\st_t(\pi_v)) \longrightarrow
 \lexp p j^{=tg}_{!*} HT (\pi_v,\st_t(\pi_v)) \\
  \longrightarrow  \lexp p j^{=tg}_{\neq c,!*} HT_{\neq c}
(\pi_v,\st_t(\pi_v)) \rightarrow 0,
\end{multline*}
which gives us in particular
$$\lexp p \hi^{-1} i^{tg+1,*}_{c'} \bigl( \lexp p j^{=tg}_{c,!*} HT_{c}
(\pi_v,\st_t(\pi_v)) \bigr ) \hookrightarrow \lexp p j^{=(t+1)g}_{c',!*} HT_{c}(\pi_v,
\st_{t+1}(\pi_v)  ) (\frac{1}{2}),$$
which factorizes to $\lexp p j^{=(t+1)g}_{\langle c,c' \rangle,!*} HT_{\langle c,c' \rangle}  
( \pi_v, \st_{t+1}(\pi_v)  ) (\frac{1}{2})$, just by considering the supports.

%
%
%
%
\end{proof}

%
%
%
%

\section{Some coarse filtrations of $\Psi_\IC$}
\label{para-coarse}

\subsection{Filtrations of free perverse sheaves}
\label{para-filtration}

Let $S=\spec \Fm_q$ and $X/S$ of finite type, then the usual $t$-structure on
$\DC(X,\overline \Zm_l):=D^b_c(X,\overline \Zm_l)$ is
$$\begin{array}{l}
A \in \lexp p \DC^{\leq 0}(X,\overline \Zm_l)
\Leftrightarrow \forall x \in X,~\hi^k i_x^* A=0,~\forall k >- \dim \overline{\{ x \} } \\
A \in \lexp p \DC^{\geq 0}(X,\overline \Zm_l) \Leftrightarrow \forall x \in X,~\
\hi^k i_x^! A=0,~\forall k <- \dim \overline{\{ x \} }
\end{array}$$
where $i_x:\spec \kappa(x) \hookrightarrow X$ and $\hi^k(K)$ is the $k$-th sheaf 
of cohomology of $K$.

\begin{nota} 
Denote by $\lexp p \CC(X,\overline \Zm_l)$ the heart of this $t$-structure with associated 
cohomology functors $\lexp p \hi^i$. For a functor $T$ we denote by 
$\lexp p T:=\lexp p \hi^0 \circ T$.
\end{nota}

The category  $\lexp p \CC(X,\overline \Zm_l)$ is abelian equipped, cf.
\cite{boyer-torsion} \S 1.1, with a torsion theory 
$(\TC,\FC)$ where $\TC$ (resp. $\FC$) is the full subcategory of objects $T$ (resp. $F$) 
such that $l^N 1_T$ is trivial for some large enough $N$(resp. $l.1_F$ is a monomorphism). 
Recall that this means in particular 
that for every object $A$ there exists a short exact sequence
$$0 \rightarrow A_{tor} \longrightarrow A \longrightarrow A_{free} \rightarrow 0$$
with $A_{tor} \in \TC$ and $A_{free} \in \FC$.
Applying Grothendieck-Verdier duality, we obtain
$$\begin{array}{l}
\lexp {p+} \DC^{\leq 0}(X,\overline \Zm_l):= \{ A \in \lexp p \DC^{\leq 1}(X,\overline \Zm_l):~
\lexp p \hi^1(A) \in \TC \} \\
\lexp {p+} \DC^{\geq 0}(X,\overline \Zm_l):= \{ A \in \lexp p \DC^{\geq 0}(X,\overline \Zm_l):~
\lexp p \hi^0(A) \in \FC \} \\
\end{array}$$
with heart$\lexp {p+} \CC(X,\overline \Zm_l)$  equipped with its torsion theory
$(\FC,\TC[-1])$.

\begin{defin} (cf. \cite{boyer-torsion} \S 1.3) \label{defi-FC}
Let
$$\FC(X,\overline \Zm_l):=\lexp p \CC(X,\overline \Zm_l) \cap \lexp {p+} \CC(X,\overline \Zm_l)
=\lexp p \DC^{\leq 0}(X,\overline \Zm_l) \cap \lexp {p+} \DC^{\geq 0}(X,\overline \Zm_l)$$
the quasi-abelian category of free perverse sheaves over $X$.
\end{defin}

\noindent \textit{Remark}: for an object $L$ of $\FC(X,\overline \Zm_l)$, 
we will consider filtrations
$$L_1 \subset L_2 \subset \cdots \subset L_e=L$$ 
such that for every $1 \leq i \leq e-1$, $L_i \hookrightarrow L_{i+1}$ is a
strict monomorphism, i.e. $L_{i+1}/L_i$ is an object of $\FC(X,\overline \Zm_l)$.

Consider an open subscheme $j:U \hookrightarrow X$ and 
$i: F:=X \backslash U \hookrightarrow X$. Then
$$\lexp {p+} j_! \FC(U,\overline \Zm_l) \subset \FC(X,\overline \Zm_l) \quad \hbox{ and } \quad \lexp p j_* 
\FC(U,\overline \Zm_l) \subset \FC(X,\overline \Zm_l).$$
Moreover, if $j$ is affine then $j_!$ is $t$-exact and $j_!=\lexp p j_!=\lexp {p+} j_!$.

\begin{lemma} \label{lem-libre0}
Consider $L \in \FC(X,\overline \Zm_l)$ such that $j_! j^* L \in \FC(X,\overline \Zm_l)$. Then
$i_*\lexp p \hi^{-\delta} i^* L$ is trivial for every $\delta \neq 0,1$; for $\delta=1$
it belongs to $\FC(X,\overline \Zm_l)$.
\end{lemma}

\rem If $j$ is affine then the condition $j_! j^* L \in \FC(X,\overline \Zm_l)$ is fulfilled.

\begin{proof}
Start from the following distinguished triangle
$j_!j^* L \longrightarrow L \longrightarrow i_*i^* L \leadsto$.
From the perversity of $L$ and $j_! j^* L$, the long exact sequence of perverse cohomology is
$$0 \rightarrow i_* \lexp p \hi^{-1} i^* L \longrightarrow \lexp p j_! j^* L
\longrightarrow L \longrightarrow i_* \lexp p \hi^0i^* L \rightarrow 0.$$
The freeness of $ i_* \lexp p \hi^{-1} i^* L$ then follows from those of 
$\lexp p j_! j^* L=j_!j^* L$.
\end{proof}

\begin{defin} Recall the following notions, cf. \cite{boyer-torsion} definition 1.3.4.
\begin{itemize}
\item A \emph{bimorphism} of $\FC(X,\overline \Zm_l)$, written $L \htarrow L'$, 
is both a monomorphism in $\lexp {p} \CC(X,\overline \Zm_l)$ and an epimorphism
in $\lexp {p+} \CC(X,\overline \Zm_l)$. If moreover the cokernel in 
$\lexp {p} \CC(X,\overline \Zm_l)$ is of dimension strictly less than those of the support of $L$,
we will write $L \htarrow_+ L'$.

\item A morphism $L \longrightarrow L'$ is a strict monomorphism (resp.
a strict epimorphism) if it is a monomorphism (resp. an epimorphism) in
$\lexp {p+} \CC(X,\overline \Zm_l)$ (resp. in $\lexp {p} \CC(X,\overline \Zm_l)$) in which case
we denote it by $L \harrow L'$ (resp. $L \tarrow L'$).
\end{itemize}
\end{defin}

For a free $L \in \FC(X,\overline \Zm_l)$, we consider the following diagram
$$\xymatrix{ 
& L \ar[drr]^{\can_{*,L}} \\
\lexp {p+} j_! j^* L \ar[ur]^{\can_{!,L}} 
\ar@{->>}[r]|-{+} & \lexp p j_{!*}j^* L \ar@{^{(}->>}[r]_+ & 
\lexp {p+} j_{!*} j^* L \ar@{^{(}->}[r]|-{+} & \lexp p j_*j^* L
}$$
where the bottom row is, cf. the remark following 1.3.10 of \cite{boyer-torsion},
the canonical factorisation of $\lexp {p+} j_! j^* L \longrightarrow \lexp {p} j_* j^* L$
and where the maps $\can_{!,L}$ and $\can_{*,L}$ are given by the adjunction property.

\begin{nota} \label{nota-filtration0} (cf. lemma 2.1.2 of \cite{boyer-torsion})
We introduce the filtration $\Fil^{-1}_{U,!}(L) \subset \Fil^0_{U,!}(L) \subset L$ with
$$\Fil^0_{U,!}(L)=\im_\FC (\can_{!,L}) \quad \hbox{ and } \quad \Fil^{-1}_{U,!}(L)=
\im_\FC \Bigl ( (\can_{!,L})_{|\PC_L} \Bigr ),$$
where $\PC_L:=i_*\lexp p \hi^{-1}_{free} i^*j_* j^* L$ is the kernel of
$\ker_\FC \Bigl ( \lexp {p+} j_!  j^* L \twoheadrightarrow \lexp p j_{!*} j^* L \Bigr )$.
\end{nota}

\noindent \textit{Remark}: we have
$L/\Fil^0_{U,!}(L) \simeq i_* \lexp {p+} i^* L$ and 
$\lexp p j_{!*} j^* L \htarrow_+ \Fil^0_{U,!}(L)/\Fil^{-1}_{U,!}(L)$, which gives, cf.
lemma 1.3.11 of \cite{boyer-torsion}, a commutative triangle
$$\xymatrix{
\lexp p j_{!*} j^* L \ar@{^{(}->>}[rr]_+ \ar@{^{(}->>}[drr]_+ 
& & \Fil^0_{U,!}(L)/\Fil^{-1}_{U,!}(L) \ar@{^{(}->>}[d]_+ \\
& & \lexp {p+} j_{!*} j^* L.
}$$

\begin{nota}  \label{nota-cofiltration0} (cf. \cite{boyer-torsion} 2.1.4) 
Dually there is a cofiltration 
$L \twoheadrightarrow \CoFil_{U,*,0}(L) \twoheadrightarrow \CoFil_{U,*,1}(L)$
where
$$\CoFil_{U,*,0}(L)=\coim_\FC(\can_{*,L}) \quad \hbox{ and } \quad 
\CoFil_{U,*,1}(L)=\coim_\FC(p_L \circ \can_{*,L}),$$
with $p_L:\lexp {p+} j_{!*} j^* L \twoheadrightarrow Q_L:=i_* \lexp p \hi^0_{free} i^*j_*j^* L$.
\end{nota}

\noindent \textit{Remark}: the kernel $\cogr_{U,*,1}(L)$ of 
$\CoFil_{U,*,0}(L) \twoheadrightarrow \CoFil_{U,*,1}(L)$ verifies
$$ \lexp p j_{!*} j^* L \htarrow_+ \cogr_{U,*,1}(L) \htarrow_+ \lexp {p+} j_{!*} j^* L.$$
The kernel  $\cogr_{U,*,0}(L)$ of $L \twoheadrightarrow \CoFil_{U,*,0}(L)$ is
isomorphic to $i_* \lexp p i^! L$.

Consider now $X$ equipped with a stratification consisting of closed subsets
$$X=X^{\geq 1} \supset X^{\geq 2} \supset \cdots \supset X^{\geq d,}$$
and let $L \in \FC(X,\overline \Zm_l)$.
For $1 \leq h <d$, denote by
$X^{1 \leq h}:=X^{\geq 1}-X^{\geq h+1}$ and $j^{1 \leq h}:X^{1 \leq h} \hookrightarrow
X^{\geq 1}$. We then define
$$\Fil^r_{!}(L):=\im_\FC \Bigl ( \lexp {p+} j^{1 \leq r}_! j^{1 \leq r,*}L \longrightarrow L\Bigr ),$$
which gives a filtration
$$0=\Fil^{0}_{!}(L) \subset \Fil^1_{!}(L) \subset \Fil^2_{!}(L) \cdots \subset \Fil^{d-1}_{!}(L) 
\subset \Fil^d_{!}(L)=L.$$

Dually, the following 
$$\CoFil_{*,r}(L)=\coim_\FC\Bigl ( L \longrightarrow \lexp {p} j^{1 \leq r}_* j^{1 \leq r,*}L \Bigr ),$$
define a cofiltration
\begin{multline*}
L = \CoFil_{*,d}(L) \twoheadrightarrow \CoFil_{*,d-1}(L) \twoheadrightarrow \cdots \\
\cdots
\twoheadrightarrow \CoFil_{*,1}(L) \twoheadrightarrow \CoFil_{*,0}(L)=0,
\end{multline*}
and a filtration
$$0=\Fil_{*}^{-d}(L) \subset \Fil_{*}^{1-d}(L) \subset \cdots \subset \Fil_{*}^0(L)=L$$
where $$\Fil_{*}^{-r}(L):=\ker_\FC \bigl ( L \twoheadrightarrow \CoFil_{*,r}(L) \bigr ).$$
Note these two constructions are exchanged by Grothendieck-Verdier duality, 
$$D \Bigl ( \CoFil_{!,-r}(L) \Bigr ) \simeq \Fil^{-r}_{*}( D(L)) \hbox{ and }
D \Bigl ( \CoFil_{*,r}(L) \Bigr ) \simeq \Fil^{r}_{!}( D(L)).$$

We can also refine the previous filtrations with the help of $\Fil^{-1}_{U,!}(L)$, 
cf. \cite{boyer-torsion} proposition 2.3.2, to obtain exhaustive filtrations
\addtocounter{thm}{1}
\begin{multline} \label{eq-fill}
0=\Fill^{-2^{d-1}}_{!}(L) \subset \Fill^{-2^{d-1}+1}_{!}(L) \subset \cdots  \\
\cdots \subset
\Fill^0_{!}(L) \subset \cdots \subset \Fill^{2^{d-1}-1}_{!}(L)=L,
\end{multline}
such that the graded pieces $\grr^k(L)$ are simple over $\overline \Qm_l$,
i.e. verify $\lexp p j^{= h}_{!*} j^{= h,*} \grr^k(L)  \htarrow_+ \grr^k(L)$
for some $h$.
Dually using $:\coFil_{U,*,1}(L)$, we construct a cofiltration
$$L=\CoFill_{*,2^{d-1}}(L) \twoheadrightarrow \CoFill_{*,2^{d-1}-1}(L)
\twoheadrightarrow \cdots \twoheadrightarrow \CoFill_{*,-2^{d-1}}(L)=0$$
and a filtration $\Fill_{*}^{-r}(L):=\ker_\FC \bigl ( L \twoheadrightarrow \CoFill_{*,r}(L) \bigr )$.
These two constructions are exchanged by duality 
$$D \bigl ( \CoFill_{*,r}(L) \bigr ) \simeq \Fill^r_{!} \bigl ( D(L) \bigr ) \hbox{ and }
D \bigl ( \CoFill_{!,r}(L) \bigr ) \simeq \Fill^r_{*} \bigl ( D(L) \bigr )$$
and can be mixed if we want to.

\subsection{Remarks and terminology about perverse sheaves}
\label{para-terminology}

Be cause the previous definitions come 
from the geometry, it is then possible to
construct filtrations whatever is the ring of coefficients. 
Moreover when you want to understand the graded pieces, 
\begin{itemize}
\item you can first look at
these filtrations over $\overline \Qm_l$ which gives you the simple perverse sheaves described in
terms of an intermediate extension $i_* j_{!*} \LC[-\delta]$ of some local system $\LC$
living in some locally closed stratum $\xymatrix{
U \ar@{^{(}->}[r]^j & \overline U \ar@{^{(}->}[r]^i & X }$
where $\delta$ is the dimension of $U$.

\item Then you have to understand the $\overline \Zm_l$-lattice of $\LC$; in the following we will
speak about \emph{the lattice of the perverse sheaf}.

\item And finally determine \emph{the position} of the graded piece between the two natural 
intermediate extension
$i_* \lexp p j_{!*} \LC[-\delta] \htarrow_+ i_* \lexp {p+} j_{!*} \LC[-\delta]$.
\end{itemize}
One also have to take into account that the lattices and the positions depend strongly on the 
order of the graded pieces,
i.e. for two different filtrations $\Fil_1^\bullet$ and $\Fil_2^\bullet$, then two
graded pieces $\gr^{k_1}_1$ and $\gr^{k_2}_2$ which are isomorphic over $\overline \Qm_l$, 
might be not isomorphic over $\overline \Zm_l$ either because the lattices, 
or their positions, are different.

Finally as remarked in the introduction, when taking image in $\FC$
or kernel in $\FC^+$, you loose control of lattices and positions:\footnote{The understanding of positions in the previous 
meaning, is solved in \cite{boyer-duke}.} 
we then speak of a
\emph{saturation} process as it corresponds to the usual saturation of lattices in the
case where the geometric support is of dimension $0$. In the following we will
focus on graded pieces concentrated on the supersingular locus so that there is no issue
about the positions. Concerning the lattices of these graded pieces,
we advise the reader
when reading the arguments of \S \ref{para-lattice} to focus on this issue to understand how
we manage to recover the lattices.

\rem by arguing inductively on the Lubin-Tate spaces, we will in fact be able, cf. the proof of
theorem \ref{thm-principal1}, to understand among all
the lattices of Harris-Taylor perverse sheaf $P(t,\pi_v)(\frac{1-t+2k}{2})$ with $0 \leq k \leq t-1$,
given by filtrations of stratification of $\Psi_\IC$, those for $k=0$ and $k=t-1$. It would then be possible, and quite
easy using the coarse filtrations of \S \ref{para-coarse}, to describe the others which then depend on how we decide
to filtrate $\Psi_\IC$, as opposed to the cases where $k=0,t-1$ considered here. 
We postpone this work to the day when we will find an application.

\subsection{Supercuspidal decomposition of $\Psi_\IC$}
\label{para-dec-psi}

\begin{nota}
For $I \in \IC$, let 
$$\Psi_{I,\Lambda}:=R\Psi_{\eta_v,I}(\Lambda[d-1])(\frac{d-1}{2})$$
be the vanishing cycle autodual perverse sheaf on $X_{I,\bar s}$.
When $\Lambda=\overline \Zm_l$, we will simply write $\Psi_I$.
\end{nota}

As before, we will use the notation $\Psi_\IC$ for the system $(\Psi_I)_{I \in \IC}$.
Recall the following result of \cite{h-t} relating $\Psi_{\IC}$ with Harris-Taylor local systems.

\begin{propo} \label{prop-fbartau} 
(cf. \cite{h-t} proposition IV.2.2 and \S 2.4 of \cite{boyer-invent2}) \\
There is an isomorphism 
$G(\Am^{\oo,v}) \times P_{h,d}(F_v) \times W_{F_v}$-equivariant
$$\ind_{(D_{v,h}^\times)^0 \varpi_v^\Zm}^{D_{v,h}^\times} 
\bigl ( \hi^{h-d-i} \Psi_{\IC,\overline \Zm_l} \bigr )_{|X^{=h}_{\IC,\bar s,\overline{1_h}}} \simeq
\bigoplus_{\bar \tau \in \RC_{ \overline \Fm_l}(h)} \LC_{\overline \Zm_l,\overline{1_h}}
(\UC^{h-1-i}_{\bar \tau,\Nm}),$$
where
\begin{itemize}
\item $\LC_{\overline \Zm_l,\overline{1_h}}(\UC^{h-1-i}_{\bar \tau,\Nm})$ is the local
system over $X^{=h}_{\IC,\bar s,\overline{1_h}}$ associated with
$\UC^{h-1-i}_{\bar \tau,\Nm}$ viewed as a representation of $\DC_{v,h}^\times$,
cf. the remark before \ref{nota-rem};

\item $\RC_{ \overline \Fm_l}(h)$ is the set of equivalence classes of irreducible
$\overline \Fm_l$-representations of $D_{v,h}^\times$;

\item for $\bar \tau \in \RC_{ \overline \Fm_l}(h)$ and $V$ a 
$\overline \Zm_l$-representation of $D_{v,h}^\times$, then $V_{\bar \tau}$
denotes, cf. \cite{dat-torsion} \S B.2, the direct factor of $V$ whose irreducible
subquotients are isomorphic to a subquotient of $\bar \tau_{|\DC^\times_{v,h}}$
where $\DC_{v,h}$ is the maximal order of $D_{v,h}$.

\item With the previous notation, $\UC^i_{\bar \tau,\Nm}:= 
\big ( \UC^i_{F_v,\overline \Zm_l,d,n} \big )_{\bar \tau}$.

\item The matching at finite levels between the system indexed by $\IC$ and those by $\Nm$
is given by the map $m_v:\IC \longrightarrow \Nm$.
\end{itemize}
\end{propo}

\rem for $\bar \tau \in \RC_{\overline \Fm_l}(h)$, and a lifting $\tau$
which by Jacquet-Langlands correspondence can be written $\tau \simeq \pi[t]_D$
for $\pi$ irreductible cuspidal, let $\varrho \in \scusp_{\overline \Fm_l}(g)$
be in the supercuspidal support. Then the inertial class of $\varrho$ depends only on
$\bar \tau$ and we will use the following notation.

\begin{nota}
With the previous notation, we write $V_\varrho$ for $V_{\bar \tau}$. 
\end{nota}

The description of the various filtration from the previous section
applied to $\Psi_{\IC,\overline \Qm_l}$ is given in \cite{boyer-torsion} \S 3.4.
Over $\overline \Zm_l$, first note that
$\Psi_{\IC, \overline \Zm_l}$ is an object of $\FC(X_{\IC,\bar s}, \overline \Zm_l)$.
Indeed, by \cite{ast} proposition 4.4.2,
$\Psi_{\IC, \overline \Zm_l}$ is an object of 
$\lexp p \DC^{\leq 0}(X_{\IC,\bar s},\overline \Zm_l)$. By
\cite{ill} variant 4.4 of theorem 4.2, we have 
$D \Psi_{\IC, \overline \Zm_l} \simeq \Psi_{\IC, \overline \Zm_l}$, so that
$$\Psi_{\IC, \overline \Zm_l} \in \lexp p \DC^{\leq 0}(X_{\IC,\bar s},\overline \Zm_l) \cap 
\lexp {p+} \DC^{\geq 0}(X_{\IC,\bar s},\overline \Zm_l)=\FC(X_{\IC,\bar s}, \overline \Zm_l).$$

We can then deduce from the description of the filtrations of $\Psi_{\IC,\overline \Qm_l}$
the same sort of description except that first
we have no control on the bimorphism $\lexp p j^{= h}_{!*} j^{= h,*} \grr^k(L) 
\htarrow_+ \grr^k(L)$ and secondly all the contribution relatively to irreducible cuspidal
$\overline \Qm_l$-representations should be considered altogether.
About this last point, we have the following result.

\begin{propo} \label{prop-decomposition}
We have a decomposition
$$\Psi_\IC \simeq \bigoplus_{g=1}^d \bigoplus_{\varrho \in \scusp_{\overline \Fm_l}(g)} \Psi_{\varrho}$$
with $\Psi_\varrho \otimes_{\overline \Zm_l} \overline \Qm_l \simeq 
\bigoplus_{\pi_v \in \cusp(\varrho)} \Psi_{\pi_v}$ where
the irreducible constituent of 
$\Psi_{\pi_v}$ are exactly the
perverse Harris-Taylor sheaves attached to $\pi_v$, i.e. with the notations of \ref{nota-rem},
the $P(\pi_v,t)(\frac{1-t+\delta}{2})$ with $1 \leq t \leq d/g$ and $0 \leq \delta \leq t-1$.
\end{propo}

\noindent \textit{Remark}:
the graded pieces $\gr^h_!(\Psi_{\varrho})$ of the previous filtration of stratification
of $\Psi_{\varrho}$ verify
$$j^{=h,*} \gr^h_!(\Psi_{\varrho}) \simeq 
\left \{ \begin{array}{ll} 0 & \hbox{if } g \nmid h \\
\LC_{\overline \Zm_l}(\varrho[t]_D) & \hbox{for } h=tg. \end{array} \right.$$

\begin{proof}
We argue by induction on $r$ to show that there exists a decomposition
$$\Fil^r_!(\Psi_\IC)= \bigoplus_{g=1}^d \bigoplus_{\varrho \in \scusp_{\overline \Fm_l}(g)}  
\Fil^r_{!,\varrho} (\Psi_\IC).$$
The case $r=0$ being trivial, we suppose it is true for $r-1$. From
$j^{= r,*} \gr^r_!(\Psi_\IC) \simeq  \bigoplus_{g|r=tg}
\bigoplus_{\varrho \in \scusp_{\overline \Fm_l}(g)} \LC_{\overline \Zm_l} (\varrho[t]_D)$,
we obtain
$$\gr^r_!(\Psi_\IC) \simeq \bigoplus_{g|r}\bigoplus_{\varrho \in \scusp_{\overline \Fm_l}(g)} 
\gr^r_{!,\varrho} (\Psi_\IC)$$
with $j^{= r}_! \LC_{\overline \Zm_l}(\varrho[t]_D)[d-r] \twoheadrightarrow \gr^r_{!,\varrho} (\Psi_\IC)$
where the irreducible constituents of $\gr^r_{!,\varrho} (\Psi_\IC)
\otimes_{\overline \Zm_l} \overline \Qm_l$ are of type $\varrho$.

Consider two free perverse sheaves $A_1$ and $A_2$ and let $A$ be an extension 
$$0 \rightarrow A_1 \longrightarrow A \longrightarrow A_2 \rightarrow 0,$$
supposed to be split over $\overline \Qm_l$. Denote then the pull back $A'_2$ 
$$\xymatrix{
A'_2 \ar@{^{(}-->}[r] \ar@{^{(}-->}[d] & A \ar@{^{(}->}[d] \\
A_2\otimes_{\overline \Zm_l} \overline \Qm_l  \ar@{^{(}->}[r] &  A  \otimes_{\overline \Zm_l} \overline \Qm_l 
}$$
so that
\addtocounter{thm}{1}
\begin{equation} \label{eq-prop-extension}
\xymatrix{
 & A_1 \ar@{^{(}->}[d] \ar@{=}[r] & A_1 \ar@{^{(}->}[d] \\
 A'_2 \ar@{^{(}->}[r] \ar@{=}[d] & A \ar@{->>}[r] \ar@{->>}[d] & A'_1 \ar@{->>}[d]  \\
 A'_2 \ar@{^{(}->}[r] & A_2 \ar@{->>}[r] & T  \\
}
\end{equation}
where $T$ is the common cokernel of $A_1 \hookrightarrow A'_1$
and $A'_2 \hookrightarrow A_2$. Then $T=0$ if and only the extension $A$ is split.
Now suppose that $A_1$ (resp. $A_2$) is a Harris-Taylor perverse sheaf
of type $\varrho_1$ (resp. $\varrho_2$) with $\varrho_1$ and $\varrho_2$ 
not belonging to the same Zelevinsky line. Then the action of the Weil group on
$T[l]$ seen as a quotient of $A'_1$ (resp. of $A_2$) is isotypic relatively to 
the galois representation associated with $\varrho_1$ (resp. $\varrho_2$)
by the Langlands-Vigneras correspondence, which imposes that $T=0$.

By applying this general remark to $\gr^r_{!,\varrho_2}(\Psi_\IC)$, we conclude it is
in a direct sum with $\Fil^{r-1}_{!,\varrho_1}(\Psi_\IC)$, which, by varying  
$\varrho_1$ and $\varrho_2$, proves the result.
\end{proof}

In order to understand the next computations on $\Psi_\IC$, it might be useful for the reader to recall
the following description in the Grothendieck group of $\Psi_{\pi_v}$ given in \cite{boyer-invent2}
\addtocounter{thm}{1}
\begin{equation} \label{eq-psi-groth}
\bigl [ \Psi_{\pi_v} \bigr ]=\sum_{t=1}^{\lfloor \frac{d}{g} \rfloor } \sum_{k=0}^{t-1}
P(t,\pi_v)(\frac{1-t+2k}{2}),
\end{equation}
where $\pi_v$ is an irreducible cuspidal representation of $GL_g(F_v)$.
Then, cf. \cite{boyer-torsion} \S 3.4, 
the graded pieces $\gr^r_!(\Psi_{\pi_v})$ are zero if $r \not \in \{ g,2g,\cdots, \lfloor \frac{d}{g} \rfloor g \}$
and otherwise its image in the Grothendieck group is
\addtocounter{thm}{1}
\begin{equation} \label{eq-psi-groth2}
\bigl [ \gr^{kg}_! \bigl ( \Psi_{\pi_v} \bigr ) \bigr ]=\sum_{t=k}^{\lfloor \frac{d}{g} \rfloor } 
P(t,\pi_v)(\frac{1+t-2k}{2}).
\end{equation}
In particular for $k=1$, then 
\addtocounter{thm}{1}
\begin{equation} \label{eq-psi-groth3}
\bigl [ \gr^{g}_! \bigl ( \Psi_{\pi_v} \bigr ) \bigr ]=\sum_{t=1}^{\lfloor \frac{d}{g} \rfloor } 
P(t,\pi_v)(\frac{t-1}{2}).
\end{equation}

\subsection{Filtrations with the use of $j_{\neq c}$}

Denote by
$$\bar j: X_{\IC,\bar \eta} \hookrightarrow \overline X_{\IC} \hookleftarrow X_{\IC,\bar s}: \bar i,$$
and consider the following $t$-structure on 
$\overline X_{\IC}:=X_{\IC} \times_{\spec \OC_v} \spec \overline \OC_v$ obtained by
glueing
$$\Bigl ( \lexp p D^{\leq -1}(X_{\IC,\overline \eta},\overline \Zm_l), 
\lexp p D^{\geq -1}(X_{\IC,\overline \eta},\overline \Zm_l) \Bigr ) \quad \hbox{and} \quad
\Bigl ( \lexp p D^{\leq 0}(X_{\IC,\overline s},\overline \Zm_l), 
\lexp p D^{\geq 0}(X_{\IC,\overline s},\overline \Zm_l) \Bigr ).$$
The functors $\bar j_!$ and $\bar j_*=\lexp p {\bar j_{!*}}$ are then $t$-exact with
$$0 \rightarrow \Psi_{\IC} \longrightarrow \bar j_! \overline \Zm_l[d-1](\frac{d-1}{2}) 
\longrightarrow \bar j_* \overline \Zm_l[d-1](\frac{d-1}{2}) \rightarrow 0.$$

Consider now a pure stratum $X^{\geq 1}_{\IC,\bar s,c}$. Note then that
the morphism $\bar j_{\neq c}:\overline X_{\IC} \setminus 
X^{\geq 1}_{\IC,\bar s,c} \hookrightarrow \overline X_{\IC}$ is affine, cf. \cite{boyer-FT} 
beginning of \S 7.

\begin{lemma} \label{lem-icpsi}
The perverse sheaf $\Psi_c:=i^1_{c,*} \lexp p \hi^0 i^{1,*}_{c} \bigl ( \Psi_{\IC} \bigr )$ is free.
\end{lemma}

\begin{proof}
Let 
$\overline F:=\bar j_*\overline \Zm_l[d-1](\frac{d-1}{2})=
\bar j_{!*}\overline \Zm_l[d-1](\frac{d-1}{2})$
over $\overline X_{\IC}$.
Denote by $i_c^1:X^{\geq 1}_{\IC,\bar s,c} \hookrightarrow X^{\geq 1}_{\IC,\bar s}$, and
$\bar i_c:=\bar i \circ i_c^{\geq 1}$.
As $\Psi_{\IC}=\lexp p \hi^{-1} \bar i^* \bar j_* \overline \Zm_l[d-1](\frac{d-1}{2})$, we have to prove
that $i_c^{1,*} \lexp p \hi^{-1} \bar i^* \overline F$ is perverse for the $t$-structures $p$ and $p+$.
Consider the spectral sequence
$$E_2^{r,s}=\lexp p \hi^r i^{1*}_{c} \Bigl ( \lexp p \hi^s  \bar i^* \overline F \Bigr ) \Rightarrow 
\lexp p \hi^{r+s} \bar i^*_{c} \overline F.$$
As $\bar j$ is affine, by lemma \ref{lem-libre0}, we know that $\lexp p \hi^s \bar i^{*} \overline F$ is
trivial for $s<-1$. The epimorphism $\bar j_! \bar j^{*} \overline F \twoheadrightarrow \overline F$, 
gives also that $\lexp p \hi^0 \bar i^{*} \overline F=0$ so that the previous spectral sequence
degenerates at $E_2$ with
$$\lexp p \hi^r \bar i^{*}_{c} \overline F \simeq \lexp p \hi^{r+1} i^{1,*}_{c}
\bigl ( \lexp p \hi^{-1} \bar i^{*}\overline F \bigr ).$$
In the same way as $\bar j_{\neq c}:X_{\IC} \setminus X^{\geq 1}_{\IC,\bar s,c} 
\hookrightarrow X_{\IC}$ is affine, then, by lemma \ref{lem-libre0},
$\lexp p \hi^r \bar i^{*}_{c} \overline F$ is trivial for $r<-1$ and free for $r=-1$
which finishes the proof.
\end{proof}

The decomposition of \ref{prop-decomposition} gives
$\Psi_c \simeq \bigoplus_{g=1}^d\bigoplus_{\varrho \in \scusp_{\overline \Fm_l}(g)} 
\Psi_{\varrho,c}$. 
For any $\varrho \in \scusp_{\overline \Fm_l}(g)$ 
we then have the following short exact sequence of free perverse sheaves
\addtocounter{thm}{1}
\begin{equation} \label{eq-sec-psi}
0 \rightarrow j_{\neq c,!} j^{*}_{\neq c} \Psi_{\varrho}
\longrightarrow \Psi_{\varrho} \longrightarrow \Psi_{\varrho,!,c} \rightarrow 0,
\end{equation}
where $j_{\neq c}:X^{\geq 1}_{\IC,\bar s} \setminus X^{\geq 1}_{\IC,\bar s,c} \hookrightarrow
X^{\geq 1}_{\IC,\bar s}$.

\rem applied to $\Psi_{\pi_v}$, the equality (\ref{eq-psi-groth}) becomes, cf. \cite{boyer-torsion} \S 3.4
\addtocounter{thm}{1}
\begin{equation} \label{eq-psi-groth4}
\bigl [ \Psi_{\pi_v,!,c} \bigr ]=\sum_{t=1}^{\lfloor \frac{d}{g} \rfloor } 
P(t,\pi_v)_c(\frac{1-t}{2}).
\end{equation}

Consider the filtration of stratification
$$0=\Fil_{*}^{-d}(\Psi_{\varrho,!,c} ) \subset \Fil_{*}^{1-d}(\Psi_{\varrho,!,c} ) 
\subset \cdots \subset \Fil_{*}^0(\Psi_{\varrho,!,c} )=\Psi_{\varrho,!,c}.$$

\begin{propo} \label{prop-filtration1}
The graded pieces $\gr^h_*(\Psi_{\varrho,!,c})$ verify the following properties
\begin{itemize}
\item it is trivial if $h$ is not equal to some$-g_i(\varrho)+1 > -d$ for $i \geq -1$;

\item for such $i \geq -1$ with $g_i(\varrho) \leq d$, then
$$\gr^{-g_i(\varrho)+1}_*(\Psi_{\varrho,!,c}) \otimes_{\overline \Zm_l} \overline \Qm_l \simeq 
\bigoplus_{\pi_v \in \cusp(\varrho,i)} \gr^{-g_i(\varrho)+1}_*(\Psi_{\pi_v,!,c})$$
where $\gr^{-g_i(\varrho)+1}_*(\Psi_{\pi_v,!,c})$ is the push forward
$$\xymatrix{
\Psi_{\pi_v} \ar@{->>}[r] \ar@{->>}[d] & \Psi_{\pi_v,!,c} \ar@{-->>}[d] \\
\coFil_{*,g_i(\varrho)}(\Psi_{\pi_v}) \ar@{-->>}[r] & \gr^{-g_i(\varrho)+1}_*(\Psi_{\pi_v,!,c}).
}$$
\end{itemize}
\end{propo}

\rem from \cite{boyer-FT} proposition 7.1, the graded pieces 
$\gr^h_! \bigl ( \gr^{-g_i(\varrho)+1}_*(\Psi_{\pi_v,!,c}) \bigr )$
of the filtration of stratification are then
\begin{itemize}
\item trivial if $h$ is not of the shape $tg_i(\varrho) \leq d$,

\item and for $h=tg_i(\varrho) \leq d$, we have, if we consider for simplicity $c=\overline{1_1}$
$$\gr^{tg_i(\varrho)}_! \bigl ( \gr^{-g_i(\varrho)+1}_*(\Psi_{\pi_v,!,c}) \bigr )\simeq
\ind_{P_{1,h,d}(F_v)}^{P_{1,d}(F_v)} 
P(t,\pi_v)_{\overline{1_h}}(\frac{1-t}{2}).$$
\end{itemize}
Proposition 7.1 of \cite{boyer-FT} only deals with the Grothendieck group and not with the filtration,
which is the main point of our proposition here.

\begin{proof}
As explained in the remark before \S \ref{para-dec-psi}, as long as the statement do not speak
about the lattices and the positions, then it is only a statement on $\Psi_{\pi_v}$ and the precise description 
of $\Psi_{\pi_v}$.
For simplicity we suppose $c=\overline{1_1}$.
From \cite{boyer-torsion}, the graded piece
$\gr^h_! \bigl ( \coFil_{*,g_i(\varrho)}(\Psi_{\pi_v}) \bigr )$
of the filtration $\Fil^\bullet_! \bigl ( \coFil_{*,g_i(\varrho)}(\Psi_{\pi_v}) \bigr )$
are trivial for all $h\neq tg_i(\varrho) \leq d$ and
$$\gr^{tg_i(\varrho)}_! \bigl ( \coFil_{*,g_i(\varrho)}(\Psi_{\pi_v}) \bigr ) \simeq 
P(t,\pi_v)(\frac{1-t}{2}).$$

\begin{lemma} For $0 \leq r \leq d$, the 
$$\gr^h_! \Bigl ( i^{1}_{c,*} \lexp p \hi^0 i^{1,*}_{c}\Fil^r_! 
\bigl (\coFil_{*,g_i(\varrho)}(\Psi_{\pi_v}) \bigr ) \Bigr )$$
verify the following properties
\begin{itemize}
\item they are trivial if $r<g_i(\varrho)$;

\item for $tg_i(\varrho) \leq r < (t+1)g_i(\varrho)$, they are trivial if $h \neq tg_i(\varrho)$;

\item otherwise, for $a \in GL_d(F_v)/P_{h,d}(F_v)$ such that
$c \subset a$, it is isomorphic to
$$\ind_{P_{c \subset a}(F_v)}^{P_{c}(F_v)} P(t,\pi_v)_{a}(\frac{1-t}{2}).$$
\end{itemize}
\end{lemma}

\rem in the last point with $h=tg_i(\varrho)$, 
for $c=\overline{1_1}$ and $a=\overline{1_h}$ the formula is
$$\ind_{P_{1,h,d}(F_v)}^{P_{1,d}(F_v)} P(t,\pi_v)_{\overline{1_h}}(\frac{1-t}{2}).$$

\begin{proof}
Note first  that the statement is trivially true for $r < g_i(\varrho)$. Recall moreover that 
$$i^{1}_{c,*} \lexp p \hi^0 i^{1,*}_c P(t,\pi_v) \simeq P(t,\pi_v)_c:=
\ind_{P_{c \subset a}(F_v)}^{P_{c}(F_v)} P(t,\pi_v)_{a}(\frac{1-t}{2}),$$
where $a \in GL_d(F_v)/P_{tg_i(\varrho),d}(F_v)$ is such that $c \subset a$.
We then argue by induction through the short exact sequence
\begin{multline*}
0 \rightarrow \Fil^{r-1}_! \bigl (\coFil_{*,g_i(\varrho)}(\Psi_{\pi_v}) \bigr ) \Bigr ) \longrightarrow
\Fil^r_! \bigl (\coFil_{*,g_i(\varrho)}(\Psi_{\pi_v}) \bigr ) \Bigr ) \\
\longrightarrow
\gr^r_! \bigl (\coFil_{*,g_i(\varrho)}(\Psi_{\pi_v}) \bigr ) \Bigr ) \rightarrow 0.
\end{multline*}
If $r$ is not of the shape $tg_i(\varrho)$ there is nothing to prove, otherwise as
\begin{itemize}
\item the irreducible constituents of $i^{1}_{c,*} \lexp p \hi^0 i^{1,*}_c 
\Fil^{r-1}_! \bigl (\coFil_{*,g_i(\varrho)}(\Psi_{\pi_v}) \bigr ) \Bigr )$ are, by induction, 
intermediate extensions of Harris-Taylor local systems on $X^{=i}_{\IC,\bar s}$ for $i \leq r$,

\item and $i^{1}_{c,*} \lexp p \hi^{-1} i^{1,*}_c P(t,\pi_v) $ is supported 
on $X^{\geq r+1}_{\IC,\bar s}$,
\end{itemize}
then the cone map $i^{1}_{c,*} \lexp p \hi^{-1} i^{1,*}_c P(t,\pi_v) \longrightarrow 
i^{1}_{c,*} \lexp p \hi^0 i^{1,*}_c \Fil^{r-1}_! \bigl (\coFil_{*,g_i(\varrho)}(\Psi_{\pi_v}) 
\bigr )$
is trivial. The result follows then from the short exact sequence
\begin{multline*}
0 \rightarrow i^{1}_{c,*} \lexp p \hi^0 i^{1,*}_c 
\Fil^{r-1}_! \bigl (\coFil_{*,g_i(\varrho)}(\Psi_{\pi_v}) \bigr ) \longrightarrow
i^{1}_{c,*} \lexp p \hi^0 i^{1,*}_c 
\Fil^{r}_! \bigl (\coFil_{*,g_i(\varrho)}(\Psi_{\pi_v}) \bigr )  \\ \longrightarrow
i^{1}_{c,*} \lexp p \hi^0 i^{1,*}_c \gr^r_! \bigl (\coFil_{*,g_i(\varrho)}(\Psi_{\pi_v}) \bigr )    
\rightarrow 0.
\end{multline*}
\end{proof}

It suffices now to prove that the epimorphism
$$i^{1}_{c,*} \lexp p \hi^0 i^{1,*}_{c} \Psi_{\pi_v} \twoheadrightarrow 
i^{1}_{c,*} \lexp p \hi^0 i^{1,*}_{c}  \bigl (\coFil_{*,g_i(\varrho)}(\Psi_{\pi_v}) \bigr ) $$
is an isomorphism. For that it suffices to prove that, for every geometric point $z$, the germs at $z$
of the sheaves cohomology groups of these two perverse sheaves, are the same.

Let then $z$ be a geometric point of $X^{=h}_{\IC,\bar s,\overline{1_h}}$. 
By \cite{boyer-invent2}, the germ at $z$ of the $i$-th sheaf of cohomology
 $\hi^i j^{=kg}_{\overline{1_{kg}},*} HT_{\overline{1_{kg}}}(\pi_v,\st_k(\pi_v))
\otimes \Xi^{\frac{1-k}{2}}$ is zero if $(h,i)$ is not of the shape $(d-tg,tg-d+k-t)$ with 
$k \leq t \leq \lfloor \frac{d}{g} \rfloor$ and otherwise isomorphic to those of 
$$HT_{\overline{1_{tg}}} \bigl ( \pi_v,\st_k(\pi_v\{ \frac{k-t}{2} \}) \otimes 
\speh_{t-k} ( \pi_v \{ \frac{k}{2} \} ) \bigr ) \otimes \Xi^{\frac{1+t-2k}{2}}.$$
Then for $h=d-tg$ and $i=tg-d+k-t$, the fiber at $z$ of $\hi^i j^{=kg}_{\overline{1_1},*} HT_{\overline{1_1}}
(\pi_v,\st_k(\pi_v)) \otimes \Xi^{\frac{1-k}{2}}$  is isomorphic to those of
$$HT_{\overline{1_{tg}}} \Bigl ( \pi_v, \bigl ( \st_k(\pi_v\{ \frac{k-t}{2} \} 
)_{|P_{1,kg}(F_v)} \times 
\speh_{t-k} ( \pi_v \{ \frac{k}{2} \} ) \bigr ) \Bigr ) \otimes \Xi^{\frac{1+t-2k}{2}},$$
where we induce from $P_{1,kg}(F_v) \otimes GL_{(t-k)g}(F_v)$ to $P_{1,tg}(F_v)$. 
Moreover considering the weights, we see that the spectral sequence computing the 
fibers of sheaves of cohomology of $\Psi_{\pi_v,c}$ from those of 
$\gr^k_!(\Psi_{\pi_v,c})$ degenerate at $E_1$. From \ref{lem-mirabolique}, we have
$$\bigl ( \st_k(\pi_v\{ \frac{k-t}{2} \}) \bigr )_{|P_{1,kg}(F_v)} \times \speh_{t-k} ( \pi_v \{ \frac{k}{2} \} )
\simeq \bigl ( LT_\pi(k,t-1-k)_{\pi_v} \bigr )_{|P_{1,tg}(F_v)}$$
so that, by the main result of \cite{boyer-invent2}, the fiber at $z$ of 
$\hi^i  \Psi_{\pi_v,c}$ is isomorphic to those of  
$\hi^i \bigl ( \lexp p \hi^0 i_{c}^{1,*} \Psi_{\pi_v} \bigr )$, so we are done.
\end{proof}

Dually we have
\addtocounter{thm}{1}
\begin{equation} \label{eq-sec2-psi}
0 \rightarrow \Psi_{\varrho,*,c} \longrightarrow \Psi_{\varrho} \longrightarrow
j_{\neq c,*} j^{*}_{\neq c} \Psi_{\varrho} \rightarrow 0,
\end{equation}
such that the graded pieces $\gr^h_!(\Psi_{\varrho,*,c})$ verify the following properties
\begin{itemize}
\item it is trivial if $h$ is not equal to some $g_i(\varrho) \leq d$ for $i \geq -1$;

\item for such $i \geq -1$ with $g_i(\varrho) \leq d$, then
$$\gr^{g_i(\varrho)}_!(\Psi_{\varrho,*,c}) \otimes_{\overline \Zm_l} \overline \Qm_l \simeq 
\bigoplus_{\pi_v \in \cusp(\varrho,i)} \gr^{g_i(\varrho)}_!(\Psi_{\pi_v,*,c})$$
where $\gr^{g_i(\varrho)}_!(\Psi_{\pi_v,*,c})$ is the pull back
$$\xymatrix{
\gr^{g_i(\varrho)}_!(\Psi_{\pi_v,*,c}) \ar@{^{(}-->}[r] \ar@{^{(}-->}[d] &  \Psi_{\pi_v,*,c}
\ar@{^{(}->}[d] \\
\Fil^{g_i(\varrho)}_!(\Psi_{\pi_v}) \ar@{^{(}->}[r] &  \Psi_{\pi_v}.
}$$
\end{itemize}

\rem applied to $\Psi_{\pi_v}$, the equality (\ref{eq-psi-groth}) becomes, cf. \cite{boyer-torsion} \S 3.4
\addtocounter{thm}{1}
\begin{equation} \label{eq-psi-groth5}
\bigl [ \Psi_{\pi_v,*,c} \bigr ]=\sum_{t=1}^{\lfloor \frac{d}{g} \rfloor } 
P(t,\pi_v)_c(\frac{t-1}{2}).
\end{equation}
More precisely for $\pi_v \in \cusp(\varrho,i)$, then
$\gr^{g_i(\varrho)}_!(\Psi_{\pi_v,*,c})$ has a filtration $\Fil^k(\gr^{g_i(\varrho)}_!(\Psi_{\pi_v,*,c}))$
for $0 \leq k \leq  s_i(\varrho):=\lfloor \frac{d}{g_i(\varrho)} \rfloor$ with graded pieces 
$\gr^k(\gr^{g_i(\varrho)}_!(\Psi_{\pi_v,*,c})) \simeq P(s_i(\varrho)-k+1,\pi_v)_c(\frac{s_i(\varrho)-k}{2})$.

\section{Non-degeneracy property for submodules}
\label{para-lattice}

Recall first that, for a fixed irreducible $\overline \Fm_l$-representation $\varrho$ of
$D_{v,d}^\times$, the notation $\VC^{d-1}_{\varrho,\Nm}$ (resp. 
$\UC^{d-1}_{\varrho,\Nm}$) designates the direct factor
of $\VC^{d-1}_{F_v,\overline \Zm_l,d}$ (resp. the free quotient $\UC^{d-1}_{F_v,\overline \Zm_l,d,free}$)
associated with $\varrho$ in the sense of
\cite{dat-torsion} \S B.2. 
Let $i_z: z \hookrightarrow X^{=d}_{\IC,\bar s}$,
be any supersingular point. Then, from the main theorem of
Berkovitch in \cite{berk2}, we have an isomorphism
\addtocounter{thm}{1}
\begin{equation} \label{eq-berk}
\ind_{(D_{v,d}^\times)^0 \varpi_v^\Zm}^{D_{v,d}^\times} \lexp p \hi^0 i_z^!
\Psi_{\varrho} \simeq \VC^{d-1}_{\varrho,\Nm},
\end{equation}
and respectively
\addtocounter{thm}{1}
\begin{equation} \label{eq-berk2}
\bigl ( \ind_{(D_{v,d}^\times)^0 \varpi_v^\Zm}^{D_{v,d}^\times} \lexp p \hi^0 i_z^*
\Psi_{\varrho} \bigr )_{free} \simeq \UC^{d-1}_{\varrho,\Nm},
\end{equation}
which are equivariant for $D_{v,d}^\times \times GL_d(F_v) \times W_{F_v}$,
so that we are led to compute $\lexp p \hi^0 i_z^! \Psi_{\varrho}$ (resp. the free quotient of
$\lexp p \hi^0 i_z^* \Psi_{\varrho}$).

\rem $\lexp {p+} \hi^0 i_z^! \Psi_\varrho$ might have torsion\footnote{The main theorem of
\cite{boyer-duke} tells that this is not the case.} but by definition 
$\lexp {p} \hi^0 i_z^! \Psi_\varrho$ is necessarily free.

Recall that our strategy is to construct a filtration of $\VC^{d-1}_{\varrho,\Nm}$
(resp. $\UC^{d-1}_{\varrho,\Nm}$)
with irreducible graded pieces as a 
$\overline \Zm_l$-representation of $GL_d(F_v) \times D_{F_v,d}^\times \times W_{F_v}$, which means
that they are free and irreducible after tensoring with $\overline \Qm_l$.
Of course the idea is to obtain such a filtration from a filtration $\Fil^\bullet(\Psi_\varrho)$
of $\Psi_\varrho$, constructed using the Newton stratification, so that
the associated spectral sequences 
$$E_{!,1}^{r,s}:=\lexp p \hi^{r+s} i_z^! \gr^{-r} (\Psi_\varrho) \Rightarrow \lexp p \hi^{r+s} i_z^! \Psi_\varrho$$
and
$$E_{*,1}^{r,s}:=\lexp p \hi^{r+s} i_z^* \gr^{-r} (\Psi_\varrho) \Rightarrow \lexp p \hi^{r+s} i_z^* \Psi_\varrho,$$
give us the expected filtrations of
$\hi^0 i_z^! \Psi_{\varrho}$ and $\lexp p \hi^0 i_z^* \Psi_{\varrho}$,
where $\gr^\bullet(\Psi_\varrho)$ are the graded pieces of $\Fil^\bullet(\Psi_\varrho)$.

\rem as long as we are only interested in $\VC^{d-1}_{\varrho,\Nm}$ (resp. $\UC^{d-1}_{\varrho,\Nm}$)
and not with the others 
$\VC^{\bullet}_{\varrho,\Nm}$ and $\UC^{\bullet}_{\varrho,\Nm}$,
we have in fact only to bother with the perverse sheaves concentrated on the supersingular locus.
More precisely note that for a perverse sheaf $\PC$ not concentrated in the supersingular locus,
then $\lexp p \hi^\delta i_z^! \PC$ (resp. $\lexp p \hi^{-\delta} i_z^* \PC$)
is zero for $\delta \leq 0$ (resp. zero for $\delta <0$ and torsion for $\delta=0$). 

Meanwhile
some of the $E_{!,1}^{r,s}$ (resp. $E_{*,1}^{r,s}$) for $r+s=1$ (resp. $r+s=0$) might be torsion so that
to control the lattices, it is better if all the perverse sheaves concentrated on the supersingular locus
appears before (resp. after) the others, see the fourth step in the next section. 

\rem of course as long as you
are only concerned with perverse sheaves on the supersingular locus, you do not need to bother about
the positions of these perverse sheaves but only about their lattices, cf. \S \ref{para-terminology}.  At the end 
the non-degeneracy persistence property will be deduced from proposition \ref{prop-zele}.

\subsection{The case of $\VC^{d-1}_{\varrho,\Nm}$}
\label{para-VC}

The main goal is then to construct a filtration of $\Psi_\varrho$.

\medskip

\noindent \textit{First step}: we start with the following three pieces filtration.
For every $c \in GL_d(F_v)/P_{1,d}(F_v)$, note that any supersingular point
belongs to the pure stratum $X^{\geq 1}_{\IC,\bar s,c}$, so that in the
short exact sequence (\ref{eq-sec2-psi})
\addtocounter{thm}{1}
\begin{equation} \label{eq-sat1}
0 \rightarrow \Psi_{\varrho,*,c} \longrightarrow \Psi_{\varrho} \longrightarrow
j_{\neq c,*} j^{*}_{\neq c} \Psi_{\varrho} \rightarrow 0,
\end{equation}
we have, with harmless abuse of notations, 
$\hi^0 i_z^! \Psi_{\varrho} \simeq \hi^0 i_z^! \Psi_{\varrho,*,c}$.
Consider then another pure stratum $X^{\geq 1}_{\IC,\bar s,c'}$ with $c' \neq c$.

\begin{lemma} \label{lem-icpsi2}
The perverse sheaf $\lexp p \hi^i i^{1,*}_{c'} \Psi_{\varrho,*,c}$ is zero for $i \neq 0$ 
and it is free for $i=0$.
\end{lemma}

\begin{proof}
Note first, cf. lemma \ref{lem-icpsi}, that the result is true for $\Psi_\varrho$. Moreover for any
perverse free sheaf $P$, we have $\lexp p \hi^i i^*_{c'} P=0$ if $i \not \in \{ 0,-1 \}$ and 
it is free for $i=-1$. The result then follows easily from the long exact sequence associated
to the previous short exact sequence when we apply $i^*_{c'}$.
\end{proof}

In particular in the following short exact sequence
$$0 \rightarrow j_{\neq c',!} j^{*}_{\neq c'} \Psi_{\varrho,*,c} \longrightarrow \Psi_\varrho
\longrightarrow \Psi_{\varrho,*,c,!,c'} \rightarrow 0$$ 
the perverse sheaf $\Psi_{\varrho,*,c,!,c'}$ is free. Moreover as the cokernel of
$j_{\neq c',!} j^{*}_{\neq c'} \Psi_{\varrho,*,c} \hookrightarrow \Psi_{\varrho,*,c}$
is $i^1_{c',*} \lexp p \hi^0 i^{1,*}_{c'} \Psi_{\varrho,*,c} $, we have the following short exact sequence
$$0 \rightarrow i^1_{c',*} \lexp p \hi^0 i^{1,*}_{c'} \Psi_{\varrho,*,c} \longrightarrow 
\Psi_{\varrho,*,c,!,c'} \longrightarrow j_{\neq c,*} j^{*}_{\neq c} \Psi_{\varrho} \rightarrow 0.$$

\rem with the terminology of \S \ref{para-terminology}, the dual version of lemma \ref{lem-icpsi}
tells us that in the short exact sequence (\ref{eq-sat1}), there is no saturation process: in fact all the questions
about saturation keep inside the left and right terms of this short exact sequence, cf. \cite{boyer-duke}
for more details. In the same way,
the previous lemma tells us that there is no saturation process in considering the adjunction morphism
$j_{\neq c',!} j^{*}_{\neq c'} \Psi_{\varrho,*,c} \longrightarrow \Psi_{\varrho,*,c}$.

\medskip

\noindent \textit{Second step}: we want to refine the filtration of the first step in order to compute 
$\lexp p \hi^0 i_z^! \Psi_{\varrho}$ for any supersingular point $z$. Note first that, as
$z$ belongs to any $X^{\geq 1}_{\IC,\bar s,c'}$, then
$\lexp p \hi^0 i_z^!  j_{\neq c,*} j^{*}_{\neq c} \Psi_{\varrho} =(0)$ so in fact we just need to
bother about the first two graded pieces of the previous filtration, which corresponds to a filtration of
$\Psi_{\varrho,*,c}$.
With the terminology of \S \ref{para-terminology}, at this step we do not want to
deal with lattices and positions, but just describe the irreducible sub-quotients over $\overline \Qm_l$.

Consider the first graded piece $j_{\neq c',!} j^{*}_{\neq c'} \Psi_{\varrho,*,c}$.
As $j_{\neq c'}$ is affine, $j_{\neq c',!} j_{\neq c'}^*$ is an exact functor, so that from the
previous section, $j_{\neq c',!} j^{*}_{\neq c'} \Psi_{\varrho,*,c}$ has a filtration
$\Fil^\bullet \bigl ( j_{\neq c',!} j^{*}_{\neq c'} \Psi_{\varrho,*,c} \bigr )$
with graded pieces $j_{\neq c',!} j^{*}_{\neq c'}  \gr_!^{g_i(\varrho)} \Psi_{\varrho,*,c}$.
\footnote{Again here there is no need of saturation.}

We then have
$$j_{\neq c',!} j^{*}_{\neq c'} \gr^{g_i(\varrho)}_!(\Psi_{\varrho,*,c}) \otimes_{\overline \Zm_l} 
\overline \Qm_l \simeq \bigoplus_{\pi_v \in \cusp(\varrho,i)} 
j_{\neq c',!} j^{*}_{\neq c'}\gr^{g_i(\varrho)}_!(\Psi_{\pi_v,*,c}),$$
where $j_{\neq c',!} j^{*}_{\neq c'} \gr^{g_i(\varrho)}_!(\Psi_{\pi_v,*,c})$ has a filtration 
whose graded pieces, by combining (\ref{eq-psi-groth5}) and lemma \ref{lem-j-c}, are, in the 
order of appearance from the socle to the top, 
\begin{itemize}
\item for $t=s_i(\varrho), \cdots, 2$, the $P(t,\pi_v)_c(\frac{t-1}{2})$ obtained through the
following short exact sequence twisted by $(\frac{t-1}{2})$
$$0 \rightarrow P(t,\pi_v)_{c,\neq c'} \longrightarrow 
P(t,\pi_v)_c
\longrightarrow P(t,\pi_v)_{\langle c,c' \rangle } \rightarrow 0,$$
where
$$P(t,\pi_v)_{c,\neq c'}:=
\lexp p  j^{=1}_{\neq c',!*}j^{=1,*}_{\neq c'} P_{\overline \Qm_l,c}(t,\pi_v),$$

\item and with last quotient $P(1,\pi_v)_{c,\neq c'}$.
\end{itemize}
The second graded piece $ i^1_{c',*} \lexp p \hi^0 i^{1,*}_{c'} \Psi_{\varrho,*,c}$
of the filtration of the first step, is
a free perverse sheaf which, over $\overline \Qm_l$, have then irreducible constituents 
$P(1,\pi_v)_{\langle c , c' \rangle}$ for $\pi_v \in \cusp(\varrho)$ an irreducible 
representation of $GL_{g}(F_v)$ with $g<d$.

\medskip

\noindent \textit{Third step}: 
We now want, using the terminology of \S \ref{para-terminology}, 
to understand lattices and positions of the perverse sheaves of the second step. Our aim
is to give a filtration over $\overline \Zm_l$ from which we will be able to compute
$\lexp p \hi^0 i_z^! \Psi_{\varrho}$ thanks to a spectral sequence as usual.

(a) Start first with $ i^1_{c',*} \lexp p \hi^0 i^{1,*}_{c'} \Psi_{\varrho,*,c}$. For $\pi_v$ an irreducible
cuspidal representation of $GL_g(F_v)$ and $1 \leq t \leq s:=\lfloor d/g \rfloor$, note that, if $tg<d$, then
$\lexp p \hi^0 i_z^! P(t,\pi_v)=(0)$. In particular if
$\lexp p \hi^0 i_z^! P(1,\pi_v)_{\langle c , c' \rangle} \neq (0)$ then $g=d$ so that
\begin{itemize}
\item trivially $P(1,\pi_v)_{\langle c , c' \rangle}  = P(1,\pi_v)$,

\item concerning the action of
$GL_d(F_v)$, as the modulo $l$ reduction of $\pi_v$ is still irreducible,
all the stable lattices are homothetic and 

\item there is only one position possible for the intermediate extension
as the perverse sheaf is concentrated on a zero dimensional sub-scheme. 
\end{itemize}
In short, we do not have to bother with
$ i^1_{c',*} \lexp p \hi^0 i^{1,*}_{c'} \Psi_{\varrho,*,c}$.

(b)
We focus then on $\lexp p \hi^0 i_z^! \bigl ( j_{\neq c',!} j^{*}_{\neq c'} \Psi_{\varrho,*,c} \bigr )$
by refining the previous filtration
$\Fil^\bullet \bigl ( j_{\neq c',!} j^{*}_{\neq c'} \Psi_{\varrho,*,c} \bigr )$ with graded pieces
$j_{\neq c',!} j^{*}_{\neq c'}  \gr_!^{g_i(\varrho)} \Psi_{\varrho,*,c}$.
Recall that
$$j^{=g_i(\varrho),*} j_{\neq c',!} j^{*}_{\neq c'}  \gr_!^{g_i(\varrho)} \Psi_{\varrho,*,c}
\otimes_{\overline \Zm_l} \overline \Qm_l\simeq \bigoplus_{\pi_v \in \cusp(\varrho,i)}
HT_{c,\neq c'}(\pi_v,\pi_v),$$
and by fixing any numbering of $\cusp(\varrho,i)=\{ \pi_{v,1},\cdots \}$ the pull-back
$$\xymatrix{
\Fil^k \Bigl (j^{=g_i(\varrho),*} j_{\neq c',!} j^{*}_{\neq c'}  \gr_!^{g_i(\varrho)} \Psi_{\varrho,*,c}
\Bigr ) \ar@{^{(}-->}[r] \ar@{^{(}-->}[d] & j^{=g_i(\varrho),*} j_{\neq c',!} j^{*}_{\neq c'}  \gr_!^{g_i(\varrho)} \Psi_{\varrho,*,c}  \ar@{^{(}->}[d] \\
\bigoplus_{i=1}^k HT_{c,\neq c'}(\pi_{v,i},\pi_{v,i}) \ar@{^{(}->}[r] &
j^{=g_i(\varrho),*} j_{\neq c',!} j^{*}_{\neq c'}  \gr_!^{g_i(\varrho)} \Psi_{\varrho,*,c}
\otimes_{\overline \Zm_l} \overline \Qm_l,
}$$
define then a naive filtration of 
$j^{=g_i(\varrho),*} j_{\neq c',!} j^{*}_{\neq c'}  \gr_!^{g_i(\varrho)} \Psi_{\varrho,*,c}$
such that the graded pieces are some lattices of $HT_{c,\neq c'}(\pi_v,\pi_v)$
for $\pi_v$ describing $\cusp(\varrho,i)$. By taking the image by
$j^{=g_i(\varrho)}_!$ of this filtration, we obtain a filtration of 
$j_{\neq c',!} j^{*}_{\neq c'}  \gr_!^{g_i(\varrho)} \Psi_{\varrho,*,c}$
whose graded pieces are entire versions of 
$j_{\neq c',!} j^{*}_{\neq c'} \gr^{g_i(\varrho)}_!(\Psi_{\pi_v,*,c})$ for $\pi_v$
describing $\cusp(\varrho,i)$. Then we can filtrate each of these graded pieces to obtain
a filtration denoted $\Fill^\bullet \bigl ( j_{\neq c',!} j^{*}_{\neq c'} \Psi_{\varrho,*,c} \bigr )$ 
whose graded pieces $\grr^k \bigl ( j_{\neq c',!} j^{*}_{\neq c'} \Psi_{\varrho,*,c} \bigr )$ 
are some entire version of the $P(t,\pi_v)_c$ if $2 \leq t \leq s_i(\varrho)$
(resp. $P(1,\pi_v)_{c,\neq c'}$), for $\pi_v \in \cusp(\varrho,i)$ with $i \geq -1$: these
entire perverse sheaves may depend on all the choices. 

\rem as pointed out above, we just have to deal with the graded pieces concentrated
on the supersingular locus for which with do not have to bother about the position of
the intermediate extension, but now as $t$ might be strictly greater than one, we have
to describe the lattices. For $\pi_v \in \cusp(\varrho,i)$
and $t$ such that $(t+1)g_i(\varrho)=d$, let then denote by 
$$\PC_{\Fill,!,c}(t+1,\pi_v),$$
the lattice defined above starting from a filtration of 
$j_{\neq c',!} j^{*}_{\neq c'} \Psi_{\pi_v,*,c}$, cf. (\ref{eq-pfillc}). 
Note at this point that finally 
$\PC_{\Fill,!,c}(t+1,\pi_v)$ is given by the short exact sequence
of lemma \ref{lem-j-c} so that we can easily identify its lattice as explained in 
theorem \ref{thm-principal1}.

%

\medskip

\noindent \textit{Fourth step}: now we want to modify the previous filtration of 
$\Psi_{\varrho,*,c}$
 by reorganizing the order
of the graded pieces so that those concentrated on the supersingular locus appear first.
Let first explain why it is possible to do so. Denote by 
$$j^{1 \leq d-1}: X^{\geq 1}_{\IC,\bar s} \setminus X^{=d}_{\IC,\bar s} \hookrightarrow X^{\geq 1}_{\IC,\bar s}$$
and consider the adjunction morphism
$$\Psi_{\varrho,*,c} \longrightarrow \lexp {p+} j^{1 \leq d-1}_* j^{1 \leq d-1,*} \Psi_{\varrho,*,c},$$
where (\ref{eq-psi-groth5}) described $\Psi_{\varrho,*,c}$, at least over $\overline \Qm_l$.
Then the kernel $K_{\varrho,*,c,d}$ of this morphism, is by construction free and in the Grothendieck group we have
$$\Bigl [ K_{\varrho,*,c,d} \otimes_{\overline \Zm_l} \overline \Qm_l \Bigr ]=
\sum_{i \geq -1}
\sum_{\atop{\pi_v \in \cusp(\varrho,i)}{d=t_ig_i(\varrho)}} \bigl [ P(t_i,\pi_v)(\frac{t_i-1}{2}) \bigr ],$$
i.e. $ K_{\varrho,*,c,d}$ gathers all the irreducible constituents of $\Psi_{\varrho,*,c}$ concentrated on
the supersingular locus. To see this, it suffices to argue on $\Psi_{\pi_v}$  for $\pi_v \in \cusp(\varrho,i)$ with $K_{\pi_v,*,c,d}$ the kernel of the adjunction morphism
$\Psi_{\pi_v,*,c} \longrightarrow \lexp {p+} j^{1 \leq d-1}_* j^{1 \leq d-1,*} \Psi_{\pi_v,*,c}.$ We then notice that
\begin{itemize}
\item for a pure perverse sheaf $P$
of weight $0$ then the irreducible constituents of $\lexp p j^{1 \leq d-1}_{*} j^{1 \leq d-1,*} P$ are of non-negative weight;

\item and the irreducible constituents of $j^{1 \leq d-1,*} \Psi_{\pi_v,*,c}$ are the $P(t,\pi_v)(\frac{t-1}{2})$ 
with $t\leq s_i:=\lfloor \frac{d}{g_i(\varrho)} \rfloor$ (resp. $t < s_i$) if $g_i(\varrho)$ do not divide $d$
(resp. if $s_ig_i(\varrho)=d$). 
\end{itemize}
Then when $g_i(\varrho)$ divide $d$,  all the irreducible constituents
are of weight strictly greater than those of $P(s_i,\pi_v)(\frac{s_i-1}{2})$ so that
$K_{\pi_v,*,c,d}=P(s_i,\pi_v)(\frac{s_i-1}{2})$.
Otherwise $K_{\pi_v,*,c,d}$ is trivial.

\medskip

We start then from the previous filtration which is $P_{c,c'}(F_v)$-equivariant
and where the order of its graded pieces verifies the following property: 
\begin{itemize}
\item $P(1,\pi_v)$ for $\pi_v \in \cusp(\varrho,i)$ with $g_i(\varrho)=d$ appears after, 
i.e. in higher graded pieces than those associated with either $P(t,\pi'_v)_{c,\neq c'}$ 
or $P(t,\pi'_v)_{\langle c,c' \rangle}$ for $\pi'_v \in \cusp(\varrho,i')$ 
where $tg_i(\varrho)<d$ or $t>1$;

\item let $k$ be the index of graded piece associated with $P(t,\pi_v)$ with $t>1$,
$\pi_v \in \cusp(\varrho,i)$ with $tg_i(\varrho)=d$. Then for any graduate piece
of index $k'<k$ associated with some $P(t',\pi'_v)_{c,\neq c'}$ or 
$P(t',\pi'_v)_{\langle c,c' \rangle}$ with $\pi'_v \in \cusp(\varrho,i')$ and
$t'g_{i'}(\varrho)<d$, then $i'<i$.
\end{itemize}
Note also that, for any $\pi_v \in \cusp(\varrho,i)$ and $t$ such that $tg_i(\varrho)<d$, then
if $k_{\neq c'}>k_{c'}$ are the indexes of the graded pieces associated with
respectively $P(t',\pi'_v)_{c,\neq c'}$ and $P(t',\pi'_v)_{\langle c,c' \rangle}$, then
for any $k_{c'} < k < k_{\neq c'}$, the associated graded piece is never concentrated
in the supersingular locus. This implies that we can modify the filtration such that 
\begin{itemize}
\item the graded pieces are of the shape $P(t,\pi_v)_c$ without modifying the lattices of the
graded pieces concentrated in the supersingular locus.

\item Now the filtration is equivariant for the action of $P_c(F_v)$ as all the graduate and the whole of the perverse sheaf,
are $P_c(F_v)$-equivariant.

\end{itemize}

\begin{lemma}
Consider a $P_{c}(F_v)$-equivariant perverse sheaf $X$ which can be written
$$0 \rightarrow A_1 \longrightarrow X \longrightarrow A_2 \rightarrow 0$$
where 
\begin{itemize}
\item $A_2$ is a free perverse sheaf with $A_2 \otimes_{\overline \Zm_l} \overline \Qm_l=P(t,\pi_v)$ for $\pi_v \in \cusp(\varrho,i)$ and $tg_i(\varrho)=d$,

\item and $A_1$ is some free perverse sheaf with
$A_1 \otimes_{\overline \Zm_l} \overline \Qm_l$
isomorphic to $P(t',\pi_v')_{c}$
with $\pi'_v \in \cusp(\varrho,i')$, $h=t'g_{i'}(\varrho)$ and $i'<i$.
\end{itemize}
Suppose moreover that 
$$X \otimes_{\overline \Zm_l} \overline \Qm_l \simeq (A_1 \otimes_{\overline \Zm_l} 
\overline \Qm_l)  \oplus (A_2 \otimes_{\overline \Zm_l} \overline \Qm_l),$$
then $X \simeq A_1 \oplus A_2$.
\end{lemma}

\begin{proof}
We have a diagram like (\ref{eq-prop-extension}) where $T$ is supported on
$X^{=d}_{\IC,\bar s}$ so that $A_1 \hookrightarrow A'_1 \twoheadrightarrow T$
is obtained through 
$$\lexp p j^{=h}_{c,!*} j^{=h,*}_c A_1 \htarrow_+ A_1 \htarrow_+ A'_1 \htarrow_+ \lexp {p+} 
j^{=h}_{c,!*} j^{=h,*}_c A_1.$$
Suppose by absurdity, that $T \neq (0)$. 
\begin{itemize}
\item Then as a quotient of $A'_1$ and as a representation of $P_{c}(F_v)$
$T \otimes_{\overline \Zm_l} \overline \Fm_l$ is isomorphic to a small mirabolic induced
representation $(r_l(\st_{t'}(\pi'_v))_{|P_{c}(F_v)} \times \tau$ for some
$\overline \Fm_l$-representation $\tau$ of $GL_{d-t'g_{i'}(\varrho)}(F_v)$, and where
$r_l$ designates the modulo $l$ reduction functor. In particular,
proposition \ref{prop-zele} imposes that $T \otimes_{\overline \Zm_l} \overline \Fm_l$ must
have an irreducible sub-quotient with derivative of order $g_{i'}(\varrho)$.

\item On the other side, note that as a quotient of $A_2$ all of its derivative have order 
$\geq g_i(\varrho)>g_{i'}(\varrho)$.
\end{itemize}
So $T$ must be trivial, i.e. $X=A_1 \oplus A_2$.
\end{proof}

We now sum up what we have done until now.

\begin{propo} \label{prop-fil-final1}
There exists a filtration
$$(0)=\Fil^0(\Psi_{\varrho,*,c}) \subset \Fil^1(\Psi_{\varrho,*,c}) \subset
\Fil^2(\Psi_{\varrho,*,c})=\Psi_{\varrho,*,c}$$ 
such that 
\begin{itemize}
\item the irreducible constituents of $\gr^i(\Psi_{\varrho,*,c}) \otimes_{\overline \Zm_l}
\overline \Qm_l$ for $i=1$ (resp. $i=2$) are all with support in $X^{=d}_{\IC,\bar s}$
(resp. are of the form $P(t,\pi_v)_c$ with $\pi_v \in \cusp(\varrho,i)$ and
$tg_i(\varrho)<d$).

\item Moreover there is a filtration of 
\begin{multline*}
(0)=\Fil^{-2}( \gr^1(\Psi_{\varrho,*,c})) \subset \Fil^{-1}( \gr^1(\Psi_{\varrho,*,c})) 
\subset \cdots \\ \cdots
\subset \Fil^s( \gr^1(\Psi_{\varrho,*,c}))=\gr^1(\Psi_{\varrho,*,c})
\end{multline*}
whose graded pieces $\gr^i ( \gr^1(\Psi_{\varrho,*,c}))$ are zero except if there exists
$t$ such that $tg_i(\varrho)=d$ in which case with the previous notations,
$\gr^i ( \gr^1(\Psi_{\varrho,*,c}))$ admits a naive filtration, cf. point (b) in the third step,
indexed by $\pi_v \in \cusp(\varrho,i)$ such that its graded pieces are, cf. (\ref{eq-pfillc}), 
the $\PC_{\Fill,!,c}(t,\pi_v)$.
\end{itemize}
\end{propo}

\rem we will next identify the $\PC_{\Fill,!,c}(t,\pi_v)$ and see that, at least for
the action of $P_c(F_v)$, they are independent of the choice of the naive filtration, so that
we will simply write 
\addtocounter{thm}{1}
\begin{equation} \label{eq-nota-approx}
\gr^i ( \gr^1(\Psi_{\varrho,*,c})) \approx \bigoplus_{\pi_v \in \cusp(\varrho,i)}
\PC_{\Fill,!,c}(t,\pi_v),
\end{equation}
in place of the long statement in the second point of the proposition above.

\medskip

\noindent \textit{Last step}: as explained in the introduction of \S \ref{para-lattice},
for a perverse sheaf $P(t,\pi_v)$ not concentrated in the supersingular locus, we have
$\lexp p \hi^0 i_z^! P(t,\pi_v)=(0)$. 
In the long exact sequence associated with the $\lexp p \hi^\bullet i_z^!$
applied to the short exact sequence
$$0 \rightarrow \gr^1(\Psi_{\varrho,*,c}) \longrightarrow \Fil^2(\Psi_{\varrho,*,c})
\longrightarrow \gr^2(\Psi_{\varrho,*,c}) \rightarrow 0,$$
we then have $\lexp p \hi^0 i_z^! \Psi_{\varrho,*,c} \simeq \lexp p \hi^0 i_z^! 
\gr^1(\Psi_{\varrho,*,c})$.
From the previous proposition we obtain a filtration
\begin{multline*}
(0)=\Fil^{-2}( \lexp p \hi^0 i_z^! (\Psi_{\varrho})) \subset 
\Fil^{-1}( \lexp p \hi^0 i_z^! (\Psi_{\varrho})) \subset \cdots \\ \cdots \subset
\Fil^s( \lexp p \hi^0 i_z^! (\Psi_{\varrho}))=\lexp p \hi^0 i_z^! (\Psi_{\varrho})
\end{multline*}
whose non zero graded pieces coincide with the indexes $i\geq -1$ such that there exists $t$ with 
$tg_i(\varrho)=d$ and then with the notation of (\ref{eq-nota-approx})
$$\gr^i ( \lexp p \hi^0 i_z^! (\Psi_{\varrho})) \approx \bigoplus_{\pi_v \in \cusp(\varrho,i)} 
\PC_{\Fill,!,c}(t,\pi_v)_z.$$

We will now simply denote by $\bigl ( \Fill^k(\Psi_{\varrho,*,c}) \bigr )_{0 \leq k \leq r}$ 
the filtration of $\gr^1(\Psi_{\varrho,*,c})$ obtained above such that its graded pieces 
$\grr^k(\Psi_{\varrho,*,c})$ are irreducible after tensoring with $\overline \Qm_l$. 
We then also denote by 
$$\Fill^k(\VC^{d-1}_{\varrho,\Nm} ):= \ind_{(D_{v,d}^\times)^0 \varpi_v^\Zm}^{D_{v,d}^\times} 
\lexp p \hi^0 i_z^! \Fill^k(\Psi_{\varrho,*,c}).$$

\begin{thm} \label{thm-principal1}
As a $\overline \Zm_l[P_d(F_v) \times D_{v,d}^\times \times W_{F_v}]$-module, 
the successive graded pieces $\gr^k (\VC^{d-1}_{\varrho,\Nm})$ are such that
there exists $i$, $\pi_v \in \cusp(\varrho,i)$ and $t$ such that $tg_i(\varrho)=d$ with
$$\gr^k (\VC^{d-1}_{\varrho,\Nm}) \simeq \Gamma_{GDW}(\pi_v),$$ 
with $\Gamma_{GDW}(\pi_v) \simeq \Gamma_G(\pi_v) \otimes 
\Gamma_D(\pi_v) \otimes \Gamma_W(\pi_v)$ where
\begin{itemize}
\item $\Gamma_D(\pi_v)$ (resp. $\Gamma_W(\pi_v)$) is a stable lattice of $\pi_v[t]_D$ (resp. 
$\Lm_{g_i(\varrho)}(\pi_v)$);

\item $\Gamma_G(\pi_v)$ is isomorphic to the stable lattice 
$\bigl ( RI_{\overline \Zm_l,-}(\pi_v,t) \bigr )_{|P_d(F_v)}$ of 
definition \ref{defi-RI}.
\end{itemize}
\end{thm}

\rem in other words, the lattice $\PC_{\Fill,!,c}(t,\pi_v)$, which is a sheaf on the
supersingular locus, is with the previous notations, fiber by fiber isomorphic to 
$\Gamma_G(\pi_v) \otimes \Gamma_D(\pi_v) \otimes \Gamma_W(\pi_v)$. 

\begin{proof}
We argue by induction on $d$. As the result is trivial for $g_{-1}(\varrho)$ because, 
as the modulo $l$ reduction is irreducible, there
is, up to isomorphism, only one stable lattice, we suppose the result true for all $h<d$.
We then use the statement of the theorem through the isomorphism (\ref{eq-berk}), 
to obtain informations on the lattices of our perverse sheaves not concentrated on the
supersingular locus. Thus arguing like before on $j^{1 \leq h,*} \Psi_{\varrho}$,
we can conclude that, for any $i\geq -1$, $\pi_v \in \cusp(\varrho,i)$ and $tg_i(\varrho)<d$,  
the lattices of $HT_c(\pi_v,\st_t(\pi_v))(\frac{t-1}{2})$
of the graded pieces $\gr^k ( j_{\neq c',!} j^*_{\neq c'} \Psi_{\varrho,*,c})$ are of the shape
\addtocounter{thm}{1}
\begin{equation} \label{eq-shape}
\LC_D \otimes \bigl ( RI_{\overline \Zm_l,-}(\pi_v,t) \bigr )_{|P_{tg_i(\varrho)}(F_v)} 
\otimes \Gamma_W,
\end{equation}
where $\LC_D$ is some stable $\overline \Zm_l$-lattice sheaf of $\LC(\pi_v[t]_D)_c$.

To prove the theorem, by the isomorphism (\ref{eq-berk}),
we now have to show that the lattice $\PC_{\Fill,!,c}(t,\pi_v)$ is a tensorial product
of lattices equipped with action by $P_d(F_v)$, $D_{v,d}^\times$ and $W_{F_v}$,
and the lattice for the action of $P_d(F_v)$ is isomorphic to
$\bigl ( RI_{\overline \Zm_l,-}(\pi_v,t) \bigr )_{|P_d(F_v)}$.

By hypothesis we have $\pi_v \in \cusp(\varrho,i)$ and $tg_i(\varrho)=d$. Recall also, 
using the exactness of $j_{c,\neq c',!}$, that $\PC_{\Fill,!,c}(t,\pi_v)$ fits in the following
short exact sequence of lemma \ref{lem-j-c}
\addtocounter{thm}{1}
\begin{multline} \label{eq-pfillc}
0 \rightarrow \PC_{\Fill,!,c}(t,\pi_v) \longrightarrow j_{\neq c',!} j^*_{\neq c'} 
\gr^k  ( j_{\neq c',!} j^*_{\neq c'} \Psi_{\varrho,*,c}) \\ \longrightarrow 
\lexp p j_{\neq c',!*} j^*_{\neq c'} \gr^k  ( j_{\neq c',!} j^*_{\neq c'} \Psi_{\varrho,*,c}) \rightarrow 0
\end{multline}
where 
$$\gr^k  ( j_{\neq c',!} j^*_{\neq c'} \Psi_{\varrho,*,c}) \otimes_{\overline \Zm_l}
\overline \Qm_l \simeq P(t-1,\pi_v)_c(\frac{t-2}{2}) \hookrightarrow P(t-1,\pi_v)(\frac{t-2}{2}).
$$
Over $\overline \Zm_l$, we have seen that
$j^{=(t-1)g_i(\varrho),*}\gr^k  ( j_{\neq c',!} j^*_{\neq c'} \Psi_{\varrho,*,c})$ is a tensorial
product of lattices where those relatively to the action of $P_{(t-1)g_i(\varrho)}(F_v)$
is $\bigl ( RI_{\overline \Zm_l,-}(\pi_v,t-1) \bigr )_{|P_{(t-1)g_i(\varrho)}(F_v)}$. 
Moreover, as the supersingular locus belongs to $X^{\geq 1}_{\IC,\bar s,c'}$, 
relatively to the supersingular locus,
$\gr^k  ( j_{\neq c',!} j^*_{\neq c'} \Psi_{\varrho,*,c})$ 
behaves like a $p$-intermediate extension, i.e. 
$$\hi^0 i_z^* \gr^k  ( j_{\neq c',!} j^*_{\neq c'} \Psi_{\varrho,*,c})=(0),$$
 for all geometric supersingular point $z$.
For such a $p$-intermediate extension we materialize this property by
writing $p(ss)$ as a left exponent. By inducing we can then write
$$\gr^k  ( j_{\neq c',!} j^*_{\neq c'} \Psi_{\varrho,*,c}) \hookrightarrow 
\lexp {p(ss)} \PC_{RI,\otimes}(t-1,\pi_v)(\frac{t-2}{2})$$
where $\lexp {p(ss)} \PC_{RI,\otimes}(t-1,\pi_v)$ is a lattice of
$P_{\overline \Qm_l}(t-1,\pi_v)$ verifying the following two properties:
\begin{itemize}
\item for all geometric supersingular point $z$, 
the perverse sheaf behaves like a $p$-intermediate extension, i.e.
$\hi^0 i_z^* \left ( \lexp {p(ss)} \PC_{RI,\otimes}(t-1,\pi_v) \right )=(0)$;

\item $j^{=(t-1)g_i(\varrho),*} \bigl ( \lexp {p(ss)} \PC_{RI,\otimes}(t-1,\pi_v)\bigr )$ is a tensorial
stable lattice where the lattice associated with the action of $P_{(t-1)g_i(\varrho)}(F_v)$ is 
isomorphic to $\bigl ( RI_{\overline \Zm_l,-}(\pi_v,t-1) \bigr )_{|P_{(t-1)g_i(\varrho)}(F_v)}$.
\end{itemize}

\rem we prefer to introduce the full of $\lexp {p(ss)} \PC_{RI,\otimes}(t-1,\pi_v) $
instead of $j^{\geq 1,*}_c \lexp {p(ss)} \PC_{RI,\otimes}(t-1,\pi_v) $
which is then isomorphic to $\gr^k  ( j_{\neq c',!} j^*_{\neq c'} \Psi_{\varrho,*,c})$,
see the next diagram.

We then have
$$\xymatrix{
\PC_{\Fill,!,c}(t,\pi_v) \ar@{^{(}->}[r] \ar@{-->}[d]^\simeq &  j_{\neq c',!} j^*_{\neq c'} 
\gr^k  ( j_{\neq c',!} j^*_{\neq c'} \Psi_{\varrho,*,c}) \ar@{^{(}->}[d] 
\\
\lexp p \hi^{-1} i_{c'}^* \lexp {p(ss)} \PC_{RI,\otimes}(t-1,\pi_v)(\frac{t-2}{2}) \ar@{^{(}->}[r] &  j_{\neq c',!} j^*_{\neq c'} 
\lexp {p(ss)} \PC_{RI,\otimes}(t-1,\pi_v)(\frac{t-2}{2}).
}$$
By (\ref{eq-pfillc}) $\PC_{\Fill,!,c}(t,\pi_v) \simeq \lexp p \hi^{-1} i_{c'}^*
\gr^k  ( j_{\neq c',!} j^*_{\neq c'} \Psi_{\varrho,*,c})$ 
and as for a Harris-Taylor perverse sheaf $P$ supported on the supersingular locus,
we have $P=P_c=P_{\langle c,c' \rangle}$, then the left map of the diagram
is an isomorphism over $\overline \Qm_l$. Moreover as each of the other
maps of this diagram are strict so is the dotted one which is then an isomorphism.

By lemma \ref{lem-important} and (\ref{eq-shape}),
the lattice relatively to the action of $P_d(F_v)$ on
$\lexp p \hi^{-1} i_{c'}^* \lexp {p(ss)} \PC_{RI,\otimes}(t-1,\pi_v)(\frac{t-2}{2})$ is given by the 
induced representation
$$\xymatrix{ 
 RI_{\overline \Zm_l,-}(\pi_v \{ \frac{-1}{2} \},t-1) \times 
 (\pi_v\{ \frac{t-1}{2} \} )_{|P_{g_i(\varrho)}(F_v)}
\ar@{^{(}->}[r] \ar[dr]^\sim & \bigl ( RI_{\overline \Zm_l,-}(\pi_v\{\frac{-1}{2} \},t-1)
\times \pi_v\{\frac{t-1}{2} \}  \bigr )_{|P_d(F_v)} \ar@{->>}[d] \\
& RI_{\overline \Zm_l,-}(\pi_v,t)_{|P_d(F_v)},
}$$
which finishes the proof.
\end{proof}

From proposition \ref{prop-defi-Vk}, we obtain the expected
non-degeneracy property.

\begin{corol} \label{coro-principal1}
Any irreducible $P_d(F_v)$-equivariant subspace
of $\VC^{d-1}_{\varrho,\Nm} \otimes_{\overline \Zm_l} \overline \Fm_l$, 
is non-degenerate and so isomorphic to $\tau_{nd}$.
\end{corol}

\subsection{The case of $\UC^{d-1}_{\varrho,\Nm}$}

In \cite{boyer-duke}, we prove that for any supercuspidal $\overline \Fm_l$-representation
$\varrho$, then $\UC^{d-1}_{\varrho,\Nm}$ is free. As at this stage we do not want to
use \cite{boyer-duke}, we introduce its free quotient $\UC^{d-1}_{\varrho,\Nm,free}$.
We then follow exactly the same steps than in the previous section, but dually.
Precisely fix a supersingular point $z$ and denote as before 
$i_z: z \hookrightarrow X^{=d}_{\IC,\bar s}$.
From the $D_{v,d}^\times \times GL_d(F_v) \times W_{F_v}$-equivariant isomorphism
\addtocounter{thm}{1}
\begin{equation} \label{eq-berk3}
\ind_{(D_{v,d}^\times)^0 \varpi_v^\Zm}^{D_{v,d}^\times} \lexp p \hi^0 i_z^*
\Psi_{\varrho} \simeq \UC^{d-1}_{\varrho,\Nm},
\end{equation}
we are led to compute the free quotient of $\lexp p \hi^0 i_z^* \Psi_\varrho$. As explained
in the introduction of \S \ref{para-lattice}, we first construct a filtration of $\Psi_\varrho$.
Dually to the previous section consider first the
short exact sequences
$$0 \rightarrow j_{\neq c,!} j^*_{\neq c} \Psi_\varrho \longrightarrow \Psi_\varrho
\longrightarrow \Psi_{\varrho,!,c} \rightarrow 0,$$
and 
$$0 \rightarrow i^1_{c',*} \lexp {p+} \hi^0 i_{c'}^{1,!} \Psi_{\varrho,!,c} \longrightarrow
\Psi_{\varrho,!,c} \longrightarrow j_{c,\neq c',*} j_{c,\neq c'}^* \Psi_{\varrho,!,c} \rightarrow 0.$$
One can also introduce $\Psi_{\varrho,!,c,*,c'}$ as the pull-back
$$\xymatrix{
 j_{\neq c,!} j^*_{\neq c} \Psi_\varrho \ar@{^{(}->}[r] \ar@{=}[d] &
 \Psi_{\varrho,!,c,*,c'} \ar@{-->>}[r] \ar@{^{(}-->}[d] & 
 i^1_{c',*} \lexp {p+} \hi^0 i_{c'}^{1,!} \Psi_{\varrho,!,c} \ar@{^{(}->}[d]  \\
 j_{\neq c,!} j^*_{\neq c} \Psi_\varrho \ar@{^{(}->}[r] & \Psi_\varrho \ar@{->>}[r] &
 \Psi_{\varrho,!,c}.
 }$$
 
 (a) Using the exactness of $j_{\neq c',*}j^*_{\neq c'}$, as in second step of the 
 previous section, the filtration 
$\Fil^\bullet_*(\Psi_{\varrho,!,c})$ of proposition \ref{prop-filtration1}, gives a filtration
$\Fil^\bullet (j_{\neq c',*} j^*_{\neq c'} \Psi_{\varrho,!,c})$ with graded pieces
$j_{\neq c',*} j^*_{\neq c'} \gr_*^{-g_i(\varrho)+1}(\Psi_{\varrho,!,c})$ such that
$$j_{\neq c',*} j^{*}_{\neq c'} \gr^{-g_i(\varrho)+1}_*(\Psi_{\varrho,!,c}) \otimes_{\overline \Zm_l} 
\overline \Qm_l \simeq \bigoplus_{\pi_v \in \cusp(\varrho,i)} 
j_{\neq c',*} j^{*}_{\neq c'}\gr^{-g_i(\varrho)+1}_*(\Psi_{\pi_v,!,c}),$$
where $j_{\neq c',*} j^{*}_{\neq c'} \gr^{-g_i(\varrho)+1}_!(\Psi_{\pi_v,!,c})$ has a filtration 
whose graded pieces, by lemma \ref{lem-j-c}, are, from quotient to subspaces,
\begin{itemize}
\item for $ s_i(\varrho) \geq t \geq 2$, the 
$P(t,\pi_v)_{\langle c,c' \rangle}(\frac{1-t}{2})$ and $P(t,\pi_v)_{c,\neq c'}(\frac{1-t}{2})$
allowing to reconstruct $P(t,\pi_v)_c(\frac{1-t}{2})$ by the short
exact sequence
$$0 \rightarrow P(t,\pi_v)_{\langle c,c' \rangle} \longrightarrow
P(t,\pi_v)_{c} \longrightarrow P(t,\pi_v)_{c, \neq c'} \rightarrow 0,$$

\item and $P(1,\pi_v)_{c,\neq c'}$.
\end{itemize}

(b) Concerning $i^1_{c',*} \lexp {p+} \hi^0 i_{c'}^{1,!} \Psi_{\varrho,!,c}$, after tensoring
with $\overline \Qm_l$, its irreducible sub-quotients are the 
$P(1,\pi_v)_{\langle c, c' \rangle}$ for $\pi_v \in \cusp(\varrho,i)$ with $g_i(\varrho)<d$.

By arguing like in the third and fourth steps of the previous section, 
we can manage to modify the previous filtration  to another one such that 
\begin{itemize}
\item it is equivariant for the action of $P_c(F_v)$;

\item its graded pieces are the $P(t,\pi_v)_c$ without modifying the lattices, given
by the starting filtration, of the perverse sheaves concentrated on the supersingular locus;

\item the perverse sheaves concentrated on the supersingular locus appears in the last
graded pieces;

\item if $P(t,\pi_v)(\frac{1-t}{2})$ (resp. $P(t',\pi'_v)(\frac{1-t'}{2})$) for 
$\pi_v \in \cusp(\varrho,i)$ (resp. $\pi'_v \in \cusp(\varrho,i')$) are concentrated in
the supersingular locus. If $i<i'$ then the indexes $k$ and $k'$ respectively associated
to them verify $k>k'$.
\end{itemize}

\rem As before we do not pay attention to the position of these perverse sheaves
between the $p$ and $p+$ intermediate extensions, but we merely concentrate
on the lattice of the associated local systems.

\begin{nota} \label{nota-lattice2}
For $\pi_v \in \cusp(\varrho,i)$
and $t$ such that $tg_i(\varrho)=d$, denote by $\PC_{\Fill,*,c}(t,\pi_v)$
the lattice obtained by the previous construction starting from the filtration of 
$j_{\neq c',*} j^{*}_{\neq c'} \Psi_{\pi_v,!,c}$.
\end{nota}

To sum up we state the analogous of proposition \ref{prop-fil-final1}.

\begin{propo} \label{prop-fil-final2}
There exists a filtration of 
$$(0)=\Fil^{-2}(\Psi_{\varrho,!,c}) \subset \Fil^{-1}(\Psi_{\varrho,!,c}) \subset
\Fil^0(\Psi_{\varrho,!,c})=\Psi_{\varrho,!,c}$$ 
such that 
\begin{itemize}
\item the irreducible constituents of $\gr^i(\Psi_{\varrho,!,c}) \otimes_{\overline \Zm_l}
\overline \Qm_l$ for $i=0$ (resp. $i=-1$) are all with support in $X^{=d}_{\IC,\bar s}$
(resp. are of the form $P(t,\pi_v)_c$ with $\pi_v \in \cusp(\varrho,i)$ and
$tg_i(\varrho)<d$).

\item Moreover there is a filtration of 
\begin{multline*}
(0)=\Fil^{-s-1}( \gr^0(\Psi_{\varrho,!,c})) \subset \Fil^{-s}( \gr^0(\Psi_{\varrho,!,c})) 
\subset \cdots \\ \cdots \subset
\Fil^{-1}( \gr^0(\Psi_{\varrho,!,c}))=\gr^0(\Psi_{\varrho,!,c})
\end{multline*}
whose graded pieces $\gr^{-i} ( \gr^0(\Psi_{\varrho,!,c}))$ are zero except if there exists
$t$ such that $tg_i(\varrho)=d$ in which case with  the notation of \ref{nota-lattice2} and
(\ref{eq-nota-approx}),
$$\gr^{-i} ( \gr^0(\Psi_{\varrho,!,c}))\approx \bigoplus_{\pi_v \in \cusp(\varrho,i)} 
\PC_{\Fill,*,c}(t,\pi_v).$$
\end{itemize}
\end{propo}

We denote by $\bigl ( \Fill^k(\Psi_{\varrho,!,c}) \bigr )_{0 \leq k \leq r}$ 
the filtration of $\gr^0(\Psi_{\varrho,!,c})$ obtained above such that its graded pieces 
$\grr^k(\Psi_{\varrho,!,c})$ are irreducible after tensoring with $\overline \Qm_l$. 
We also denote by 
$$\Fill^k(\UC^{d-1}_{\varrho,\Nm} ):= \lexp p \hi^0 i_z^* \Fill^k(\Psi_{\varrho,!,c}).$$
Using (\ref{eq-berk2}) and arguing by induction we obtain the 
$\UC^{d-1}_{\varrho,\Nm}$-version of theorem \ref{thm-principal1}.

\begin{thm} \label{thm-principal2}
As a $\overline \Zm_l[P_d(F_v) \times D_{v,d}^\times \times W_{F_v}]$-module, 
the successive graded pieces $\grr^k (\UC^{d-1}_{\varrho,\Nm})$
are such that there exists $i$, $\pi_v \in \cusp(\varrho,i)$ and $t$ such that 
$tg_i(\varrho)=d$ with
$$\grr^{k} (\UC^{d-1}_{\varrho,\Nm,free}) \simeq \Gamma_{GDW}(\pi_v),$$ 
with $\Gamma_{GDW}(\pi_v) \simeq \Gamma_G(\pi_v) \otimes 
\Gamma_D \otimes \Gamma_W(\pi_v)$ where
\begin{itemize}
\item $\Gamma_D$ (resp. $\Gamma_W$) is a stable lattice of $\pi_v[s_i(\varrho)]_D$ (resp. 
$\Lm_{g_i(\varrho)}(\pi_v)$);

\item $\Gamma_G$ is isomorphic to a stable $P_d(F_v)$-equivariant lattice of
$\st_{s_i(\varrho)}(\pi_v)$ such that every irreducible subspace of its modulo $l$ reduction, 
is isomorphic to $r_l(\tau_{nd})$.
\end{itemize}
\end{thm}

The only difference from the previous section concerns the lattice $\Gamma_G$
which is obtained through
$$\xymatrix{
j_{\neq c',*} j^*_{\neq c'} \lexp {p+(ss)} \PC_{RI,\otimes}(s_i(\varrho)-1,\pi_v)
(\frac{2-s_i(\varrho)}{2}) \ar@{->>}[r]  &
\lexp {p+} \hi^{1} i_{c'}^! \PC_{RI,\otimes}(s_i(\varrho)-1,\pi_v) (\frac{s_i(\varrho)-2}{2}) \\
j_{\neq c',*} j^*_{\neq c'} \gr^k(j_{\neq c',*} j^*_{\neq c'} \Psi_{\varrho,!,c}) \ar@{^{(}->}[u] 
\ar@{->>}[r] & \PC_{\Fill,*,c}(s_i(\varrho),\pi_v) \ar@{-->}[u]^\simeq
}$$
and where, by induction, $\lexp {p+} \hi^{1} i_{c'}^! \PC_{RI,\otimes}(s_i(\varrho)-1,\pi_v) 
(\frac{s_i(\varrho)-2}{2})$
is given by $\Gamma'_G \times (\pi_v\{ \frac{1-s_i(\varrho)}{2} \} )_{|P_{g_i(\varrho)}(F_v)}$
where by the induction hypothesis $\Gamma'_G$ is a 
$P_{(s_i(\varrho)-1)g_i(\varrho)}(F_v)$-equivariant
lattice of $\st_{s_i(\varrho)-1}(\pi_v \{ \frac{1}{2} \} )$ such that every subspace of its modulo $l$
reduction is isomorphic to $r_l(\tau_{nd})$.
The persistence of non-degeneracy property then follows from the exactness of
$\Phi^-$ and $\Psi^-$ and from proposition \ref{prop-zele}.

\rem It is not so easy than in the previous situation, to identify the lattice as now we only
have the following commutative diagram
$$\xymatrix{ 
& \st_{s_i(\varrho)}(\pi_v)_{|P_d(F_v)} \ar@{^{(}->}[d] \\
\Gamma'_G \times (\pi_v\{ \frac{1-s_i(\varrho)}{2} \} )_{|P_{g_i(\varrho)}(F_v)}
\ar@{^{(}->}[r] \ar@{->>}[dr] & \bigl ( \Gamma'_G
\times \pi_v\{\frac{1-s_i(\varrho)}{2} \}  \bigr )_{|P_d(F_v)} \ar@{->>}[d] \\
& LT_{\pi_v}(s_i(\varrho)-2,1)_{|P_d(F_v)}.
}$$

\begin{corol} \label{coro-principal2}
Any irreducible $P_d(F_v)$-equivariant subspace
of $\UC^{d-1}_{\varrho,\Nm,free} \otimes_{\overline \Zm_l} \overline \Fm_l$, 
is non-degenerate and so isomorphic to $\tau_{nd}$.
\end{corol}

\subsection{Other orders of cohomology groups}
\label{para-other}

As the situations of $\UC^{d-1-\delta}_{\varrho,\Nm}$ and $\VC^{d-1+\delta}_{\varrho,\Nm}$
are dual, consider for example the case of $\UC^{d-1-\delta}_{\varrho,\Nm}$ for $\delta>0$.
Remember the strategy explained in the introduction of \S \ref{para-lattice} which consists in
computing $\lexp p \hi^{-\delta} i_z^* \Psi_\varrho$ through the spectral sequence
$$E_{*,1}^{r,s}:=\lexp p \hi^{r+s} i_z^* \gr^{-r} (\Psi_\varrho) \Rightarrow \lexp p \hi^{r+s} i_z^* \Psi_\varrho,$$
associated with some filtration $\Fil^\bullet(\Psi_\varrho)$ of $\Psi_\varrho$. For $\delta >0$ we now need to 
consider all the perverse sheaves and not only those supported on the supersingular locus. In the previous sections,
arguing inductively on the Lubin-Tate spaces, we essentially understood the lattices but now the question is about
the positions of the Harris-Taylor perverses sheaves which is solved in \cite{boyer-duke}.

Start again from
$$0 \rightarrow j_{\neq c,!} j^*_{\neq c} \Psi_\varrho \longrightarrow \Psi_\varrho
\longrightarrow \Psi_{\varrho,!,c} \rightarrow 0,$$
and with the filtration of $\Fil^\bullet_*(\Psi_{\varrho,!,c})$ with graded pieces
$\gr_*^{-g_i(\varrho)+1}(\Psi_{\varrho,!,c})$ which can be refined as before, such that to obtain
graded pieces $\grr^k(\Psi_{\varrho,!,c})$ verifying
\begin{multline*}
\lexp p j^{=tg_i(\varrho)}_{c,!*} j^{=tg_i(\varrho),*}_c \grr^k(\Psi_{\varrho,!,c}) \\ \htarrow_+
\grr_k(\Psi_{\varrho,*,c}) \htarrow_+\\  \lexp {p+} j^{=tg_i(\varrho)}_{c,!*} 
j^{=tg_i(\varrho),*}_c \grr^k(\Psi_{\varrho,!,c}),
\end{multline*}
with $\grr^k(\Psi_{\varrho,!,c}) \otimes_{\overline \Zm_l} \overline \Qm_l \simeq
P(t,\pi_v)(\frac{1-t}{2})$. In \cite{boyer-duke}, we prove
\begin{itemize}
\item $\grr^k(\Psi_{\varrho,*,c}) \simeq \lexp p j^{=tg_i(\varrho)}_{c,!*}
 j^{=tg_i(\varrho),*}_c \grr^k(\Psi_{\varrho,!,c})$,
 
 \item and the sheaf cohomology group of $\grr_k(\Psi_{\varrho,*,c})$ are free.
 \end{itemize}
In particular for a supersingular point $z$, the spectral sequence computing
$\hi^{-\delta} i_z^* \Psi_{\varrho} \simeq \hi^{-\delta} i_z^* \Psi_{\varrho,!,c}$ through the 
$\hi^\bullet i_z^* \grr^k(\Psi_{\varrho,*,c})$, degenerates at $E_1$.
Note then that the $P_d(F_v)$-lattice is given by the induced representation
$$\Gamma_G \times
\speh_\delta(\pi_v \{ \frac{s_i(\varrho)-\delta-1}{2} \})$$
where 
\begin{itemize}
\item $\Gamma_G$ is a stable $P_{(s_i(\varrho)-\delta)g_i(\varrho)}(F_v)$-lattice
of $\st_{(s_i(\varrho)-\delta)g_i(\varrho)(F_v)}(\pi_v)$ such that any irreducible subspace
is isomorphic to $\tau_{nd}$;

\item $\speh_\delta(\pi_v)$ has, up to isomorphism, only one stable 
$GL_{\delta g_i(\varrho)}(F_v)$-stable lattice.
\end{itemize}
Like in the previous sections, we then obtain the following description of 
$\UC^{d-1-\delta}_{\varrho,\Nm}$, which is free by the main result of \cite{boyer-duke}.

\begin{propo} \label{prop-principal}
As a $\overline \Zm_l[P_d(F_v) \times D_{v,d}^\times \times W_{F_v}]$-module, 
$\UC^{d-1-\delta}_{\varrho,\Nm}$
has a filtration with successive graded pieces $\grr^k (\UC^{d-1-\delta}_{\varrho,\Nm,free})$ 
where there are an associated $i$, $\pi_v \in \cusp(\varrho,i)$ and $t$ such that
$tg_i(\varrho)=d$ and
$$\gr^{-i} (\UC^{d-1-\delta}_{\varrho,\Nm,free}) \simeq \Gamma_{GDW}(\pi_v),$$ 
with $\Gamma_{GDW}(\pi_v) \simeq \Gamma_G(\pi_v) \otimes 
\Gamma_D \otimes \Gamma_W(\pi_v)$ where
\begin{itemize}
\item $\Gamma_D$ (resp. $\Gamma_W$) is a stable lattice of $\pi_v[s_i(\varrho)]_D$ (resp. 
$\Lm_{g_i(\varrho)}(\pi_v)$);

\item $\Gamma_G$ is isomorphic to a stable $P_d(F_v)$-equivariant lattice of
$LT_{\pi_v}(s_i(\varrho)-\delta-1_i,\delta)$ such that any irreducible $P_d(F_v)$-equivariant
subspace of $\Gamma_G \otimes_{\overline \Zm_l} \overline \Fm_l$
has order of derivative equal to $g_i(\varrho)$.
\end{itemize}
\end{propo}

\rem consider the case where $s=-1$, that is $g_0(\varrho)$ does not divide $d$.
Then we see that the non-degeneracy property which would advocate that irreducible
subspaces of $\UC^{d-1-\delta}_{\varrho,\Nm} \otimes_{\overline \Zm_l} \overline \Fm_l$
should be the less possible degenerate among all the others, is no longer
true for $\delta >0$, even more this is the exact opposite as $g_i(\varrho)$
is the smallest derivative order of all irreducible subquotients of
$\bigl (\Gamma_G \times \speh_\delta(\pi_v \{ \frac{s_i(\varrho)-\delta-1}{2} \}) \bigr )
\otimes_{\overline \Zm_l} \overline \Fm_l$. One way to keep trace of the 
non-degeneracy property might be the following statement which follows trivially from
the isomorphism $(\tau \times \pi)^{(k)} \simeq \tau^{(k)} \times \pi$
for $\tau$ (resp. $\pi$) a representation of $P_d(F_v)$ (resp. $GL_s(F_v)$),
and the short exact sequence (\ref{eq-sec-indP}).

\begin{propo}
Let $\tau$ be an irreducible subspace of the modulo $l$ reduction of
$\Gamma_G \times \speh_\delta(\pi_v \{ \frac{s_i(\varrho)-\delta-1}{2} \})$. 
Then $\tau^{(g_i(\varrho))}$ is non-degenerate.
\end{propo}

To sum up, we have seen that an irreducible subspace of 
$\bigl (\Gamma_G \times \speh_\delta(\pi_v \{ \frac{s_i(\varrho)-\delta-1}{2} \}) \bigr )
\otimes_{\overline \Zm_l} \overline \Fm_l$ is necessarily 
with derivative order $g_i(\varrho)$, but among all of them it is the less degenerated one.

\section{Automorphic congruences}

The class number formula for number fields (resp. the Birch-Swinnerton-Dyer conjecture)
asserts that the order of vanishing of the Dedekind zeta function at $s=0$ of a number field $K$ 
(resp. the order of vanishing at $s=1$ of the $L$-function of some elliptic curve $E$ over a
number field $K$) is given by 
the rank of its group of units (resp. by the rank of the Mordell-Weil
group $E(K)$). Both of these statements can be restated in terms of the 
rank of Selmer groups and is
generalized for $p$-adic motivic Galois representations in the Bloch-Kato conjecture.

Since the work of Ribet, one strategy to realize a part of this conjecture is to consider some
automorphic tempered representation $\Pi$ of a reductive group $G/\Qm$ and take a prime
divisor $l$ of some special values of its $L$-function. We try then to construct an automorphic
non tempered representation $\Pi'$ of $G$ congruent to $\Pi$ modulo $l$ in some sense
so that such an automorphic congruence produces a non trivial element in some Selmer group.

For $G$ a similitude group as in \S \ref{para-KHT},
in \cite{boyer-mrl} we show how to produce automorphic congruences from torsion classes
in the cohomology of $\sh_K$ with coefficients in the local system $V_\xi$.
For example, 
\begin{itemize}
\item see corollary 2.9 of \cite{boyer-mrl}, to each non trivial torsion cohomology class of level 
$I$, we can associate an infinite collection of non isomorphic weakly congruent irreducible 
automorphic representations of the same weight and level but each of them being tempered.

\item In section 3 of \cite{boyer-mrl}, we obtained automorphic congruences between tempered 
and non tempered automorphic representations but with distinct weights.

\item In \cite{boyer-stabilization}, using completed cohomology, we construct automorphic 
congruences between tempered and non tempered automorphic representations of the same 
weight but without any
control of their respective level at $l$ which might be an issue to construct then non trivial
elements in some Selmer groups, cf. loc. cit.
\end{itemize}

Another way to interpret the computations of \cite{boyer-stabilization}, is to say that, whatever is the weight
$\xi$, if you take the level at $l$ large enough, then the cohomology groups of your KHT Shimura variety
with coefficients in $V_\xi$ can not be all free, there must exist some non trivial cohomology classes.
The main aim of this section is then to find explicit conditions for the existence of non trivial
cohomology classes with coefficients in $V_\xi$, with the control of the level at $l$.

\subsection{Cohomology groups over $\overline \Qm_l$}

%
%

\begin{defin} \label{defi-degeneracy} (cf. \cite{M-W})
For $\Pi$ an automorphic irreducible representation $\xi$-cohomological of $G(\Am)$,
then, see for example lemma 3.2 of \cite{boyer-aif},  there exists an integer $s$ called the degeneracy depth of $\Pi$,
such that through the Jacquet-Langlands correspondence and base change, its associated 
representation of $GL_d(\Am_\Qm)$ is isobaric of the following form
$$\mu | \det |^{\frac{1-s}{2}} \boxplus \mu | \det |^{\frac{3-s}{2}} \boxplus \cdots \boxplus \mu | \det |^{\frac{s-1}{2}}$$
where $\mu$ is an irreducible cuspidal representation of $GL_{d/s}(\Am_\Qm)$.
\end{defin}

\rem For a place $v$ such that $G(F_v) \simeq GL_d(F_v)$ in the sense of our previous convention,
the local component $\Pi_v$ of $\Pi$ at $v$ is isomorphic to some $\speh_s(\pi_v)$ where $\pi_v$
is an irreducible non degenerate representation, $s \geq 1$ is an integer and $\speh_s(\pi_v)$
is the Langlands quotient of the parabolic induced representation $\pi_v \{ \frac{1-s}{2} \} \times 
\pi_v \{ \frac{3-s}{2} \} \times \cdots \times \pi_v \{ \frac{s-1}{2} \}$. In terms of the Langlands
correspondence, $\speh_s(\pi_v)$ corresponds to $\sigma \oplus \sigma(1) \oplus \cdots \oplus \sigma(s-1)$ 
where $\sigma$ is the representation of $\gal(\bar F/F)$ associated with $\pi_v$ by the local Langlands 
correspondence.

\begin{nota}
For $\pi_v$ an irreducible cuspidal $\overline \Qm_l$-representation of $GL_g(F_v)$ and $t \geq 1$
such that $tg \leq d$, write
$$H^i_{I^v(\oo),!,1}(\pi_v,t,\xi):=\lim_{\atopp{\longrightarrow}{n}}
H^i_c(\sh_{I^v(n),\bar s_v,\overline{1_{tg}}}^{=tg},V_\xi \otimes
j^{=tg,*}_{\overline{1_{tg}}} P(\pi_v,t)_{\overline{1_{tg}}})$$
and its induced version
\begin{multline*}
H^i_{I^v(\oo),!}(\pi_v,t,\xi):=  \lim_{\atopp{\longrightarrow}{n}}
H^i_c(\sh_{I^v(n),\bar s_v}^{= tg},V_\xi \otimes
j^{=tg,*} P(\pi_v,t))  \\ \simeq H^i_{I^v(\oo),!,1}(\pi_v,t)\times_{P_{tg,d}(F_v)} GL_d(F_v).
\end{multline*}
We also consider
$$H^i_{I^v(\oo),!*,1}(\pi_v,t,\xi):=\lim_{\atopp{\longrightarrow}{n}}
H^i(\sh_{I^v(n),\bar s_v,\overline{1_{tg}}}^{\geq tg},V_\xi \otimes 
P(\pi_v,t)_{\overline{1_{tg}}})$$
and
\begin{multline*}
H^i(\pi_v,t,\xi)_{I^v(\oo),!*}:=\lim_{\atopp{\longrightarrow}{n}}
H^i(\sh_{I^v(n),\bar s_v}^{\geq tg},V_\xi \otimes P(\pi_v,t)) \\ \simeq 
H^i_{I^v(\oo),!*,1}(\pi_v,t)\times_{P_{tg,d}(F_v)} GL_d(F_v).
\end{multline*}
\end{nota}

In this section we only consider the $\overline \Qm_l$-cohomology groups and we recall the 
computations of \cite{boyer-compositio}. 

%
%
%
%
\begin{nota} \label{nota-T} 
Let $\Tm^S_\xi$ be the image of $\Tm^S_{abs}$ inside 
$$\bigoplus_{i =0}^{2d-2}
\lim_{\atop{\rightarrow}{I}} H^i(\sh_{I,\bar \eta},V_{\xi,\overline \Qm_l})$$
where the limit concerned the ideals $I$ which are maximal at each places outside $S$.
\end{nota}

For $\Pi^{\oo,v}$ an irreducible representation of $G(\Am^{\oo,v})$,
consider the set $S$ of finite places $w$ of $\Qm$  such that 
$G$, $I$ and $\Pi^{\oo,v}$ are unramified at $w$,
We then consider $\Pi^{\oo,v}$ as a $\Tm^S_{abs}$-module and 
we denote by $[H^i_{I^v(\oo),!}(\pi_v,t,\xi) ] \{ \Pi^{\oo,v} \}$ the associated 
$\Tm^S_{abs}$-isotypic component of $H^i_{I^v(\oo),!}(\pi_v,t,\xi)$.
We will use similar
notations with the cohomology groups introduced above.
Consider now a fixed irreducible cuspidal representation $\pi_v$ of $GL_g(F_v)$.

\begin{propo} \label{prop-coho-Ql} (cf. \cite{boyer-aif} \S 3.2 and 3.3)
Let $\Pi$ be an irreducible automorphic representation of $G(\Am)$ which is $\xi$-cohomological and
with degeneracy depth $s \geq 1$.
\begin{itemize}
\item If $s=1$ then $[H^i_{I^v(\oo),!}(\pi_v,t,\xi) ] \{ \Pi^{\oo,v} \}$ and
$[H^i_{I^v(\oo),!,*}(\pi_v,t,\xi) ] \{ \Pi^{\oo,v} \}$ are all zero for $i \neq 0$.
For $i=0$, if $[H^i_{I^v(\oo),!}(\pi_v,t,\xi) ] \{ \Pi^{\oo,v} \} \neq (0)$ (resp.
$[H^i_{I^v(\oo),!*}(\pi_v,t,\xi) ] \{ \Pi^{\oo,v} \} \neq (0)$) then 
$$\Pi_v \simeq \st_k(\widetilde \pi_v) \times \Pi'_v,$$ 
where $\Pi'_v$ is any irreducible representation, 
$\widetilde \pi_v$ is inertially equivalent to $\pi_v$ and $k \leq t$ (resp. $k=t$).

\item For $s \geq 1$, and $\Pi_v \simeq \speh_s(\pi_v \times \pi'_v)$ for
$\pi'_v$ any irreducible representation of $GL_{\frac{d-sg}{s}}(F_v)$, then
$[H^i_{I^v(\oo),!}(\pi_v,t,\xi) ] \{ \Pi^{\oo,v} \}$ (resp. $[H^i_{I^v(\oo),!*}(\pi_v,t,\xi) ] \{ \Pi^{\oo,v} \}$) is non zero if and only
if $i=s-1$ and $t \geq s$ (resp. $t=s$ and $i \equiv s-1 \mod 2$ with $|i| \leq s-1$).
\end{itemize}
\end{propo}

\rem In \cite{boyer-aif}, we give the complete description of these cohomology groups.

\subsection{Torsion for Harris-Taylor perverse sheaves}

From now on, we fix an irreducible supercuspidal $\overline \Fm_l$-representation $\varrho$ and all the
irreducible cuspidal $\overline \Qm_l$-representation $\pi_v$ considered will be of type $\varrho$.
In \cite{boyer-torsion}, using the adjunction maps $\Id \longrightarrow j^{=h}_* j^{=h,*}$, we 
construct a filtration of stratification 
$$0=\Fil^{-d}_*(\pi_v,\Pi_t) \subset \Fil^{1-d}_*(\pi_v,\Pi_t) \subset \cdots \subset
\Fil^0_*(\pi_v,\Pi_t)=j^{=tg}_{!} HT(\pi_v, \Pi_t),$$
with free gradutates $\gr^{-r}_*(\pi_v,\Pi_t):=\Fil^{-r}_*(\pi_v,\Pi_t)/\Fil^{-r-1}_*(\pi_v,\Pi_t)$ which are trival
except for $r=kg-1$ with $t \leq k \leq s$ and then verifying
\begin{multline*}
\lexp p  j^{= kg}_{!*} HT (\pi_v,\Pi_t \overrightarrow{\times} \st_{k-t}(\pi_v)) \otimes \Xi^{(t-k)/2} \\
\htarrow \gr^{1-kg}_*(\pi_v,\Pi_t) \htarrow  \\
\lexp {p+}  j^{= kg}_{!*} 
HT (\pi_v,\Pi_t \overrightarrow{\times} \st_{k-t}(\pi_v)) \otimes \Xi^{(t-k)/2},
\end{multline*}
where we recall that $\htarrow_+$ means a bimorphism, i.e. both a mono and a epi-morphism,
whose cokernel has support in $\sh^{\geq kg+1}_{\IC,\bar s}$.

\rem In \cite{boyer-duke}, we in fact proved that each of these graded parts are isomorphic to 
the $p$-intermediate extensions.

\begin{lemma} \label{lem-ext-char}
When $g=1$, i.e. $\pi_v=\chi_v$ is a character, then for all $1 \leq t \leq d$, whatever is
the representation $\Pi_t$ of $GL_t(F_v)$, we have
$$\lexp p j^{=t}_{!*} HT(\chi_v,\Pi_t) \simeq \lexp {p+}  j^{=t}_{!*} HT(\chi_v,\Pi_t).$$
\end{lemma}

\begin{proof}
For $\pi_v$ a character, the associated Harris-Taylor local system on $\sh^{=h}_{\IC,\bar s}$
 is just the trivial one
$\overline \Zm_l$ where the fundamental group $\Pi_1(\sh^{=h}_{\IC,\bar s})$ 
acts by its quotient $\Pi_1(\sh^{=h}_{\IC,\bar s}) \twoheadrightarrow \DC_{v,h}^\times$ with
$\DC_{v,h}^\times$ acting by the character $\chi_v$. Then as $\sh^{\geq h}_{\IC,\bar s_v,\overline{1_h}}$
is smooth over $\spec  \overline \Fm_p$, then this Harris-Taylor local system shifted by the dimension
$d-h$, is perverse for both $t$-structures $p$ and $p+$, in particular the two intermediate 
extensions are equal.
\end{proof}

\rem One of the main result of \cite{boyer-duke} is that this equality of perverse extensions remains true
for every Harris-Taylor local systems associated with any irreducible cuspidal representation $\pi_v$ such
that its modulo $l$ reduction is still supercuspidal, i.e. is of $\varrho$-type
$-1$.

\begin{propo}
For any representation $\Pi_t$ of
$GL_t(F_v)$, we have the following resolution of $\lexp p j_{!*}^{=t} HT(\chi_{v},\Pi_t)$
\addtocounter{thm}{1}
\begin{multline} \label{eq-resolution0}
0 \rightarrow j_!^{=d} HT(\chi_{v},\Pi_t \{ \frac{t-s}{2} \} ) \times 
\speh_{d-t}(\chi_{v}\{ t/2 \} ))
 \otimes \Xi^{\frac{s-t}{2}} \longrightarrow \cdots  \\
\longrightarrow j_!^{=t+1} HT(\chi_{v},\Pi_t \{ -1/2 \}  \times \chi_{v} \{ t/2 \} ) 
\otimes \Xi^{\frac{1}{2}} \longrightarrow \\ j_!^{=t} HT(\chi_{v},\Pi_t) 
\longrightarrow  \lexp p j_{!*}^{=t} HT(\chi_{v},\Pi_t) \rightarrow 0.
\end{multline}
\end{propo}

\begin{proof}
As explained in \cite{boyer-duke}, the statement is equivalent to the freeness of
the sheaf cohomology groups of $\lexp p j_{!*}^{=t} HT(\chi_{v},\Pi_t)$ which is
trivial when $\chi_v$ is a character. Indeed,
as the strata $\sh^{\geq h}_{I^v,\bar s_v,1}$ are smooth, then the constant sheaf, up to shift, 
is perverse and so equals to the intermediate extension of the constant sheaf, shifted by $d-h$, 
on $\sh^{=h}_{I^v,\bar s_v,\overline{1_h}}$. 
In particular we have trivially the following resolution
\begin{multline*}
0 \rightarrow j_!^{=d} HT(1_{v},\Pi_t \{ \frac{t-s}{2} \} ) \otimes 
\speh_{d-t}(1_{v}\{ t/2 \} ))
 \otimes \Xi^{\frac{s-t}{2}} \longrightarrow \cdots  \\
\longrightarrow j_!^{=t+1} HT(1_{v},\Pi_t \{ -1/2 \}  \otimes 1_{v} \{ t/2 \} ) 
\otimes \Xi^{\frac{1}{2}} \longrightarrow \\ j_!^{=t} HT(1_{v},\Pi_t) 
\longrightarrow  \lexp p j_{!*}^{=t} HT(1_{v},\Pi_t) \rightarrow 0,
\end{multline*}
where we recall that $\speh_\delta(1_v)$ is the trivial representation of $GL_\delta(F_v)$.
The resolution (\ref{eq-resolution0}) is then just the induced version of the previous
one twisted by $\chi_v$, as $HT(\chi_v,\Pi_t)$ is the $HT(1_v,\Pi_t)$ where the action
of the fundamental group factors through
$$\pi_1(\sh_{I,\bar s_v}^{=t}) \twoheadrightarrow \DC_{v,t}^\times \longmapright{\chi_v} F_v^\times.$$
\end{proof}

\rem In \cite{boyer-duke}, we prove the previous resolution 
more generally for every irreducible cuspidal representation $\pi_v$ of $GL_g(F_v)$,
\addtocounter{thm}{1}
\begin{multline} \label{eq-resolution00}
0 \rightarrow j_!^{=sg} HT(\pi_v,\Pi_t \overrightarrow{\times} \speh_{s-t}(\pi_v)
\otimes \Xi^{\frac{s-t}{2}} \longrightarrow \cdots \longrightarrow \\
j^{=(t+2)g}_! HT(\pi_v,\Pi_t \overrightarrow{\times} \speh_2(\pi_v)) \otimes \Xi^1
\longrightarrow j_!^{=(t+1)g} HT(\pi_v,\Pi_t \overrightarrow{\times} \pi_v) \otimes
\Xi^{\frac{1}{2}} \\ \longrightarrow
j^{=tg}_! HT(\pi_v,\Pi_t) \longrightarrow \lexp p j^{=tg}_{!*} HT(\pi_v,\Pi_t)
\rightarrow 0,
\end{multline}
which is again equivalent to the property that the sheaf cohomology groups of 
$\lexp p j^{=tg}_{!*} HT(\pi_v,\Pi_t)$ are torsion free. 

%

%
In \cite{boyer-mrl} we prove that torsion classes arising in some cohomology group of the whole Shimura
variety, can be raised in characteristic zero to some automorphic tempered representation of $G(\Am)$ in 
the following sense.

\begin{defin}
A torsion class either in $H^i_{I^v(\oo),!*}(\pi_v,t,\xi)_{\mathfrak m}$ or in 
$H^i_{I^v(\oo),!}(\pi_v,t,\xi)_{\mathfrak m}$,
is said tempered $\xi$-cohomological if there exists
an irreducible automorphic and $\xi$-cohomological tempered representation $\Pi$ unramified outside
$I$ and $p$ with $\Pi^\oo$ a sub-quotient of $\lim_{\rightarrow n} 
H^{d-1}(\sh_{I^v(n),\bar \eta},V_{\xi,\overline \Qm_l})_{\mathfrak m}$.
\end{defin}

From now on we denote by $\varrho$ a $\overline \Fm_l$-character of $F_v^\times$
which could be, if we admit the results of \cite{boyer-duke}, any irreducible
$\overline \Fm_l$-supercuspidal representation of $GL_{g_{-1}(\varrho)}(F_v)$. We will
write the statements and the proofs in the general case. We
moreover suppose that 
$$d=g_{-1}(\varrho)m(\varrho)l^u$$ 
and we will pay attention to
irreducible $GL_d(F_v)$-sub-quotients of either $H^i_{I^v(\oo),!*}(\pi_v,t,\xi)_{\mathfrak m}[l]$ or 
$H^i_{I^v(\oo),!}(\pi_v,t,\xi)_{\mathfrak m}[l]$, isomorphic to $\rho_u$.

\begin{lemma} \label{lem-torsion-rel}
Consider $\pi_{v,i} \in \cusp_{i}(\varrho)$ for $i \geq -1$. 
Suppose there exists a $GL_d(F_v)$-irreducible sub-quotient
of $H^j_{I^v(\oo),!*}(\pi_{v,i},t,\xi)_{\mathfrak m}[l]$ (resp. 
$H^j_{I^v(\oo),!}(\pi_{v,i},t,\xi)_{\mathfrak m}[l]$), 
isomorphic to $\rho_u$, then $j \in \{ 0,1 \}$ (resp. $j=1$).
\end{lemma}

\begin{proof} 
(a) Consider first the case of $i=-1$.
We argue by induction from $t=s=m(\varrho)l^u$ to $t=1$ with both 
$H^j_{I^v(\oo),!,1}(\pi_{v,-1},t,\xi)_{\mathfrak m}$
and $H^j_{I^v(\oo),!*}(\pi_{v,-1},t,\xi)_{\mathfrak m}$. 
Concerning $H^j_{I^v(\oo),!*}(\pi_{v,-1},t,\xi)_{\mathfrak m}$, 
recall that, as $\pi_{v,-1} \in \cusp_{-1}(\varrho)$ so that\footnote{cf. the lemma \ref{lem-ext-char} for a character and \cite{boyer-duke} for the general case.} whatever is the
representation $\Pi_t$ of $GL_{tg_{-1}(\varrho)}(F_v)$,
$$\lexp p j^{=tg}_{!*} HT(\pi_{v,-1},\Pi_t)  \simeq \lexp {p+} j^{=tg}_{!*} HT(\pi_{v,-1},\Pi_t),$$
then we only have to consider the case $j \leq 0$.
By Artin's theorem, see for example theorem 4.1.1 of \cite{BBD}, using the 
affiness of $\sh^{=h}_{I,\bar s_v}$, 
we know that $H^j_{I^v(\oo),!}(\pi_{v,-1},t,\xi)_{\mathfrak m}$ is zero for every $j<0$ and 
is torsion free for $j=0$.

- Note first that for $t=s$, then $HT(\pi_{v,-1},\Pi_s)$ has support in dimension zero, so that 
$H^j_{I^v(\oo),!*}(\pi_{v,-1},s,\xi)_{\mathfrak m}=H^j_{I^v(\oo),!}(\pi_{v,-1},s,\xi)_{\mathfrak m}$
 is zero for $j \neq 0$ and free for $j=0$, 
so the result is trivially true. 

- Suppose by induction, the result is true for all $t'>t$ and consider 
the case of $H^j_{I^v(\oo),!*}(\pi_{v,-1},t,\xi)_{\mathfrak m}$ through the spectral sequence 
associated with the  resolution (\ref{eq-resolution0}). Note first that concerning irreducible 
sub-quotients of the $l$-torsion of the cohomology groups which are $GL_d(F_v)$-isomorphic 
to $\rho_u$, then we can truncate
(\ref{eq-resolution0}) to the short exact sequence of its last three terms.
\addtocounter{thm}{1}
\begin{multline} \label{eq-resolution1}
0 \dashrightarrow j_!^{= (t+1)g} HT(\pi_{v,-1},\Pi_t \overrightarrow{\times} \pi_{v,-1}) \otimes \Xi^{\frac{1}{2}} 
\longrightarrow  \\
j^{=tg}_! HT(\pi_{v,-1},\Pi_t) \longrightarrow \lexp p j^{=tg}_{!*} HT(\pi_{v,-1},\Pi_t)
\rightarrow 0.
\end{multline}
Then considering our problem for $H^j_{I^v(\oo),!*}(\pi_{v,-1},t,\xi)_{\mathfrak m}$, 
for $j \leq -1$, there is no torsion with an irreducible sub-quotient isomorphic to
$\rho_u$.
%
We are then done with $H^j_{I^v(\oo),!*}(\pi_{v,-1},t,\xi)_{\mathfrak m}$.
The result about $H^j_{I^v(\oo),!}(\pi_{v,-1},t,\xi)_{\mathfrak m}$, then follows from the long 
exact sequence associated with (\ref{eq-resolution1}) using the fact that for $j=0$,
it is torsion free.
%

\medskip

(b) Consider now the case $i \geq 0$. Recall, cf. \cite{dat-jl} proposition 2.3.3, that the semi-simplification of
the modulo $l$ reduction of $\pi_{v,i}[t]_D$, does not depend of the choice of a stable lattice, and is equal to
$$\sum_{k=0}^{m(\varrho)l^i -1}\tau\{ - \frac{m(\varrho)l^i-1}{2} +k\}$$ 
where $\tau$ is the modulo $l$ reduction of $\pi_{v,-1}[tm(\varrho)l^i]_D$ which is irreducible, and
$\tau \{ n \}:=\tau \otimes q^{-n \val \circ \nrd}$ where $\nrd$ is the reduced norm.
In particular for any representation $\Pi_t$ of $GL_{tg_{-1}(\varrho)}(F_v)$, we have
\addtocounter{thm}{1}
\begin{multline} \label{eq-F-chgt2}
m(\varrho)l^i \Fm \Bigl [  j_!^{= tm(\varrho)l^ig_{-1}(\varrho)} HT(\pi_{v,-1},\Pi)\Bigr ] =m(\varrho)l^{i}
j_!^{= tm(\varrho)l^ig_{-1}(\varrho)} \Bigl [ \Fm HT(\pi_{v,-1},\Pi) \Bigr ]  \\ = 
j_!^{= tg_i(\varrho)} \Bigl [ \Fm HT(\pi_{v,i},\Pi) \Bigr ] = \Fm \Bigl [ 
j_!^{= tg_i(\varrho)} HT(\pi_{v,i},\Pi) \Bigr ],$$
\end{multline}
where $\Fm(\bullet)=\bullet \otimes^\Lm_{\overline \Zm_l} \overline \Fm_l$.
By the computation of \cite{boyer-compositio} \S 5, we note that for $j > 0$, 
the irreducible sub-quotients of
$H^j(\sh_{I,\bar s_v},j^{=tg}_{!} HT(\pi_{v,-1},\Pi_t) \otimes V_\xi) 
\otimes_{\overline \Zm_l} \overline \Qm_l$ are not tempered except if $t=s-1$
and $j=1$. Then concerning sub-quotients isomorphic to $\rho_u$,
the only case where it can appeared in the modulo $l$ reduction of some irreducible
sub-quotient of the free quotient of $H^j(\sh_{I,\bar s_v},j^{=tg}_{!} HT(\pi_{v,-1},\Pi_t) \otimes V_\xi)$ is when 
either $(t,j)=(s-1,1)$ or $j=0$.
The result about $H^j_{I^v(\oo),!}(\pi_{v,i},t,\xi)_{\mathfrak m}[l]$ 
then follows from the previous case where $i=-1$ using (\ref{eq-F-chgt2}) and the following 
wellknown short exact sequence
$$0 \to H^n(X,\PC) \otimes_{\Zm_l} \Fm_l \longrightarrow H^n(X,\Fm \PC ) 
\longrightarrow H^{n+1}(X,\PC) [l] \to 0,$$
for any $\Fm_q$-scheme $X$ and any $\Zm_l$-perverse free sheaf $\PC$.

Then the result about the cohomology of $\lexp p j^{=tg}_{!*} HT(\pi_{v,i},\Pi_t)$ follows from
the resolution analog of (\ref{eq-resolution1}), and the case of $\lexp {p+} j^{=tg}_{!*} HT(\pi_{v,i},\Pi_t)$
is obtained by Grothendieck-Verdier duality.
\end{proof}

\subsection{Tempered and non tempered congruences}
\label{para-torsion-classes1}

\begin{propo} \label{prop-congruence}
Let $\Pi$ be  an irreducible automorphic cuspidal representation of $G(\Am)$ verifying the 
following properties:
\begin{itemize}
\item it is $\xi$-cohomological with non trivial invariant under some fixed $I \in \IC$;

\item its degeneracy depth is equal to $s>1$;

\item its local component at $v$ is isomorphic to $\speh_s(\pi_v)$ with 
$\pi_v \in \cusp(\varrho,-1)$
and where\footnote{For $\pi_v$ the trivial character, the hypothesis $d=g_u(\varrho)$ for $u=0$ 
is equivalent to ask that the order of $q \in \Fm_l$, which is the cardinal of the residue field of 
$F_v$, is equal to $d$.} $d=g_u(\varrho)$ for some $u \geq 0$.
\end{itemize}
Denote by $\mathfrak m$ the maximal ideal of $\Tm_I$ associated with $\Pi$.
Then for any $w \in \spl$ such that $I_w$ is maximal, and distinct from $l$, there exists an irreducible tempered 
representation $\Pi(w)$ of $G(\Am)$ such that:
\begin{itemize}
\item it is $\xi$-cohomological, 

\item of level $I(w)=I^wI_w$ where $I_w$ is the subgroup of elements of
$GL_d(\OC_w)$ which, modulo the maximal ideal of $\OC_w$, belong to
the parabolic subgroup $P_{1,d}(\kappa(w))$;

\item $\Pi(w)$ is weakly $\mathfrak m$-congruent to $\Pi$ in the sense 
it shares the same multiset of Satake's parameters than $\Pi$ outside $I(w)$. 
\end{itemize}
\end{propo}

\rem In particular for $s=2$, as in Ribet's proof of Herbrand theorem, we should obtain a non 
trivial element in the Selmer group
of the adjoint representation of the Galois $\overline \Fm_l$-representation associated with 
$\mathfrak m$.

Thanks to the main result of \cite{boyer-mrl}, it suffices to prove that 
under the previous hypothesis, the torsion of $H^1(\sh_{I,\bar \eta_v},V_\xi[d-1])_{\mathfrak m}$
is non trivial. Note moreover that $\Pi(w)_w$ looks like $\st_2(\chi_w)\times
\chi_{w,1} \times\cdots \times \chi_{w,d-2}$ for unramified characters 
$\chi_w$, $\chi_{w,1},\cdots,\chi_{w,d-2}$.


\begin{proof}
Thanks to the main result of \cite{boyer-mrl}, it suffices to prove that 
under the previous hypothesis, the torsion of $H^1(\sh_{I,\bar \eta_v},V_\xi[d-1])_{\mathfrak m}$
is non trivial. To do so,
consider the spectral sequence 
$$E_1^{p,q}=H^{p+q}(\sh_{I,\bar s_v},\grr^{-p}_{!,\varrho})_{\mathfrak m} \Rightarrow
H^{p+q}(\sh_{I,\bar \eta_v},V_\xi[d-1])_{\mathfrak m}$$
associated with the filtration $\Fill^{\bullet}_{!,\varrho}$ of $\Psi_\varrho$.
Up to translation we may suppose that $E_1^{p,q}=0$ for all $p<0$.
\begin{itemize}
\item The first idea to construct torsion classes, 
could be to find some non trivial torsion classes in the $E_1$-page, i.e.
in the cohomology of the Harris-Taylor perverse sheaves. For example in \cite{boyer-aif} proposition 4.5.1,
we prove that if the modulo $l$ reduction of such $\pi_v$ is cuspidal but not supercuspidal, 
then, for a well chosen level, the cohomology groups of the associated Harris-Taylor perverse 
sheaves, 
can not be all free, so there is torsion on the $E_1$ page. Unfortunately it seems not so clear
that such torsion cohomology class remains in the $E_\oo$-page.

\item The idea is then to produce torsion in the $E_2$ page by finding a map $d_1^{p,q}$ with
$$\xymatrix{
E_1^{p,q} \otimes_{\overline \Zm_l} \overline \Qm_l 
\ar[rr]^{d_1^{p,q} \otimes_{\overline \Zm_l} \overline \Qm_l} \ar@{->>}[d] & & E_1^{p+1,q}
\otimes_{\overline \Zm_l} \overline \Qm_l \\
Q \ar[rr]^\sim && Q' \ar@{^{(}->}[u]
}$$
such that the $\overline \Zm_l$-lattices of $Q$ and $Q'$ respectively induced by $E_1^{p,q}$
and $E_1^{p+1,q}$, are not isomorphic.
\end{itemize}
%
%
%
First note that over $\Qm_l$:
\begin{itemize}
\item $E_1^{-r,r} \otimes_{\overline \Zm_l} \overline \Qm_l$ has a direct factor isomorphic to 
$(\Pi^{\oo,v})^I \otimes \st_s(\pi_v) \otimes \Lm(\pi_v) (\frac{1-s}{2})$ where we recall that 
the contragredient 
of $\Lm(\pi_v)$ is the Galois representation attached to $\pi_v$ by the local
Langlands correspondance;

\item $d_1^{-r,r} \otimes_{\overline \Zm_l} \overline \Qm_l$ induces a injection from the previous
direct factor into a direct factor of $E_1^{-r+1,r}$ which, as a representation of $GL_d(F_v)$,
is parabolically induced from $P_{(s-1)g_{-1}(\varrho),d}(F_v)$ to $GL_d(F_v)$.
\end{itemize}
From the last remark of the previous section,  
$\ind_{(D_{v,d}^\times)^0 \varpi_v^\Zm}^{D_{v,d}^\times} \lexp p \hi^0 i_z^* E_1^{-r,r}$
as a $GL_d(F_v)$-representation, has a sub-space isomorphic to
 $\Gamma_G(\pi_v)$
where $\Gamma_G(\pi_v)$ is a stable lattice of $\st_s(\pi_v)$
such that $\rho_u$ is the only irreducible sub-representation of 
$\Gamma_G(\pi_v) \otimes_{\overline \Zm_l} \overline \Qm_l$. Moreover we know that
$\rho_u$ can not be a sub-space of a parabolically induced representation.
From these facts we conclude that the torsion of $E_2^{-r+1,r}$ is non trivial and more precisely
that $\rho_u$ is a sub-quotient of $E_2^{-r+1,r}[l]$.

If $\rho_u$ as a sub-quotient of $E_2^{-r+1,r}[l]$ remains a
subquotient of $E_\oo^1[l]$ then we are done. Suppose by absurdity it is not the case. 
First about the free quotient $E_{k,free}^{p,q}$ of the $E_k^{p,q}$, 
we know from\footnote{see also \S 3.2 of \cite{boyer-aif} and more specially proposition 3.2.5} 
\cite{boyer-compositio} that:
\begin{itemize}
\item if $\rho_u$ is a sub-quotient of $E_{1,free}^{p,q} \otimes_{\overline \Zm_l} \overline \Fm_l$
with $p+q \neq 0$, then $\grr_{!,\varrho}^{-p}$ is isomorphic to some $P(\pi_v,s_i(\varrho)-1)$
with $\pi_v \in \cusp(\varrho,i)$ and then $p+q=\pm 1$;

\item for $k \geq 2$ and $p+q \neq 0$, as the $\overline \Qm_l$-spectral sequence degenerates
in $E_2$ and that for $n \neq 0$, 
$E_\oo^n \otimes_{\overline \Zm_l} \overline \Qm_l$ does not have a tempered sub-quotient,
then $\rho_u$ is never a sub-quotient of
$E_{k,free}^{p,q} \otimes_{\overline \Zm_l} \overline \Fm_l$.
\end{itemize}
Then there must exist $(p,q)$ and a torsion class in $(E_{1,tor}^{p,q})_{\mathfrak m}$ with
$p+q=2$ such that $\rho_u$ is a sub-quotient of its $l$-torsion which contradicts lemma \ref{lem-torsion-rel}.
\end{proof}
%

\bibliographystyle{plain}
\bibliography{bib-ok}

\address{Universit\'e Sorbonne Paris Nord\\
LAGA, CNRS, UMR 7539 \\ 
 F-93430, Villetaneuse, France,
CoLoss ANR AAPG2019}

\email{boyer@math.univ-paris13.fr}

\end{document}